\documentclass{article}

\usepackage{amssymb,amsmath,mathrsfs}
\usepackage[colorlinks=false]{hyperref}
\usepackage{comment}
\usepackage{diagbox}
\usepackage{esint}
\usepackage{geometry}
\usepackage{diagbox}
\usepackage{manyfoot}

\usepackage{graphicx,color}
\usepackage[outercaption]{sidecap}
\usepackage{xfrac}
\usepackage{subfig}
\usepackage{tikz}
\usepackage{esint}
\usepackage{enumitem}
\usepackage{pdfsync}
\usepackage{fancyhdr}
\usepackage[nottoc,notlof,notlot]{tocbibind}
\usepackage[titles,subfigure]{tocloft}

\usepackage{amssymb}
\usepackage{amsmath}
\usepackage{latexsym}
\usepackage{amsthm}
\usepackage{eucal}
\usepackage{amsthm}
\theoremstyle{plain}
\newtheorem{theorem}{Theorem}[section]

\newtheorem{lemma}[theorem]{Lemma}
\newtheorem{proposition}[theorem]{Proposition}
\theoremstyle{definition}

\newtheorem{remark}[theorem]{Remark}

\def\e{\varepsilon}
\def\rr{{\mathbb R}}

\def\NN{{\mathbb N}}
\def\ZZ{{\mathbb Z}}

\def\ms{\mu_\sigma}

\DeclareMathOperator*{\R}{\mathbb{R}}

\DeclareMathOperator*{\dist}{dist}

\newcommand{\f}{\varphi}
\newcommand{\ej}{\varepsilon_j}
\newcommand{\rj}{r_j}
\newcommand{\s}{\sigma}
\newcommand{\sj}{\sigma_j}
\newcommand{\dej}{\delta_j}
\newcommand{\A}{\mathcal{A}}
\newcommand{\F}{\mathcal{F}}
\newcommand{\Q}
{\overline{Q}}

\def\XXint#1#2#3{{\setbox0=\hbox{$#1{#2#3}{\int}$} 
  \vcenter{\hbox{$#2#3$}}\kern-.5\wd0}}

\numberwithin{equation}{section}

\begin{document}

\title{Emergence of a convex strange term via homogenization \\ of non-local energies at the critical exponent}

\author{ Giuseppe Cosma Brusca \\
{\small SISSA}\\ \small
via Bonomea 265\\ \small
34146 Trieste, Italy\\ \small gbrusca@sissa.it
\and
Giuliana Fusco\\ 
{\small Scuola Superiore Meridionale}\\ {\small via Mezzocannone 4}\\\small 80134 Napoli, Italy \\ \small  g.fusco@ssmeridionale.it
}
\date{}

\maketitle
{
  \renewcommand{\thefootnote}{}
}

\begin{abstract}
   We derive the $\Gamma$-limit of convolution-type and discrete energies subject to Dirichlet boundary conditions on periodically perforated domains at the critical exponent. We assume that the length-scale of the non-local interactions is much smaller than the side-length of the cubic perforations and prove that a separation of scales occurs. Exploiting the analogies between the variational frameworks of our interest, we employ a unified argument to overcome the technical difficulties that arise from the scaling invariance of the energies. Our multiscale analysis yields a novel observation that is not related to the non-local nature of the functionals, but rather to the analysis at the critical exponent: we prove that the energy density of the \textit{strange term} is convex, even in the vector-valued setting.
\end{abstract}

\section{Introduction}

In a celebrated paper \cite{CioMur}, Cioranescu and Murat investigated the asymptotic behaviour of Dirichlet problems in periodically perforated domains of the form
\begin{equation*}\Omega_\delta:=\Omega\setminus \bigcup_{i\in \ZZ^d} \Q(i\delta,r_\delta),
\end{equation*}
where $\Omega$ is a regular domain of $\rr^d$, $\delta$ is the period of the array of perforations and the cubic perforation $\Q(i\delta, r_\delta)=i\delta+[-\frac{r_\delta}{2},\frac{r_\delta}{2}]^d$ has side-length
\begin{equation*}
    r_\delta= \begin{cases}
   \exp(-\delta^{-2}) & \text{ if } d=2, \\[5pt]
     \delta^\frac{d}{d-2} & \text{ if } d\geq 3;
    \end{cases}
\end{equation*}
see also the earlier paper by Marchenko and Khruslov \cite{MK}. They proved that  this choice of the size of the perforations is critical since, for fixed $\phi\in H^{-1}(\Omega)$, the family of solutions $\{u_\delta\}_\delta$ to the problems
\begin{equation*}
    -\Delta u_\delta = \phi, \qquad u_\delta\in H^1_0(\Omega_\delta),
\end{equation*}
weakly converges to a function $u\in H^1_0(\Omega)$ that solves a different limit problem, namely
\begin{equation*}
     -\Delta u + C_d u = \phi, \qquad u\in H^1_0(\Omega),
\end{equation*}
where an additional zero-order term, the so-called {\em strange term}, originates from the presence of the perforations. The constant $C_d$ that appears in the limit equation is characterized by a capacitary problem and, similarly to the scale $r_\delta$, its nature varies in accordance with the dimension $d$. If $d>2$, the constant is given by the $2$-capacity of the unit cube in $\R^d$. If $d=2$, the same characterization no longer holds since, by scaling invariance, the $2$-capacity of every bounded subset in $\R^2$ is zero and the constant is expressed in terms of the capacities of the unit square relative to squares whose side-lengths diverge. These relative $2$-capacities exhibit a certain logarithmic decay (with respect to the side-length of the larger square) which also occurs for the relative $d$-capacities of a $d$-dimensional cube when $d>2$, as a general consequence of the scaling invariance. Such a logarithmic behaviour is singular as it is in contrast with the polynomial behaviour of the relative $p$-capacity of a $d$-dimensional cube for $p\in(1,d)$ and $d\geq 2$, and for this reason we refer to $p\in(1,d)$ as a subcritical exponent and to $p=d$ as the {\em critical exponent}.

In \cite{AnsBraJMPA}, Ansini and Braides pursued a variational approach for the study of a non-linear vector-valued variant of the original problem for subcritical exponents $p\in(1,d)$. They computed the $\Gamma$-limit of the functionals $F_\delta:L^{p}(\Omega; \mathbb R^m)\to[0,+\infty]$ defined as
    \begin{equation*}
        F_\delta(u):=\begin{cases} \displaystyle
            \int_\Omega f(\nabla u)\, dx  & \text{ if } u\in W^{1,p}_0(\Omega_\delta; \mathbb R^m), \\
            +\infty & \text{ otherwise}, 
        \end{cases}
    \end{equation*}
    for $f: \mathbb R^{m\times d}\to[0,+\infty)$ a quasiconvex $p$-homogeneous integrand satisfying $p$-growth conditions. Conformally with the work by Cioranescu and Murat, they assumed that $r_\delta = \delta^{d/(d-p)}$ and proved that the $\Gamma$-limit of $\{F_\delta\}_\delta$ equals 
\begin{equation*}
    \int_\Omega f(\nabla u)\, dx+\int_\Omega\varphi_p(u)\,dx, \qquad  u\in W^{1,p}_0(\Omega;\mathbb R^m),
\end{equation*}
where $\varphi_p:\R^m\to[0,+\infty)$ is given by a non-linear capacitary formula:
\begin{equation*}
    \varphi_p(z):=\min\Bigl\{\int_{\mathbb R^d} f(\nabla u)\, dx : u-z\in W^{1,p}(\mathbb R^d; \mathbb R^m), u=0 \text{ on } Q(1)\Bigr\}.
\end{equation*}
This result can be regarded as a generalization of the one by Cioranescu and Murat to a non-linear and vectorial setting since, upon assuming that $f$ is convex and differentiable, the property of the convergence of minima of $\Gamma$-convergence implies the convergence of the solutions to the corresponding Euler-Lagrange equations (see \cite{BDF, DM}). 

The analogous problem at the critical exponent $p=d$ carries the technical complications entailed by the scaling invariance and, for this reason, the corresponding extension of the above analysis required a different proof provided by Sigalotti in \cite{Sig-cont}. It is shown that, under the assumption $r_\delta= \exp(-\delta^{d/(1-d)})$, the limit energy is given by 
\begin{equation*}
    \int_\Omega f(\nabla u)\, dx+\int_\Omega\varphi_d(u)\,dx, \qquad  u\in W^{1,d}_0(\Omega;\mathbb R^m),
\end{equation*}
where the density $\varphi_d:\R^m\to[0,+\infty)$ is now described by a {\em homogenization-type formula}:
\begin{equation*}\varphi_d(z):= \lim_{L\to+\infty}(\log L)^{d-1}\min\Bigl\{\int_{Q(L)} f(\nabla u)\, dx : u-z\in W^{1,d}_0(Q(L); \mathbb R^m), u=0 \text{ on } Q(1)\Bigr\}
\end{equation*}
and the factor $(\log L)^{d-1}$ compensates for the logarithmic decay of the relative non-linear capacity at the critical exponent. 

In a similar vein, the aim of the present paper is to perform the variational analysis of Dirichlet problems on periodically perforated domains at the critical exponent in a general non-local framework, providing an extension of the results for subcritical exponents recently obtained by Alicandro, Gelli, and Leone \cite{AGL} and by the second author \cite{Fus}. We propose a unified analysis that encompasses both convolution-type functionals, for which a variational theory has been developed by Alicandro, Ansini, Braides, Piatnitski, and Tribuzio \cite{AABPT}, and discrete functionals within the framework investigated by Alicandro and Cicalese \cite{AliCic}. Functionals considered in these two settings share many similarities, starting from their structure of pairwise-interactions energies. Given an energy density $f$ that satisfies standard growth conditions of order $p>1$, convolution-type energies studied in \cite{AABPT} are of the form 
\begin{equation*}
    \frac{1}{\e^{d+p}}\int_{\Omega\times\Omega}f\Bigl(\frac{y-x}{\e}, u(y)-u(x)\Bigr)\, d\mathcal{L}^d(x)\, d\mathcal{L}^d(y)
\end{equation*}
and can be regarded as a continuous counterpart of the discrete energies considered in \cite{AliCic} given by
\begin{equation*}
\e^{d-p}\sum_{{\substack{\alpha,\beta\in \Omega\cap\e\ZZ^d \\ [\alpha,\beta]\subset \Omega}}}f\Bigl(\frac{\beta-\alpha}{\e}, u(\beta)-u(\alpha)\Bigr),
\end{equation*}
where $[\alpha,\beta]:=\{(1-t)\alpha+t\beta: t\in[0,1]\}$, up to taking into account all possible interactions and not only those involving points connected by a segment that lies in the domain. In both instances, $\e$ is a small positive length-scale (that in the discrete case corresponds to the lattice spacing) responsible for a non-local-to-local passage, provided that the density $f$ satisfies suitable decay assumptions in the first variable. As $\e\to0$, the limit energy is given by the \textit{local} functional
\begin{equation*}
    \int_\Omega f_{\rm hom}(\nabla u)\, dx, \qquad u \in W^{1,p}(\Omega; \rr^m),
\end{equation*}
where $f_{\rm hom}$ is expressed by an asymptotic formula having the same structure in both the continuous and discrete case. As for convolution-type energies, such limit behaviour can be interpreted as a `concentration' phenomenon since relevant pairwise interactions are close to the diagonal of the domain of integration $\Omega\times \Omega$ as $\e\to0$ (see the seminal papers by Bourgain, Brezis, and Mironescu \cite{BBM} and by Ponce \cite{P2, P1}), while the discrete-to-continuum passage is usually referred to as a homogenization process. For the sake of notation, we let $f_{\rm hom}$ denote the `homogenized' energy density in both circumstances. 

Putting aside temporarily the difference between the admissible pairs of interaction (that, in fact, coincide if $\Omega$ is convex), the comparable structure of convolution-type and discrete energies suggests rewriting both functionals as the same double integral that, at the critical exponent $p=d$, reads as
\begin{equation*}
\int_{\Omega\times\Omega}f\Bigl(\frac{y-x}{\e}, u(y)-u(x)\Bigr)\, d\mu_\e(x)\, d\mu_\e(y),
\end{equation*}
upon introducing the reference measures
\begin{equation*}
    \mu_\e:=\begin{cases}
        \displaystyle\frac{1}{\e^d}\mathcal{L}^d & \qquad  \textit{(continuous case),} \\[7pt]
        \displaystyle\sum_{\alpha\in \e\ZZ^d}\delta_\alpha & \qquad \textit{(discrete case),}
    \end{cases}
\end{equation*}
for every $\e>0$. With this approach, admissible configurations are $\mu_\e$-measurable functions; i.e., either Lebesgue-measurable functions (continuous case) or functions defined on $\e\mathbb Z^d$ (discrete case) that we identify with their piecewise-constant interpolations on the cells of the lattice. This identification allows us to regard both continuous and discrete configurations as elements of $L^d(\Omega;\mathbb R^m)$ and to study their asymptotic behaviour through $\Gamma$-convergence with respect to the same strong topology. At this point, boundary conditions can be imposed $\mu_\e$-a.e. on the set of perforations obtaining the constrained non-local energies 
\begin{equation}\label{energieintro}
F_{\e,\delta}(u):=\int_{\Omega\times\Omega}f\Bigl(\frac{y-x}{\e}, u(y)-u(x)\Bigr)\, d\mu_\e(x)\, d\mu_\e(y),\qquad \text{ if } u(x)=0 \text{ for } \mu_\e\text{-a.e. }x \in \Omega\setminus\Omega_\delta.
\end{equation}

In this work we assume that the density $f$ is $d$-homogeneous and satisfies $d$-growth conditions in the second variable and that the vanishing parameters $\delta, r_\delta$ are ruled by $\e$; i.e, 
\begin{equation*}
 \delta=\delta_\e, \qquad r_\e=r_{\delta_\e}\sim\exp(
-\delta_\e^{\frac{d}{1-d}}).
\end{equation*} 
Up to fixing the same kind of reference measures for every $\e$, we prove that if the length-scale $\e$ is small compared to the side-length of a perforation $r_{\e}$, then a separation of scales occurs and the $\Gamma$-limit of $\{F_{\e,\delta_\e}\}_\e$ is the same that is obtained letting first $\e\to0$ and then $\delta\to0$ in \eqref{energieintro}; that is, 
\begin{equation}\label{limitintro}
\int_\Omega f_{\hom}(\nabla u)\,dx
+
\int_\Omega\varphi(u)\,dx, \qquad u\in W^{1,d}(\Omega;\mathbb R^m),
\end{equation}
where $f_{\rm hom}$  is determined by the continuous or discrete setting and, accordingly, the density of the strange term is given by the homogenization-type formula
\begin{equation}\label{phiintro}
\varphi(z)
:=
\lim_{L\to+\infty}
(\log L)^{d-1}
\min\Bigl\{
\int_{Q(L)}f_{\hom}(\nabla u)\,dx:
u-z\in W^{1,d}_0(Q(L);\mathbb R^m),
\ u=0\text{ on }Q(1)
\Bigr\}.
\end{equation}

The problem of perforated domains is a classical topic. Several results have been obtained concerning its many variants and employing different techniques such as the work by Dal Maso and Garroni \cite{DMG} and by Dal Maso and Murat \cite{DMM} for general domains, the extension theorem by Acerbi, Chiadò Piat, Dal Maso, and Percivale \cite{ACPDMP}, and the approach based on the unfolding method pursued by Cioranescu, Damlamian, Donato, Griso, and Zaki \cite{CDDGZ} (see also \cite{CDGb}). Among the results obtained over the last years in various frameworks, we mention \cite{CanCheZar, CherDonRos, KhraPlum} in the case of different boundary conditions, \cite{FocAdvMath} in the fractional setting, and \cite{BraChiDEl, BP2} in the setting of convolution-type energies. Very recently, the problem has been treated in the context of stochastic homogenization by Scardia, Zemas, and Zeppieri \cite{SZZ} (see also \cite{Bas}) and, for fractional energies, by Deangelis, Focardi, and Zeppieri \cite{DeaFocZep} (see also the earlier paper \cite{FocComPDEs}). In the fractional setting, Palatucci has thoroughly discussed the phenomenon of capacitary screening in periodically perforated domains for Gagliardo seminorms defined on the whole $\R^d$ \cite{Pal} and a description of fractional capacitary potentials has been given in \cite{FocPalZep}. Moreover, Braides, Noselli, and Vincini have studied vectorial problems in the context of \mbox{hyperelasticity} \cite{BraNosVin}. These works, as well as many others in the literature, mainly deal with the case of subcritical exponents (that, in a fractional setting, are characterized by the inequality $p<d/s$, where $s$ denotes the fractional order). The necessity of treating a scaling-invariant problem leads us to carry out a delicate study, that substantially departs from the ones performed in the subcritical cases. Our strategy is partly inspired by the works \cite{BraBru, BruNLA} where, in the scalar case, an extension to the critical exponent for the problem of perforated domains in heterogeneous media originally studied in \cite{AnsBraAAM} is obtained.

\smallskip

An outcome of our analysis is that the energy density $\varphi$ in \eqref{phiintro} is convex, a feature that is relevant (and non-trivial) in the vector-valued case $m>1$. Besides the technical aspect, the convexity of the strange term reflects the emergence of a multiscale microstructure since optimal configurations are obtained distributing energy equally on every scale from $r_\e$ (the side-length of the perforations) to the much larger scale $\delta_\e$ (the period of the perforations). This behaviour displays an analogy with the Ginzburg-Landau model (see \cite{BBH, Jer, San}). In our argument, such uniform energy distribution is obtained through a dyadic construction around each perforation similarly to \cite{ABCDP}, where the formation of topological singularities is studied in presence of inhomogeneities. The convexity of the energy density $\varphi$ is not related to the non-local nature of the problem and actually holds for the original Dirichlet problem on perforated domains at the critical exponent. In fact, this property has not been the starting point of our investigation; from our perspective, it is more appropriate to regard it as an outcome of the multiscale analysis as we shall explain in a forthcoming section. We believe that this observation, together with the flexibility of the techniques employed in the present work, may be of use for the study of other problems concerning perforated domains at the critical exponent.     

\smallskip

The approach that we pursue highlights the correspondence between convolution-type and discrete settings. As already mentioned, the main difference lies in the family of pairwise interactions taken into account: on the one hand, in the continuous case every pair of points in $\Omega$ interacts; on the other hand, discrete energies only allow for interactions of points connected by a segment inside the domain. At a technical level, it is reasonable to expect that the asymptotic analysis yields the same outcome for both kinds of interaction, upon assuming that $\Omega$ has Lipschitz boundary. Indeed, short-range interactions are favored in the passage from non-local to local and, therefore, the global geometry of $\Omega$ should not affect the limit energy. In terms of discrete energies, this means that we should be able to consider equivalently all possible pairwise interactions; however, we stick to the above technical constraint in order to resort to the results in \cite{AliCic}. 

To overcome this discrepancy, it is enough to rewrite the starting energies through a change of variables: letting $\xi:=(y-x)/\e$, convolution-type energies equal
\begin{equation*}
    \frac{1}{\e^{d}}\int_{\rr^d}\int_{\{x\in\Omega : x+\e\xi\in\Omega\}}f(\xi, u(x+\e\xi)-u(x))\, d\mathcal{L}^d(x)\, d\mathcal{L}^d(\xi)
\end{equation*}
and analogously, letting $\xi:=(\beta-\alpha)/\e$, discrete energies read as 
\begin{equation*}
 \sum_{\xi\in\ZZ^d} \sum_{\{\alpha\in \Omega\cap\e\ZZ^d: [\alpha,\alpha+\e\xi]\subset \Omega\}}f(\xi, u(\alpha+\e\xi)-u(\alpha)),
\end{equation*}
in such a way that the unconstrained energies can be written as
\begin{equation}\label{unifiedintro}
    \int_{\mathbb R^d}
\int_{\Omega_{\mu_\e}(\xi)}
f(\xi,u(x+\e\xi)-u(x))\, 
d\mu_\e(x)\,d\mu_1(\xi),
\end{equation}
where we set
\begin{equation*}
    \Omega_{\mu_\e}(\xi):=\begin{cases}
        \{x\in \Omega: x+\e\xi \in \Omega\} & \text{ if } \mu_\e=\displaystyle\frac{1}{\e^d}\mathcal{L}^d \qquad  \textit{(continuous case),} \\[7pt]
        \{x\in \Omega: [x,x+\e\xi]\subset \Omega\} & \text{ if } \displaystyle \mu_\e=\sum_{\alpha\in \e\ZZ^d}\delta_\alpha  \qquad \textit{(discrete case).}
    \end{cases}
\end{equation*}
In practice, we investigate the asymptotic behaviour of energies \eqref{unifiedintro} subject to Dirichlet boundary conditions, where the energy density
$f:\mathbb R^d\times\mathbb R^m\to[0,+\infty)$
is assumed to be $d$-homogeneous and locally Lipschitz-continuous in the second variable, coercive on short-range interactions and to satisfy suitable decay assumptions as $|\xi| \to +\infty$ (see (H), (G), and (L) in Section \ref{Sec: Setting of the problem and the main result}). A simultaneous treatment is made possible by the common variational structure of these frameworks that allows us to employ a common argument for the derivation of the $\Gamma$-limit drawing from the general results contained in \cite{AABPT}  and \cite{AliCic}. Nevertheless, it is worth mentioning that some useful tools such as Poincaré and Poincaré-Wirtinger inequalities in discrete form seem to be lacking in the literature. These constitute a substantial part of our technical analysis and they are proved in Appendix \ref{appendix} using some genuinely discrete arguments such as the construction of paths involving nearest-neighbours, uniform estimates on their multiplicity, and the use of suitable piecewise-affine interpolations.

   \smallskip
   
Finally, we note that energies \eqref{energieintro} are ruled by two small parameters $\e,\delta$; hence, their asymptotic behaviour is affected by the interplay between these scales and other significant regimes may be investigated, such as $\e\sim r_\delta$. This regime has been already analyzed at the critical exponent in the discrete case (see \cite{Sig-dis}) and at subcritical exponents for convolution-type energies (see \cite{AGL}). As for the latter, the corresponding extension to the critical exponent is not dealt with in this work; however, we expect that this could be achieved by adapting (and possibly simplifying) some of the arguments illustrated herein and that the limit energy is of the same form as \eqref{limitintro}, with a density $\f$ given now in terms of some non-linear \textit{non-local} capacitary problems.

\smallskip

The paper is organized as follows. In Section \ref{Sec: notation} we introduce the notation and recall the notions of capacity and $\Gamma$-convergence used throughout the paper. In Section \ref{Sec: Setting of the problem and the main result} we present the unified continuous-discrete framework, recall the homogenization result for the unconstrained energies (see Theorem \ref{theorem: Gamma unconstrained}), and state the main result of the paper (see Theorem \ref{thm: main discreto}). In Section \ref{Sec: preliminary results} we recall some useful results such as estimates for long-range interactions and Poincaré-type inequalities. Here, we also prove the convexity of the energy density of the strange term. In Section \ref{Sec: Approximating capacitary-type energy densities} we study the asymptotics of the capacitary-type energy densities and, finally, in Section \ref{Sec: asymptotic analysis} we carry out the asymptotic analysis: first, we prove a joining lemma on varying domains and an auxiliary result that allows one to reconstruct the strange term as the limit contribution of the energy near the perforations; then, we prove the lower and the upper bounds for the $\Gamma$-limit. The appendix contains the proofs of the Poincaré and Poincaré-Wirtinger inequalities, with particular emphasis on the discrete setting.

\section{Notation and basic definitions}
\label{Sec: notation}

In this section we introduce the main notation and recall the definitions of  capacity and $\Gamma$-convergence.

\smallskip

{\bf Notation.} In what follows we let $d,m \in \mathbb{N}$ denote two fixed positive integers corresponding to the dimension of the reference and the target space of functions we consider, respectively. We assume that $d\geq2$ and let $\Omega$ denote a bounded open subset of $\rr^d$ with Lipschitz boundary. We let $\{e_{1},\dots,e_{d}\}$ denote the standard orthonormal basis in $\mathbb{R}^{d}$. Given any positive integer $k$, we let $|x|$ denote the Euclidean norm of $x\in \rr^k$. We denote by $\mathbb{S}^{d-1}$ and $\mathbb{S}^{m-1}$ the unit sphere in $\mathbb{R}^{d}$ and $\rr^m$, respectively. For $x\in\rr^d$ and $r>0,$ we denote the open cube in $\mathbb{R}^{d}$ of center $x$
and side length $r$ as
\begin{equation*}
    Q(x,r):=x+\Bigl(-\frac{r}{2},\frac{r}{2}\Bigr)^{d},
\end{equation*}
and we denote the open ball in $\mathbb{R}^{d}$ of center $x$
and radius $r$ as $B(x,r)$. When the center $x$ coincides with $0$, we simply write $Q(r)$ or $B(r)$. Given points $x,y\in \mathbb R^d$, we set $[x,y]:=\{(1-t)x+ty: t\in[0,1]\}$. If $A$ is a subset of $\mathbb{R}^{d}$, we let $\overline{A}$ denote its closure, and set $A^c:=\rr^d\setminus A$ and $\text{dist}_\infty(x,A):=\inf\{|y-x|_\infty:y \in A\}$ for all $x\in \rr^d$, where we let $|x|_\infty:=\max\{|x_i|:i\in\{1,\dots,d\}\}$. Given $t \in \mathbb{R}$, we let  $\lfloor t\rfloor$ and $\lceil t\rceil$ denote the lower and the upper integer part of $t$, respectively. Unless otherwise stated, $C$ will always denote a generic strictly positive constant that may change from line to line.

\smallskip

{\bf Capacity.} We recall the definition of $p$-capacity for $p>1$. We refer to \cite{HKM} for further details. Given open sets $U\subset V \subseteq \rr^d$ and $p> 1,$ we define the $p$-capacity of $U$ relative to $V$ as
\begin{equation*}
    {\rm Cap}_p(U,V):=\min\Bigl\{\int_V |\nabla v|^p\, dx : v\in W^{1,p}_0(V): v= 1 \text{ on } U\Bigr\}.
\end{equation*}
If $U=B(r)$ and $V=B(R)$ for some $0<r<R$, the explicit computation of the corresponding relative capacity can be performed by the reduction to a $1$-dimensional problem. In particular, when $p=d$, it holds that
\begin{equation}\label{cap balls}
    {\rm Cap}_d(B(r),B(R))= \mathcal{H}^{d-1}(\mathbb{S}^{d-1})\Bigl(\log \Bigl(\frac{R}{r}\Bigr)\Bigr)^{1-d}. 
\end{equation}
This logarithmic behaviour differs from that exhibited in the case $p<d$, where the relative capacity depends on the radii $r$ and $R$ with polynomial growth. For this reason, we address to $p=d$ as the {\it critical exponent} for the capacity. By \eqref{cap balls}, we have that ${\rm Cap}_d(B(1),B(L))=\mathcal{H}^{d-1}(\mathbb{S}^{d-1})(\log L)^{1-d}$ and, through a comparison argument, it is possible to infer that 
\begin{equation*}
    \lim_{L\to+\infty}(\log L)^{d-1} {\rm Cap}_d(Q(1),Q(L))= \mathcal{H}^{d-1}(\mathbb{S}^{d-1}),
\end{equation*}
so that, in particular, there exists a positive constant $C$ such that
\begin{equation}\label{stima uniforme capacita}
    \frac{1}{C}\leq (\log L)^{d-1} {\rm Cap}_d(Q(1),Q(L)) \leq C
\end{equation}
for every $L>4$.

Analogous capacitary problems can be set in a vector-valued framework. Given $z\in\rr^m$ and $U,V$ as above, we may consider the minimum problem
\begin{equation*}
    \min\Bigl\{\int_V |\nabla v|^d\, dx : v\in W^{1,d}_0(V; \rr^m): v= z \text{ on } U\Bigr\}.
\end{equation*}
Due to the scaling property and rotational invariance, it is easily seen that the above minimum coincides with
\begin{multline*}
     \min\Bigl\{\int_V |\nabla v|^d\, dx : v\in W^{1,d}_0(V; \rr^m): v= (1,0,...,0) \text{ on } U\Bigr\}|z|^d \\
     = \min\Bigl\{\int_V |\nabla v|^d\, dx : v\in W^{1,d}_0(V): v= 1\text{ on } U\Bigr\}|z|^d=  {\rm Cap}_d(U,V)|z|^d.
\end{multline*}

\smallskip

{\bf $\Gamma$-convergence.} For our purposes, it suffices to recall the definition of $\Gamma$-convergence with respect to the $L^d$-topology. We refer to \cite{BDF} and \cite{DM}. Given functionals $F_\e:L^d(\Omega;\rr^m)\to[0,+\infty], \ \e>0,$ and $F:L^d(\Omega;\rr^m)\to[0,+\infty]$, we say that $F$ is the $\Gamma$-limit of $\{F_\e\}_\e$ with respect to the strong topology of $L^d(\Omega;\rr^m)$ as $\e\to0$ if, given any sequence $\{\ej\}_j$ converging to $0$, the following hold:
\begin{itemize}
    \item[$(i)$] for every $u\in L^d(\Omega;\rr^m)$ and every $\{u_j\}_j\subset L^d(\Omega;\rr^m)$ such that $u_j\to u$ in $L^d(\Omega;\rr^m)$ as $j\to+\infty$, it holds
    \begin{equation*}
        \liminf_{j\to+\infty} F_{\ej}(u_j)\geq F(u) \qquad \qquad \text{(liminf inequality)},
    \end{equation*}
    \item[$(ii)$] for every $u\in L^d(\Omega;\rr^m)$ there exists $\{v_j\}_j\subset L^d(\Omega;\rr^m)$ such that $v_j\to u$ in $L^d(\Omega;\rr^m)$ as $j\to+\infty$ and
    \begin{equation*}
        \limsup_{j\to+\infty} F_{\ej}(v_j)\leq F(u) \qquad \qquad \text{(limsup inequality)}.
    \end{equation*}
\end{itemize}

\section{Setting of the problem and the main result}
\label{Sec: Setting of the problem and the main result}
In this section we introduce further notation, we describe the setting of our problem, and we state our main result.

To make our notation more flexible for later use, in this first part we let $\s$ denote a positive parameter. We let
\begin{equation*}
    \ms:=\begin{cases}
        \displaystyle\frac{1}{\s^d}\mathcal{L}^d & \qquad  \textit{(continuous case),} \\[7pt]
        \displaystyle\sum_{\alpha\in \s\ZZ^d}\delta_\alpha & \qquad \textit{(discrete case),}
    \end{cases}
\end{equation*}
and, given $\xi \in \mathbb{R}^{d}$ and $A\subseteq \rr^d$ a $\ms$-measurable set, we define 
\begin{equation}
    A_{\ms}(\xi):=\begin{cases}
        \{x\in A: x+\s\xi \in A\} & \text{ if } \ms=\displaystyle\frac{1}{\s^d}\mathcal{L}^d \qquad  \textit{(continuous case),} \\[7pt]
        \{x\in A: [x,x+\s\xi]\subset A\} & \text{ if } \displaystyle \ms=\sum_{\alpha\in \s\ZZ^d}\delta_\alpha  \qquad \textit{(discrete case).}
    \end{cases}
    \label{set}
\end{equation}
 In the continuous case, $A_{\ms}({\xi})=A\cap(A-\s\xi)$. In the discrete case, $A_{\ms}(\xi)$ is only contained in $A\cap(A-\s\xi)$ and these two sets coincide if $A$ is convex. 

 In the discrete case, a $\ms$-measurable function $u$ on $A$ is a function defined on the restriction of the lattice $u:A\cap\s\ZZ^d\to \rr^m$. It is not restrictive to extend such a function to the whole lattice $\s\ZZ^d$ and to identify it with its piecewise constant interpolation on the cells that intersect $A\cap\s\ZZ^d$; that is,
 \begin{equation*}
     u: \mathbb{R}^d\rightarrow \mathbb{R}^{m}, \qquad u(x)=u(\alpha) \text{ if } x\in \alpha +[0,\s)^{d},\alpha \in A\cap \s\ZZ^d.
 \end{equation*}
 The piecewise constant identification of $u$ is a $\mathcal{L}^d$-measurable function; hence, in the discrete setting, we define 
\begin{equation*}
    \mathcal{A}_{\ms}(A;\mathbb{R}^{m}):=\{u: \mathbb{R}^d\rightarrow \mathbb{R}^{m}: u(x)=u(\alpha) \text{ if } x\in \alpha +[0,\s)^{d},\alpha \in A\cap \s\ZZ^d \}\cap L^d(A;\rr^m).
\end{equation*}
In the continuous case, a $\ms$-measurable function on $A$ is simply a $\mathcal{L}^d$-measurable function on $A$. For mere notational consistency, we then let
\begin{equation*}
    \mathcal{A}_{\ms}(A;\mathbb{R}^{m}):=\{u:\rr^d\to\rr^m: u \text{ is }\mathcal{L}^d\text{-measurable}\}\cap L^d(A;\rr^m).
\end{equation*}
With these choices, we regard functions in  $\mathcal{A}_{\ms}(A;\mathbb{R}^{m})$ as functions in $L^d(A;\rr^m)$ for every $\s>0$ in both continuous and discrete case, making then possible the asymptotic analysis via $\Gamma$-convergence with respect to the strong $L^d$-topology. 

\smallskip

We let $f:\mathbb{R}^{d}\times \mathbb{R}^{m}\rightarrow[0,+\infty)$ be a function satisfying the following assumptions:
\begin{itemize}
    \item[(H)]$f(\xi, \cdot)$ is $d$-homogeneous for every $\xi\in\mathbb{R}^d$; i.e., $f(\xi,tz)=t^df(\xi,z)$ for every $t\geq 0, \xi \in \mathbb{R}^{d}$, and $z\in \mathbb{R}^{m}$; 
    \item[(G)]
    the functions $m(\xi):=\displaystyle\inf_{z \in \mathbb{S}^{m-1}} f(\xi,z)$ and $M(\xi):=\displaystyle \sup_{z \in \mathbb{S}^{m-1}} f(\xi,z)$ defined for every $\xi\in \rr^d$ satisfy: 
    \begin{itemize}
        \item[(G0)] there exist $C>0$ and $r_0>1$ such that $m(\xi)\geq  C$ for $\mu_1$-a.e. $\xi$ such that $|\xi|\leq r_0$,
        \item[(G1)] $\displaystyle\int_{\mathbb{R}^d}M(\xi)(1+|\xi|^d)\, d\mu_1(\xi)<+\infty$;
    \end{itemize}    
    \item[(L)] there exists a positive constant $C$ such that 
    \begin{equation*}
        |f(\xi,z)-f(\xi,z')|\leq CM(\xi)(|z|^{d-1}+|z'|^{d-1})|z-z'|
    \end{equation*}
    for every $\xi \in \mathbb{R}^{d}$ and for every $z, z' \in \mathbb{R}^{m}$.
\end{itemize}

\begin{remark}
\label{rmk: bound uniforme} We comment on the above hypotheses:
\begin{itemize}
    \item[(i)] assumption (H) yields $m(\xi)|z|^d\leq f(\xi,z)\leq M(\xi)|z|^d$ for every $(\xi,z)\in \mathbb{R}^d\times \mathbb{R}^m$;
    \item[(ii)] the constant $r_0$ in (G0) is assumed to be strictly larger than $1$ to deal with the discrete case. In such a case, the relevant values of the energy density $f$ are attained at $\ZZ^d\times \rr^m$; therefore, it is needed to require that $r_0>1$ in order to have a meaningful condition on the function $m$ defined in (G). In the continuous case, it would suffice to suppose $r_0>0$; we choose to assume that $r_0>1$ in order to keep our argument consistent in both frameworks;  
    \item[(iii)] in the discrete case, assuming without loss of generality that $M(0)<+\infty$, we have that (G1) is equivalent to
    \begin{equation*}
    \sum_{\xi\in\ZZ^d}M(\xi)|\xi|^d<+\infty,
    \end{equation*}
    which corresponds to the hypothesis (G1) in \cite{Fus}, up to adapting the notation;
    \item[(iv)] hypothesis (L) is satisfied if $f$ fulfills (H), (G), and $f(\xi,\cdot)$ is convex for every $\xi \in \mathbb{R}^{d}$. 
\end{itemize}
    
\end{remark}

We introduce the family of unconstrained functionals and, for later convenience, we also highlight the dependence on the domain. For fixed $\s>0$ and $A\subseteq\rr^d$ a $\ms$-measurable set, we let $\mathcal{F}_\s(\cdot, A):L^d(A;\mathbb{R}^{m})\rightarrow [0,+\infty]$ be defined as
\begin{equation} \label{funzionali unconstrained}
\mathcal{F}_\s(u,A):=\begin{cases}
\displaystyle\int_{\mathbb{R}^{d}}\int_{A_{\ms}(\xi)}f(\xi, u(x+\s\xi)-u(x))\, d\ms(x)\, d\mu_1(\xi) & \text{ if } u\in\A_{\ms}(A;\rr^m),\\[3pt]
+\infty & \text{ otherwise}.
\end{cases}
\end{equation}
For the sake of notation, if $A=\Omega$ we simply write $\mathcal{F}_\s(u)$ in place of $\mathcal{F}_\s(u,\Omega)$. 

The asymptotic analysis of the above energies and the convergence of minimum problems are obtained in \cite[Theorem 6.1, Proposition 6.4]{AABPT} for the continuous case and in \cite[Theorem 4.1, Corollary 4.6]{AliCic} for the discrete case. We reformulate these results in a unified statement suited to our purposes.

\begin{theorem}\label{theorem: Gamma unconstrained}
    Let $A$ be a bounded open set with Lipschitz boundary and let 
    \begin{equation*}
 \ms=
        \frac{1}{\s^d}\mathcal{L}^d \qquad \text{ or } \qquad \ms=
        \sum_{\alpha\in \s\ZZ^d}\delta_\alpha
\end{equation*}
 for every $\s>0$. Let $\{\mathcal{F}_\s(\cdot,A)\}_\s$ be defined by \eqref{funzionali unconstrained} with $f$ satisfying assumptions {\rm (H)} and {\rm (G)}. Then, $\{\mathcal{F}_\s(\cdot, A)\}_\s$ $\Gamma$-converges with respect to the $L^{d}(A;\mathbb{R}^{m})$-topology as $\s\to0$ to the functional $\mathcal{F}(\cdot,A):L^d(A;\mathbb{R}^{m})\rightarrow [0,+\infty]$ defined as 
    \begin{equation*}\mathcal{F}(u,A):=
        \begin{cases}
            \displaystyle
            \int_{A}f_{\hom}(\nabla u)\, dx &\text{if}\ u \in W^{1,d}(A;\mathbb{R}^m),\\
            +\infty &\text {otherwise},
        \end{cases}
    \end{equation*}
    where $f_{\hom}:\mathbb{R}^{m\times d}\rightarrow [0,+\infty)$ is  given by the following homogenization formula
    \begin{equation}
       f_{\hom}(M):=\lim_{h \rightarrow +\infty}\frac{1}{h^{d}}\inf\Big\{\int_{Q(h)}\int_{Q(h)}f(y-x,v(y)-v(x))\, d\mu_1(x)\, d\mu_1(y) : v\in \mathcal{A}^M_{\mu_1}(Q(h))\Big\},
        \label{f_hom}
    \end{equation}
    with 
    \begin{equation}\label{A^M}
        \mathcal{A}^M_{\mu_1}(Q(h)):=\{u \in \mathcal{A}_{\mu_1}(Q(h);\mathbb{R}^{m}): u(x)=Mx \text{ for } \mu_1\text{-a.e. } x \in Q(h),  \text{\rm dist}_\infty(x, Q(h)^c)\leq 1\}
    \end{equation}
    for every $M\in\rr^{m\times d}$. In particular, $f_{\rm hom}$ is a $d$-homogeneous quasiconvex function and there exists a positive constant $C$ such that, for every $M \in \mathbb{R}^{m \times d}$,
    \begin{equation}
    \label{crescita dell' omogenizzata}
        \frac{1}{C}|M|^{d}\le f_{\hom}(M)\le C|M|^{d}.
    \end{equation}
    Moreover, for any $g:\mathbb{R}^d\to\mathbb{R}^m$ Lipschitz continuous function and $T>0$ natural, there holds 
    \begin{equation*}
        \lim_{\s\to 0}\inf\{\mathcal{F}_\s(u,A):u \in \mathcal{A}_{\ms}^{\s T,g}(A)\}=\min\Big\{\int_{A}f_{\hom}(\nabla u(x))\, dx: u-g \in W^{1,d}_0(A;\mathbb{R}^{m})\Big\},
    \end{equation*}
    where 
    \begin{equation}\label{A^sT}
        \mathcal{A}_{\ms}^{\s T,g}(A):=\{u \in \mathcal{A}_{\ms}(A;\rr^m) : u(x)=g(x) \text{ for } \ms \text{-a.e. } x \in A, \text{\rm dist}_\infty(x, A^c)\leq  \s T\}.
    \end{equation}
\end{theorem}

\begin{remark}
If $f(\xi,\cdot)$ is convex for every $\xi \in \mathbb{R}^{d}$, the asymptotic formula \eqref{f_hom} reduces to
 \begin{equation*}
     f_{\hom}(M)=\int_{\mathbb{R}^d}f(\xi,M\xi)\, d\mu_1(\xi);
 \end{equation*}
see \cite[Theorem 6.2]{AABPT}.
\end{remark}

Now we define the functionals of our interest by imposing Dirichlet boundary conditions on an array of perforations. We let $\e$ be a positive reference parameter (that shall tend to $0$) and we introduce two further positive vanishing parameters assuming that both are ruled by $\e$: $\delta=\delta_\e$ denotes the period of the centers of the perforations, and $r=r_\e$ denotes the side-length of the cubic perforations.

For the sake of simplicity we suppose that
\begin{equation*}
\frac{\delta_{\varepsilon}}{\e}\in \mathbb{N}\quad \text{ for every } \e>0.
\end{equation*}
This hypothesis is convenient for the analysis of the discrete case as it implies that the lattice $\delta_\e\ZZ^d$ is coarser than $\e \ZZ^d$, so that the perforations are centered on the support of the atomic measure $\mu_\e$; however, it is not needed for the analysis in the continuous framework. For every $\e>0$, we define the array of perforations as 
\begin{equation*}
    P_{\e}:=\bigcup_{i\in \mathbb{Z}^{d}}\overline{Q}(i\delta_\e,r_{\e})\end{equation*}
and we define the domain of our energy as 
\begin{equation*}
     \mathcal{D}_\e(\Omega;\rr^m):=\{u\in \A_{\mu_\e}(\Omega;\rr^m) : u(x)=0\text{ for } \mu_\e\text{-a.e. } x\in P_\e\}.
\end{equation*}
Then, we define functionals $F_{\varepsilon}:L^d(\Omega;\mathbb{R}^m)\to[0,+\infty]$ as 
\begin{equation}
\label{funzionali}
    F_{\e}(u):=
    \begin{cases}
\displaystyle\int_{\mathbb{R}^{d}}\int_{\Omega_{\mu_\e}(\xi)}f(\xi, u(x+\e\xi)-u(x))\, d\mu_\e(x)\, d\mu_1(\xi) &\text{if}\  u \in  \mathcal{D}_\e(\Omega;\rr^m), \\
        +\infty &\text{otherwise}.
    \end{cases}
\end{equation}
Finally, we state our main result.

\begin{theorem}\label{thm: main discreto} Let 
    \begin{equation*}
 \mu_\e=
        \frac{1}{\e^d}\mathcal{L}^d \qquad \text{ or } \qquad \mu_\e=
        \sum_{\alpha\in \e\ZZ^d}\delta_\alpha
\end{equation*}
 for every $\e>0$ and let $\{F_{\varepsilon}\}_\e$ be defined by \eqref{funzionali} with $f$ satisfying assumptions {\rm (H), (G)}, and {\rm (L)}. Consider two families of positive parameters $\{r_\e\}_\e,\{\delta_\e\}_\e$ such that 
 \begin{equation*}
     \lim_{\e\to0}r_\e= \lim_{\e\to0}\delta_\e=0
 \end{equation*}
and assume that 
\begin{equation*}
    \lim_{\e\to0}\frac{\e}{r_{\e}}=0
\end{equation*}
and that there exists $\gamma>0$ such that
\begin{equation*}
    \lim_{\e\to0} \frac{\exp(-(\gamma\delta_\e^d)^\frac{1}{1-d})}{r_\e}=1.
\end{equation*}
Then, the functionals $\{F_{\varepsilon}\}_\e$ $\Gamma$-converges with respect to the $L^{d}(\Omega;\mathbb{R}^{m})$-topology as $\e\to0$ to the functional $F:L^d(\Omega;\mathbb{R}^{m})\rightarrow [0,+\infty]$ defined as 
\begin{equation}\label{main Gamma}
    F(u):=
        \begin{cases}
            \displaystyle
            \int_{\Omega}f_{\hom}(\nabla u)\, dx+\gamma\int_\Omega\f(u)\, dx &\text{if}\ u \in W^{1,d}(\Omega;\mathbb{R}^m),\\
            +\infty &\text {otherwise},
        \end{cases}
    \end{equation}
    where $f_{\hom}$ is given by \eqref{f_hom} and $\f$ is a convex function that, for every $z\in\rr^m$, is given by
\begin{equation}\label{phi}
    \f(z):=\lim_{L\to+\infty}(\log L)^{d-1}\min\Bigl\{ \int_{Q(L)}f_{\rm hom}(\nabla v)\, dx : v-z\in W^{1,d}_0(Q(L);\mathbb{R}^m), v= 0 \text{ on } Q(1)\Bigr\}.
\end{equation}
\end{theorem}

\begin{remark}
We choose cubic perforations since they are naturally compatible with the cubic lattices used in the discrete setting. This choice is not restrictive since, by a comparison argument, the same conclusion holds true if we replace the reference cube with any bounded set having nonempty interior.    
\end{remark}

{\bf Outline of the proof and convexity of the strange term.} We conclude this section by illustrating in broad terms the idea of our proof and how this led us to observe, somewhat indirectly, the convexity of the strange term. Our argument is based on a new variant of the `joining lemma on varying domains' (see Lemma \ref{lemma: joining}) devised by Ansini and Braides in \cite{AnsBraJMPA}. It allows one to modify a converging family $\{u_\e\}_\e$ (at a small energetic cost) by imposing constant boundary conditions on several dyadic concentric annuli surrounding the perforations. In this way, one can isolate the contributions to the energy arising from the perforations and the remaining region. The latter is essentially the unconstrained energy \eqref{funzionali unconstrained} and it contributes to the limit energy yielding the first-order term in \eqref{main Gamma}. 

To estimate the energy due to a single perforation, that we assume to be centered at $0$ for simplicity, we perform a dyadic construction: we fix a large $L\in \mathbb N$ and cover the reference cell $Q(\delta_\e)$ by means of  concentric cubes
\begin{equation*}
    Q(r_\varepsilon),Q(r_\varepsilon L), Q(r_\varepsilon L^2),\dots, Q(r_\varepsilon L^{S})\sim Q(\delta_\varepsilon), \qquad S\sim \tfrac{\log (\delta_\varepsilon/r_\varepsilon)}{\log L}.
\end{equation*}
Thanks to an argument that allows us to consider only finite-range interactions, we decompose the energy close to the perforation on each ``frame" as follows:
\begin{equation*}
    F_{\varepsilon}(u_\varepsilon, Q(\delta_\varepsilon)) \sim  \sum_{h=1}^S \F_{\varepsilon}(u_\e, Q(r_\varepsilon L^{h})\setminus Q(r_\varepsilon L^{h-1})).
\end{equation*}
By the joining lemma, we can assume that there exist constants $\{u_\e^{0,h}: h\in\{0,\dots,S\}\}$ such that
\begin{equation*}
    u_\varepsilon=u_\varepsilon^{0,h} \text{ close to } \partial Q(r_\varepsilon L^{h}) \qquad \text{ for every } h\in\{0,\dots, S\},
\end{equation*}
with $u_\e^{0,0}=0$ and, taking into account such boundary conditions, we can optimize the energy on $Q(\delta_\e)$ to obtain
\begin{align*}
    F_{\varepsilon}(u_\varepsilon, Q(\delta_\varepsilon)) & \gtrsim  \sum_{h=1}^S \min\{\F_{\varepsilon}(u, Q(r_\varepsilon L^{h})\setminus Q(r_\varepsilon L^{h-1})) : \\& \quad \quad u=  u_\varepsilon^{0,h} \text{ close to }\partial Q(r_\varepsilon L^{h}),  u=u_\varepsilon^{0,h-1} \text{ close to }\partial Q(r_\varepsilon L^{h-1})\}.
\end{align*}
On a single frame $Q(r_\varepsilon L^{h})\setminus Q(r_\varepsilon L^{h-1})$, we use the scaling invariance of the energy and the convergence of minima in Theorem \ref{theorem: Gamma unconstrained} to get that
\begin{align}\notag
          &\min\{\F_{\varepsilon}(u, Q(r_\varepsilon L^{h})\setminus Q(r_\varepsilon L^{h-1})) : u=  u_\varepsilon^{0,h} \text{ close to }\partial Q(r_\varepsilon L^{h}), u= u_\varepsilon^{0,h-1} \text{ close to }\partial Q(r_\varepsilon L^{h-1})\} \\ \notag
          & =\min\{\F_{\varepsilon/(r_\e L^{h-1})}(u, Q(L)\setminus Q(1)) :   u=  u_\varepsilon^{0,h} \text{ close to }\partial Q(L), u= u_\varepsilon^{0,h-1} \text{ close to }\partial Q(1)\} \\ \label{perf: heur}
          & \sim\min\Bigl\{\int_{Q(L)} f_{\rm hom}(\nabla u)\,dx : u-(u_\varepsilon^{0,h}-u_\varepsilon^{0,h-1})\in W^{1,d}_0(Q(L); \mathbb R^m),  u=0 \text{ on } Q(1)\Bigr\}
     \end{align}
as $\e\to0$. Note that here we made use of the assumption $\e/r_\e\to0$ and of a uniform-convergence argument that allows one to keep the boundary conditions unaltered in spite of the fact that these actually depend on $\e$. Recalling \eqref{phi}, we have that \eqref{perf: heur} is approximately
    \begin{equation*}
         (\log L)^{1-d}\varphi(u_\varepsilon^{0,h}-u_\varepsilon^{0,h-1}).
    \end{equation*}

Since $S\sim \frac{\log (\delta_\varepsilon/r_\varepsilon)}{\log L}$ and $r_\varepsilon \sim\exp(-\delta_\varepsilon^{d/(1-d)})$, we have that $(\log L)^{1-d}\sim \delta_\varepsilon^dS^{d-1}$; hence, summing up over the frames, we have
\begin{equation*}
      F_{\varepsilon}(u_\varepsilon, Q(\delta_\varepsilon))  \gtrsim \delta_\varepsilon^dS^{d-1}\sum_{h=1}^S \varphi(u_\varepsilon^{0,h}-u_\varepsilon^{0,h-1}).
      \end{equation*}   
Upon observing that
\begin{equation*}
    \sum_{h=1}^S(u_\varepsilon^{0,h}-u_\varepsilon^{0,h-1})=u_\varepsilon^{0,S}-u_\varepsilon^{0,0}=u_\varepsilon^{0,S},
\end{equation*}
we can further optimize the previous sum and use the $d$-homogeneity of $\varphi$ to get
\begin{equation*}
    F_{\varepsilon}(u_\varepsilon, Q(\delta_\varepsilon)) \gtrsim \delta_\varepsilon^d\min\Bigl\{ S^{d-1}\sum_{h=1}^S \varphi(y_h): \sum_{h=1}^Sy_h=u_\varepsilon^{0,S}\Bigr\} = \delta_\varepsilon^d \min\Bigl\{\frac{1}{S}\sum_{h=1}^S \varphi(y_h): \frac{1}{S}\sum_{h=1}^Sy_h=u_\varepsilon^{0,S}\Bigr\}.
\end{equation*}
The latter term is seen to be approximately $\delta_\varepsilon^d\varphi^{**}(u_\varepsilon^{0,S})$ as $S\to+\infty$, where $\f^{**}$ denotes the convex envelope of $\f$. Then, summing up over all the perforations and using that $u_\e\to u$, we infer that 
\begin{equation*}
   F_{\varepsilon}\Bigl(u_\varepsilon, \bigcup_{i\in \mathbb Z^d \cap\delta_\varepsilon^{-1}\Omega} Q(i\delta_\varepsilon, \delta_\varepsilon)\Bigr)\gtrsim \sum_{i\in \mathbb Z^d \cap\delta_\varepsilon^{-1}\Omega}  \delta_\varepsilon^d\varphi^{**}(u_\varepsilon^{i,S}) \sim \int_{\Omega}\varphi^{**}(u)\, dx
\end{equation*}
as $\e\to0$ and $S\to+\infty$.

The heuristic that yields this lower bound (that is seen to be optimal) seems to be in contrast with the separation of scales argument leading to \eqref{main Gamma} and, to some extent, hints that $\varphi$ is convex. Conscious of this fact, we directly prove that $\f$ coincides with its convex envelope starting from its very definition (see Proposition \ref{prop: convexity}).

\section{Preliminary results}
\label{Sec: preliminary results}

In this section we collect some preliminary results that are instrumental for our analysis.   

\subsection{Auxiliary reference energies}
Given $r_0$ as in (G0), for every $\s>0$, $A \subseteq \mathbb{R}^{d}$ a $\ms$-measurable set, and $u \in \A_{\ms}(A ;\mathbb{R}^{m})$, we define the reference energies
\begin{equation}\label{Funzionali ausiliari G}
G_\s(u,A):=\int_{B(r_0)}\int_{A_{\ms}(\xi)}
|u(x+\s\xi)-u(x)|^d\,d\ms(x) \, d\mu_1(\xi),
\end{equation}
with $A_{\ms}(\xi)$ as in \eqref{set}. In the continuous case, such functionals can be rewritten as
\begin{equation*}
        \int_{B(r_0)}\int_{A\cap (A-\s\xi)}
\Bigl|\frac{u(x+\s\xi)-u(x)}{\s}\Bigr|^d\,dx \, d\xi;
    \end{equation*}
in the discrete case, under the additional (and not restrictive) assumption $r_0<2$, these can be rewritten as 
\begin{equation*}
    \sum_{k=1}^d \sum_{\substack{\alpha\in A\cap\s\ZZ^d \\ [\alpha,\alpha+\s e_k]\subset A}} |u(\alpha+\s e_k)-u(\alpha)|^d.
\end{equation*}
From now on, we assume without loss of generality that $r_0\in(1,2)$, so that the above energies account for nearest neighborhood interactions in the discrete framework. 

Concerning these auxiliary functionals, we first recall a lemma that allows us to control interactions of any range with short-range interactions only. This result is obtained recasting the corresponding results in the continuous case (see \cite[Lemma 4.1]{AABPT}) and in the discrete one (see \cite[Lemma 3.6]{AliCic}). 

\begin{lemma}\label{lemma: controllo long range on frames}
There exists a positive constant $C$ such that for every $\Omega$ open set, $A\subseteq \Omega$ bounded open set with Lipschitz boundary, $T>0$, $\xi \in B(T)$, and $\s>0$ such that 
\begin{equation}\label{hp: long range}
    \text{\rm dist}_\infty(A, \Omega^c)>(T+3\sqrt{d})\s,
\end{equation}
it holds
\begin{equation*}
    \int_{A}|u(x+\s\xi)-u(x)|^d\,d\ms(x)\leq C(1+|\xi|^d)G_\s(u,A+Q((T+3\sqrt{d})\s))
\end{equation*}
for every $u \in \A_{\ms}(\Omega ;\mathbb{R}^{m})$. 
\end{lemma}
\begin{proof}
Assume that $\ms=\s^{-d}\mathcal{L}^d$ and let $A+(0,\s)\xi:=\{x+t\xi: x\in A, t\in(0,\s)\}$. By \eqref{hp: long range} and the fact that $r_0<2$ we have
    \begin{equation*}
        \text{\rm dist}(A+(0,\s)\xi, \Omega^c) > \text{\rm dist}(A+B(\s T), \Omega^c) > \text{\rm dist}_\infty(A+Q(\s T), \Omega^c) >3\sqrt{d}\s>r_0\s,
    \end{equation*}
    so that, applying \cite[Lemma 4.1]{AABPT}, we get
    \begin{equation*}
        \int_{A}\Bigl|\frac{u(x+\s\xi)-u(x)}{\s}\Bigr|^d\,dx\leq C(1+|\xi|^d)G_\s(u,A+(0,\s)\xi+B(r_0\s)),
    \end{equation*}
    and the conclusion follows since $A+(0,\s)\xi+B(r_0\s)\subset A+Q((T+3\sqrt{d})\s)$.
    
    Assume now that $\ms=\sum_{\alpha\in\s\ZZ^d}\delta_\alpha$ and, given $E\subset \Omega$, consider the sets 
\begin{equation*}
    E_{\s}:=\{x\in E: \text{ \rm dist}(x,E^c)>2\sqrt{d}\s\} \qquad \text{ and } \qquad R_\s^\xi(E):=\{\alpha\in E\cap  \s\ZZ^d: [\alpha, \alpha+\s\xi]\subset E\}.
\end{equation*} 
By \cite[Lemma 3.6]{AliCic}, we have 
\begin{equation*}
    \sum_{\alpha\in R_\s^\xi(E_\s)}|u(\alpha+\s\xi)-u(\alpha)|^d \leq C |\xi|^d G_\s(u,E).
\end{equation*}
The conclusion follows letting $E=A+Q((T+3\sqrt{d})\s)$. Indeed, it holds that
\begin{equation*}
    [x,x+\s\xi]\subset A+B((T+\sqrt{d})\s) \qquad \text{ for every } x\in A \text{ and } \xi\in B(T) 
\end{equation*}
and then, by \eqref{hp: long range}, it follows
    \begin{equation*}
        A\cap \s\ZZ^d\subseteq R_\s^\xi(A+B((T+\sqrt{d})\s))\subseteq R_\s^\xi((A+B((T+3\sqrt{d})\s))_\s) \subseteq R_\s^\xi((A+Q((T+3\sqrt{d})\s))_\s), 
    \end{equation*}
    which yields the thesis.
\end{proof}

\begin{remark}
    The above lemma also applies with $\Omega=\rr^d$. In particular, there exists a positive constant $C$ such that for every $\xi\in \rr^d$ and $\s>0$, it holds
\begin{equation}\label{long-range Rd}
    \int_{\rr^d}|u(x+\s\xi)-u(x)|^d\,d\ms(x)\leq C(1+|\xi|^d)G_\s(u,\rr^d)
\end{equation}
for every $u \in \A_{\ms}(\rr^d ;\mathbb{R}^{m})$.
\end{remark}

The following proposition is the rescaled version of Poincar\'e-Wirtinger inequality. We prove it in a slightly more general form in the Appendix (see Proposition \ref{Poincaré-Wirtinger inequality on union of rectangles}). In the continuous setting this can be inferred from \cite[Proposition 4.2]{AABPT}. In the discrete setting, we include a proof that also addresses a possible gap in the argument of \cite[Lemma 2]{BraSig}.

\begin{proposition} \label{prop: poincare-wirtinger} Let $0<\ell<L$, consider the set $A:=Q(L)\setminus \overline{Q}(\ell)$ or $A:=Q(L)$, and let $A'\subseteq A$ be a $\ms$-measurable set and $x_0\in \rr^d$. There exist positive constants $\s_{0}$ and $C$ depending on $A$ such that, having set 
\begin{equation*}
    (u)^{\ms}_E:=\frac{1}{\ms(E)}\int_{E} u\,d\ms
    \end{equation*}
    for every $E$ $\ms$-measurable set with $\ms(E)>0$, it holds 
\begin{equation*}
    \int_{x_0+\tau A} |u(x)-(u)^{\ms}_{x_0+\tau A'}|^d\,d\ms(x)\leq C\frac{\ms(x_0+\tau A)}{\ms(x_0+\tau A')}\Bigl(\frac{\tau}{\s}\Bigr)^d G_\s(u,x_0+\tau A)
\end{equation*}
    for every $\tau>0$, $\s\in(0,\s_0\tau)$, and $u \in \mathcal{A}_{\mu_\sigma}(x_0+\tau A;\mathbb{R}^m)$.  
\end{proposition}

We conclude this subsection with the rescaled version of Poincaré inequality. In the continuous case this is proved in \cite[Proposition 4.1]{AABPT}; as for the discrete case, a proof can be found in the Appendix (see Proposition \ref{Appendix Discrete Poincare}).
\begin{proposition}
\label{prop: pcre}
    Let $T>1, x_0\in \rr^d,$ and let $A\subset \mathbb{R}^d$ be a bounded open set. There exists a positive constant $C$ depending on $A$ such that 
    \begin{equation*}
        \int_{x_0+\tau A}|u(x)|^d\, d\mu_\s(x)\leq C\Bigl(\frac{\tau}{\s}\Bigr)^d G_\s(u,x_0+\tau A)
    \end{equation*}
    for every $\tau,\s>0$ and for every $u \in \mathcal{A}_{\mu_\s}(x_0+\tau A;\mathbb{R}^m)$ such that $u(x)=0$ for $\mu_\sigma$-a.e. $x\in x_0+\tau A$ such that $\dist_\infty(x,(x_0+\tau A)^c)\leq \s T$.
\end{proposition}

\subsection{Functionals with finite-range interactions} Now we introduce ``truncated" versions of the functionals in \eqref{funzionali unconstrained}. For every $T>2$ and for every $z\in\rr^m$ we let
\begin{equation*}
    f^T(\xi,z):=\begin{cases} f(\xi,z) &\text{ if } |\xi|\leq T, \\
    0 &\text{ otherwise},
\end{cases}
\end{equation*}
 and, given $\s>0$ and $A\subseteq\rr^d$ $\ms$-measurable, we let $\mathcal{F}^T_\s(\cdot, A):L^d(A;\mathbb{R}^{m})\rightarrow [0,+\infty]$ be defined as
\begin{equation}\label{funz unco tronc loc}
\mathcal{F}^T_{\s}(u,A):=\begin{cases}
\displaystyle\int_{\mathbb{R}^{d}}\int_{A_{\ms}(\xi)}f^T(\xi, u(x+\s\xi)-u(x))\, d\ms(x)\, d\mu_1(\xi) & \text{ if } u\in\A_{\ms}(A;\rr^m),\\[3pt]
+\infty & \text{ otherwise}.
\end{cases}
\end{equation}
Also in this case, we write $\mathcal{F}^T_\s(u)$ in place of $\mathcal{F}^T_\s(u,\Omega)$ when $A=\Omega$. 
\begin{remark} Since $T>2>r_0$, by (G0) and Remark \ref{rmk: bound uniforme}(i) it follows
    \begin{equation}\label{eq: riferimento lower} 
       CG_\s(u,A)\leq  \F^T_\s(u,A).
    \end{equation}
\end{remark}

As for the constrained functionals \eqref{funzionali}, for $\e>0$ we define $F^T_\e:L^d(\Omega;\mathbb{R}^{m})\rightarrow [0,+\infty]$ as 
\begin{equation} \label{funz tronc loc}
    F^T_{\varepsilon}(u):=
     \begin{cases}
\displaystyle\int_{\mathbb{R}^{d}}\int_{\Omega_{\mu_\e}(\xi)}f^T(\xi, u(x+\e\xi)-u(x))\, d\mu_\e(x)\, d\mu_1(\xi) &\text{if}\  u \in \mathcal{D}_\e(\Omega;\mathbb{R}^m), \\
        +\infty &\text{otherwise}.
    \end{cases}
\end{equation}      

The $\Gamma$-limit of the functionals $\{\F^T_\s\}_\s$ in \eqref{funz unco tronc loc} can be computed making use of Theorem \ref{theorem: Gamma unconstrained}.  According to the following proposition, the corresponding energy densities, denoted by $f_{\rm hom}^T$, converge as $T\to+\infty$ to $f_{\rm hom}$, the energy density of the $\Gamma$-limit of the functionals $\{\F_\s\}_\s$ defined by \eqref{funzionali unconstrained}.

\begin{proposition}
\label{prop_conv_fhom_uniforme}
     Assume $T>2$ and let
  \begin{equation}
       f^T_{\hom}(M):=\lim_{h \rightarrow +\infty}\frac{1}{h^{d}}\inf\Big\{\int_{Q(h)}\int_{Q(h)}f^T(y-x,v(y)-v(x))\, d\mu_1(x)\, d\mu_1(y) : v\in \mathcal{A}^M_{\mu_1}(Q(h))\Big\}
        \label{f^T_hom}
    \end{equation}
for every $M\in \rr^{m \times d}$, where $\mathcal{A}^M_{\mu_1}(Q(h))$ is as in \eqref{A^M}. Given $f_{\rm hom}$ as in \eqref{f_hom}, it holds that
\begin{equation}
\label{convergenza_fhom_troncate}
   \lim_{T\to+\infty}f^T_{\rm hom}=f_{\rm hom}
\end{equation}
monotonically and uniformly on compact subsets of $\mathbb{R}^{m\times d}$. 
\end{proposition}

\begin{proof}
    The proof of the pointwise convergence in \eqref{convergenza_fhom_troncate} is contained in the proof of \cite[Theorem 6.1]{AABPT} for the continuous case and in the proof of \cite[Proposition 4.2]{AliCic} for the discrete case (in particular, see equation $(4.9)$ therein). As for the uniform convergence on compact subsets of $\mathbb{R}^{m\times d}$, we note that $f_{\rm hom}^T\leq f_{\rm hom}$ for every $T>2$. By \eqref{crescita dell' omogenizzata}, this implies that there exists a positive constant $C$ such that
\begin{equation}\label{in rmk f_hom^T}
    \frac{1}{C}|M|^d\leq f_{\rm hom}^T(M)\leq f_{\rm hom}(M)\leq C|M|^d
\end{equation}
for every $T>2$ and $M\in \rr^{m\times d}$. Since $f_{\rm hom}^T$ is a $d$-homogeneous quasiconvex function, as a consequence of \cite[Remark 4.13]{BDF} and \eqref{in rmk f_hom^T}, there exists a positive constant $C$ such that
\begin{equation}\label{fhom-loc-lip}
    |f_{\rm hom}^T(M)-f_{\rm hom}^T(M')|\leq C(|M|^{d-1}+|M'|^{d-1})|M-M'|
\end{equation}
for every $T>2$ and $M,M'\in \rr^{m\times d}$.
The claim follows by the pointwise convergence, \eqref{fhom-loc-lip}, and Ascoli-Arzelà Theorem. 
\end{proof}

The main advantage in dealing with functionals that account for finite-range interactions only is that, in broad terms, they behave quite similarly to local functionals. For the constrained energies $\{F_\e\}_\e$ defined in \eqref{funzionali} as well, it is possible to perform the asymptotic analysis considering first functionals with finite-range interactions \eqref{funz tronc loc}, and then recovering the general case letting $T\to+\infty$, as stated in the proposition below. The proof is provided in \cite[Section 5, Step 1]{Fus} for the discrete case and, for the continuous case, it is obtained by arguing as in the proof of \cite[Lemma 5.1]{AABPT}, where the corresponding result for the unconstrained energies $\{\F_\s\}_\s$ is established. 

\begin{proposition}\label{proposition: gamma_troncamento} Let $\{T_h\}_h$ be an increasing sequence such that $T_h\to+\infty$ as $h\to+\infty$ and let $\{F^{T_h}_\e\}_\e$ be defined as in \eqref{funz tronc loc}. For every $h\in\NN$, assume that there exists a functional $F^{T_h}:L^d(\Omega;\mathbb{R}^{m})\rightarrow [0,+\infty]$ such that
\begin{equation*}
     \Gamma\text{-}\lim_{\e\to0} F^{T_h}_{\e}= F^{T_h}.
\end{equation*}
 Then, it holds
\begin{equation*}
    \Gamma\text{-}\lim_{\e\to0} F_{\e}=\lim_{h\to+\infty} F^{T_h}.
\end{equation*}
\end{proposition}

The next result accounts for the pre-compactness in the strong $L^d$-topology of sequences of functions with uniformly bounded energies. We provide a more general statement for non-truncated energies \eqref{funzionali unconstrained}. The proof of this fact follows combining \cite[Theorem 4.2]{AABPT} with \eqref{eq: riferimento lower} in the continuous case and can be inferred from the proof of \cite[Corollary 3.11]{AliCic} in the discrete case.  

\begin{proposition}\label{prop: compactness} Let $\{\ej\}_j$ be a positive sequence such that $\ej\to0$ as $j\to+\infty$, let $\{\F_{\ej}\}_j$ be as in \eqref{funzionali unconstrained} with $A=\Omega$, and let $u_j\in\A_{\mu_{\ej}}(\Omega;\rr^m), j\in\NN,$ be a sequence such that
\begin{equation*}
    \sup_j\,(\F_{\ej}(u_j)+\|u_j\|_{L^d(\Omega; \rr^m)})<+\infty.
\end{equation*}
Then, there exist a (not relabeled) subsequence $u_j\in\A_{\mu_{\ej}}(\Omega;\rr^m), j\in\NN,$ and a function $u\in W^{1,d}(\Omega;\rr^m)$ such that $u_j\to u$ in $L^d(\Omega; \rr^m)$ as $j\to+\infty$. 
\end{proposition}

\subsection{Homogenization-type capacitary formula}

We illustrate some relevant properties of the density $\f$ appearing in the {\it strange term} in our main result Theorem \ref{thm: main discreto}. The following proposition, which coincides with \cite[Proposition 5.1]{Sig-cont} up to replacing spherical perforations with cubic ones, determines the asymptotics of non-linear capacities modelled on $d$-homogeneous densities.

\begin{proposition}\label{prop: limite log} Let $g: \rr^{m\times d}\to [0,+\infty)$ be a continuous $d$-homogeneous function and assume that there exists a positive constant $C$ such that,  for every $M\in\rr^{m\times d}$, it holds
\begin{equation*}
    \frac{1}{C}|M|^d\leq g(M)\leq C|M|^d.
\end{equation*}
Then, there exist
\begin{equation*}
    \psi(z):=\lim_{L\to+\infty}(\log L)^{d-1}\min\Bigl\{\int_{Q(L)}g(\nabla v)\, dx : v-z\in W^{1,d}_0(Q(L); \rr^m): v= 0 \text{ on } Q(1)\Bigr\},
\end{equation*}
and a positive constant $C$ such that,  for every $z\in\rr^{m}$, it holds
\begin{equation*}
    \frac{1}{C}|z|^d\leq \psi(z)\leq C|z|^d.
\end{equation*}
\end{proposition}

Our main observation is that the function $\psi$ in Proposition \ref{prop: limite log} is convex. To see this, we first prove the following lemma.

\begin{lemma}\label{lemma: psi**} Let $\psi:\rr^m\to[0,+\infty)$ be a coercive continuous function and let $\psi^{**}$ denote its convex envelope. For every $z\in\rr^m$ it holds
\begin{equation}\label{eq: psi**}
     \psi^{**}(z)=\lim_{S\to+\infty}\min\Bigl\{\frac{1}{S}\sum_{h=1}^S\psi(z_h): \frac{1}{S}\sum_{h=1}^Sz_h=z\Bigr\}.
 \end{equation}
\end{lemma}
\begin{proof}
Let $\widetilde{\psi}(z)$ denote the right-hand side of \eqref{eq: psi**}. We only prove that $\widetilde{\psi}(z)\leq \psi^{**}(z)$ as the converse  inequality is immediately satisfied.

By Carathéodory's theorem (see \cite[Theorem 17.1]{R}), there exist $\{\lambda_1,\dots,\lambda_{m+1}\}\subset[0,1]$ and $\{z_1,\dots,z_{m+1}\}\subset\R^m$ such that
\begin{equation*}
    \sum_{i=1}^{m+1}\lambda_i\psi(z_i)=\psi^{**}(z),  \qquad \sum_{i=1}^{m+1}\lambda_iz_i=z, \qquad \text{ and } \qquad\sum_{i=1}^{m+1}\lambda_i=1.
\end{equation*}
Consider $S$ a positive integer and let
\begin{equation*}
(y_1,\dots,y_{S}):=\Bigl(\,\underbrace{\frac{\lambda_1S}{\lfloor\lambda_1S\rfloor}z_1,\dots,\frac{\lambda_1S}{\lfloor\lambda_1S\rfloor}z_1}_{\lfloor\lambda_1S\rfloor\text{ times} },\dots,\underbrace{\frac{\lambda_{m+1}S}{\lfloor\lambda_{m+1}S\rfloor}z_{m+1},\dots,\frac{\lambda_{m+1}S}{\lfloor\lambda_{m+1}S\rfloor}z_{m+1}}_{\lfloor\lambda_{m+1}S\rfloor\text{ times} }, \underbrace{0,\dots,0}_{S-\sum_{i=1}^{m+1}\lfloor\lambda_i S\rfloor \text{ times}}\hspace{-0.7cm}\Bigr).
\end{equation*}
We have that 
\begin{equation*}
    \frac{1}{S}\sum_{h=1}^Sy_h=  \frac{1}{S}\sum_{i=1}^{m+1}\lfloor\lambda_i S\rfloor\frac{\lambda_iS}{\lfloor\lambda_iS\rfloor}z_i= z
\end{equation*}
and 
\begin{equation*}
     \frac{1}{S}\sum_{h=1}^S\psi(y_h) =\frac{1}{S}\sum_{i=1}^{m+1}\lfloor\lambda_i S\rfloor\psi\Bigl(\frac{\lambda_iS}{\lfloor\lambda_iS\rfloor}z_i\Bigr)+\Bigl(1-\frac{1}{S}\sum_{i=1}^{m+1}\lfloor\lambda_iS\rfloor\Bigr)\psi(0);   
\end{equation*}
therefore, passing to the limit as $S\to+\infty$, we obtain
\begin{equation*}
    \widetilde{\psi}(z)\leq \sum_{i=1}^{m+1}\lambda_i\psi(z_i)= \psi^{**}(z),
\end{equation*}
which is the thesis.
\end{proof}

\begin{proposition}\label{prop: convexity} Let $g$ and $\psi$ be as in Proposition \ref{prop: limite log}. Then, $\psi$ is convex.    
\end{proposition}
\begin{proof}
We prove that $\psi \leq \psi^{**}$, the other inequality being trivially satisfied. Let $z\in\rr^m$. By the $d$-homogeneity of $g$, we have that $\psi$ is $d$-homogeneous as well; hence, in light of Lemma \ref{lemma: psi**}, we have
 \begin{equation}\label{psi**}
     \psi^{**}(z)=\lim_{S\to+\infty}\min\Bigl\{\frac{1}{S}\sum_{h=1}^S\psi(z_h): \frac{1}{S}\sum_{h=1}^Sz_h=z\Bigr\}=\lim_{S\to+\infty}S^{d-1}\min\Bigl\{\sum_{h=1}^S\psi(z_h):\sum_{h=1}^Sz_h=z\Bigr\}.
 \end{equation}
Let $L,S$ be fixed positive integers. Consider $(z_1^*,\dots,z_S^*)$ such that $\sum_{h=1}^Sz^*_h=z$ and
\begin{equation}\label{scelta zh}
    \sum_{h=1}^S\psi(z^*_h)= \min\Bigl\{\sum_{h=1}^S\psi(z_h):\sum_{h=1}^Sz_h=z\Bigr\},
\end{equation}
and, for every $h\in\{1,\dots,S\}$, consider a function $v_h$ such that $v_h-z_h^*\in W^{1,d}_0(Q(L);\rr^m), v_h=0$ on $Q(1)$ and 
\begin{equation}\label{scelta vh}
    \min\Bigl\{\int_{Q(L)}g(\nabla v)\, dx : v-z_h^*\in W^{1,d}_0(Q(L); \rr^m): v=0 \text{ on } Q(1)\Bigr\}= \int_{Q(L)}g(\nabla v_h)\,dx.
\end{equation}
 Now we define the function $w:Q(L^S)\to\rr^m$ as
\begin{equation*}
    w(x):=\begin{cases}
    v_1(x) & \text{ if } x\in \Q(L), \\
        v_h\Bigl(\displaystyle\frac{x}{L^{h-1}}\Bigr)+\displaystyle\sum_{l=1}^{h-1}z_l^* & \text{ if } x\in Q(L^{h})\setminus \Q(L^{h-1}), h\in\{2,\dots,S\},
    \end{cases}
\end{equation*}
and we note that $w-z\in W^{1,d}_0(Q(L^S);\rr^m)$ and $w=0$ on $Q(1)$. Using the $d$-homogeneity of $g$ and \eqref{scelta vh}, we have
\begin{align*}
    \int_{Q(L^S)}g(\nabla w)\,dx & = \sum_{h=1}^S \int_{Q(L^{h})\setminus\Q(L^{h-1})} \frac{1}{(L^{h-1})^d}g\Bigl(\nabla v_h\Bigl(\frac{x}{L^{h-1}}\Bigr)\Bigr)\,dx \\
    & =  \sum_{h=1}^S \int_{Q(L)}g(\nabla v_h)\,dx \\
    & = \sum_{h=1}^S     \min\Bigl\{\int_{Q(L)}g(\nabla v)\, dx : v-z_h^*\in W^{1,d}_0(Q(L); \rr^m): v=0 \text{ on } Q(1)\Bigr\}.
\end{align*}
Therefore, multiplying by $(\log L^S)^{d-1}$ we obtain
\begin{multline*}
( \log L^S)^{d-1}\int_{Q(L^S)}g(\nabla w)\,dx \\ =    S^{d-1}\sum_{h=1}^S (\log L)^{d-1}    \min\Bigl\{\int_{Q(L)}g(\nabla v)\, dx : v-z_h^*\in W^{1,d}_0(Q(L); \rr^m): v=0 \text{ on } Q(1)\Bigr\},
\end{multline*}
and passing to the limit as $L\to+\infty$, we get
\begin{equation*}
    \psi(z)\leq S^{d-1}\sum_{h=1}^S\psi(z_h^*) =  S^{d-1}\min\Bigl\{\sum_{h=1}^S\psi(z_h):\sum_{h=1}^Sz_h=z\Bigr\}
\end{equation*}
in virtue of \eqref{scelta zh}. Recalling \eqref{psi**} and letting $S\to+\infty$ we have $\psi(z)\leq \psi^{**}(z)$, which concludes the proof. 
\end{proof}

\section{Approximating capacitary-type energy densities}
\label{Sec: Approximating capacitary-type energy densities}

In this section we introduce some auxiliary functions that, asymptotically, shall constitute our energy density of capacitary-type. 

We let $T, \ell, L,\s$ be positive parameters satisfying
\begin{equation*}
T\in\NN,\,\, T>2 ,\qquad \ell\leq 1, \qquad  L>4, \qquad  \s <1/T,  
\end{equation*} 
and recalling the definitions of $\F_\s^T$, $f^T_{\rm hom}$, and $f_{\rm hom}$ in \eqref{funz unco tronc loc}, \eqref{f^T_hom}, and \eqref{f_hom}, respectively, we let
\begin{align}
\label{def_minimi_approx}
    \f^T_{\ell,L,\sigma}(z) & :=
     \inf\{ \F^T_{\sigma}(v, Q(L)\setminus \overline{Q}(\ell)) : v\in \A_{\mu_\sigma}(Q(L); \rr^m), \\ \notag
     & \qquad \qquad v(x)= 0 \text{ for } \ms\text{-a.e. } x\in\overline{Q}(\ell+\sigma T),
     v(x)= z \text{ for } \ms\text{-a.e. } x\in Q(L-\sigma T)^c\}, \\ \notag
    \f^T_{L}(z) & := \min\Bigl\{ \int_{Q(L)}f^T_{\rm hom}(\nabla v)\, dx : v-z\in W^{1,d}_0(Q(L);\mathbb{R}^m), v= 0 \text{ on } Q(1)\Bigr\}, \\    
\label{limite_in_L}
    \f^T(z) & := \lim_{L\to+\infty} (\log L)^{d-1}\f^T_{L}(z), \\
\label{limit_denisty}
    \f(z) & := \lim_{L\to+\infty}(\log L)^{d-1}\min\Bigl\{ \int_{Q(L)}f_{\rm hom}(\nabla v)\, dx : v-z\in W^{1,d}_0(Q(L);\mathbb{R}^m), v= 0 \text{ on } Q(1)\Bigr\},
    \end{align}
for every $z\in\rr^m$. 

\begin{remark}\label{rmk: funzioni ausiliarie}
The following properties hold true:
\begin{itemize}
\item[(i)] recalling that the energy $\F^T_\s$ accounts for pairwise interactions with range at most $\s T$ only, for any function $v$ admissible for $\f^T_{\ell,L,\s}$ we have
\begin{equation*}
    \F^T_{\sigma}(v, Q(L)\setminus \overline{Q}(\ell))=\F^T_{\sigma}(v, Q(L))=\F^T_{\sigma}(v-z, \rr^d);
\end{equation*}
    \item[(ii)] functions $\f^T_L$ and $\f$ are well-defined as minimum problems since $f^T_{\rm hom}$ and $f_{\rm hom}$ are quasiconvex functions and coerciveness is ensured by the boundary conditions and Poincaré inequality; 
    \item[(iii)] taking into account \eqref{in rmk f_hom^T} and the fact that $f^T_{\rm hom}$ and $f_{\rm hom}$ are $d$-homogeneous, we have that $\f^T$ and $\f$ are well defined in light of  Proposition \ref{prop: limite log} and are convex by Proposition \ref{prop: convexity};
     \item[(iv)] by (H) and the $d$-homogeneity of $f_{\rm hom}^T$ and $f_{\rm hom}$, we have that the functions $\f_{\ell,L,\s}^T, \f^T_L, \f^T,$ and $\f$ are $d$-homogeneous;
     \item[(v)] it holds that
     \begin{equation*}
         \f^T_{\ell,L,\s}=\f^T_{c\ell,cL,c\s}
     \end{equation*}
     for every $c>0$. Note indeed that, given any function $v\in \A_{{\mu_\sigma}}(Q(L); \rr^m)$ satisfying $v=0$ $\ms$-a.e. on $\overline{Q}(\ell+\sigma T)$ and $v=z$ $\ms$-a.e. on $Q(L-\sigma T)^c$, we have that
     \begin{equation*}
         w(x):=v\Bigl(\frac{x}{c}\Bigr), \qquad x\in \rr^d,
     \end{equation*}
     is such that $w\in \A_{\ms}(Q(cL);\rr^m)$ with $w=0$ $\ms$-a.e. on $\overline{Q}(c\ell+c\sigma T)$ and $w=z$ $\ms$-a.e. on $Q(cL-c\sigma T)^c$ and, by (H),
     \begin{align*} \notag
        \F^T_{c\s}&(w, Q(cL)\setminus \overline{Q}(c\ell)) \\ \notag
        &=\int_{\mathbb{R}^{d}}\int_{(Q(cL)\setminus \overline{Q}(c\ell))_{\mu_{c\s}}(\xi)}f^T(\xi, w(x+c\s\xi)-w(x))\, d\mu_{c\s}(x)\, d\mu_1(\xi)\\ \notag
        &=\int_{\mathbb{R}^{d}}\int_{(Q(L)\setminus \overline{Q}(\ell))_{\mu_{\s}}(\xi)}f^T(\xi, v(x+\s\xi)-v(x))\, d\mu_{\s}(x)\, d\mu_1(\xi)=\F^T_{\s}(v, Q(L)\setminus \overline{Q}(\ell)) ;
\end{align*}
      \item[(vi)] for $\ell\leq\ell'$ and $L\leq L'$ we have
    \begin{equation*}
        \f^T_{\ell,L,\s}\leq\f^T_{\ell',L,\s}, \qquad \f^T_{\ell,L',\s}\leq\f^T_{\ell,L,\s}, \qquad \text{ and } \qquad \f^T_{L'}\leq\f^T_{L}.
    \end{equation*}
\end{itemize}
\end{remark}
We establish uniform bounds for the functions in \eqref{def_minimi_approx}.

\begin{lemma} Let $\f^T_{\ell,L,\sigma}$ be as in \eqref{def_minimi_approx}, then there exists a positive constant $C$ such that
\begin{equation}
\label{crescita_minimi_discreti}
    (\log L)^{d-1}\f_{\ell,L,\s}^T(z)\leq C|z|^d
\end{equation}
and
\begin{equation}
\label{equilip_minimi_discreti}
    |(\log L)^{d-1}\f_{\ell,L,\s}^T(z)-(\log L)^{d-1}\f_{\ell,L,\s}^T(z')|\leq C(|z|^{d-1}+|z'|^{d-1})|z-z'|
\end{equation}
for every $T\in\NN,T>2, \ell\leq1,L>4$ and $\s < 1/T$, and for every $z,z'\in \rr^m$. 
\end{lemma}

\begin{proof}
Let $z \in \mathbb{R}^m$ and $u \in C^1(\mathbb{R}^d;\mathbb{R}^m)$ such that $u-z \in C^1_c(Q(L-\sigma( T+1));\mathbb{R}^m)$ and $u=0$ in $\overline{Q}(\ell+\sigma (T+1))$. We show that there exist a function $\hat{u}$ admissible for the minimum problem \eqref{def_minimi_approx} and a constant $C$ depending on $r_0$ and $d$ such that   \begin{equation}\label{claim sui bound}
        G_{\sigma}(\hat{u}-z,\mathbb{R}^d)\leq C\int_{Q(L-\s(T+1))}|\nabla u|^d\, dx,
    \end{equation}
    where $G_{\sigma}(\hat{u}-z,\mathbb{R}^d)$ is defined as in \eqref{Funzionali ausiliari G} with $A=\mathbb{R}^d$. To this end, we treat separately the continuous and the discrete case.
    
    If $\ms=\s^{-d}\mathcal{L}^d$, we let $\hat{u}=u$. Recalling that in this case $G_{\sigma}(u-z,\mathbb{R}^d)$ corresponds to
    \begin{equation*}
        \int_{B(r_0)}\int_{\rr^d}
\Bigl|\frac{u(x+\s\xi)-u(x)}{\s}\Bigr|^d\,dx \, d\xi,
    \end{equation*}
    and using that
    \begin{equation}\label{eq: identià integrale}
        \frac{u(x+\sigma \xi)-u(x)}{\sigma}=\int_{0}^{1}\partial_{\xi}u(x+t\sigma \xi )\, dt,
    \end{equation}
    by Jensen's inequality and Fubini's Theorem we obtain
    \begin{equation*}
        \begin{split}
            G_{\sigma}(u-z,\mathbb{R}^d)&=\int_{B(r_0)}\int_{\rr^d} \Bigl|\int_{0}^{1}\partial_{\xi}u(x+t\sigma \xi )\, dt \Bigr|^d \, dx\,d\xi\\
            &\leq \int_{B(r_0)}\int_{\rr^d} \int_0^1|\partial_{\xi}u(x+t\sigma  \xi)|^d\, dt\, dx\,d\xi\\
            &\leq\int_{B(r_0)}|\xi|^d\int_0^1\int_{\rr^d}|\nabla u(x)|^d\, dx\, dt\,d\xi
            \leq C \int_{Q(L-\s(T+1))}|\nabla u |^d\, dx.
        \end{split}
    \end{equation*}

 If $\ms=\sum_{\alpha\in \s\ZZ^d}\delta_\alpha$, we define the discrete function
    \begin{equation*}
        \hat{u}(\alpha):=\frac{1}{\sigma^d}\int_{\alpha+[0,\sigma)^d}u(x)\, dx, \qquad \alpha\in \s\ZZ^d.
    \end{equation*}
Recalling that now
    \begin{equation*}
        G_\s(\hat{u}-z, \rr^d)= \sum_{k=1}^d\sum_{\alpha\in \s\ZZ^d} |\hat{u}(\alpha+\s e_k)-\hat{u}(\alpha)|^d,
    \end{equation*}
    we apply \eqref{eq: identià integrale} with $\xi\in\{e_1,\dots, e_d\}$ and combine Jensen's inequality with Fubini's Theorem to obtain 
    \begin{equation*}
        \begin{split}
            G_{\sigma}(\hat{u}-z,\mathbb{R}^d)
            &=\sum_{k=1}^d\sum_{\alpha\in \sigma\mathbb{Z}^d}\Bigl|\frac{1}{\sigma^d}\int_{\alpha+[0,\sigma)^d}u(x+ \sigma e_k)-u(x)\, dx\Bigr|^d\\
            & = \sum_{k=1}^d\sum_{\alpha\in \sigma\mathbb{Z}^d}\Bigl|\frac{1}{\sigma^d}\int_{\alpha+[0,\sigma)^d}\s\int_{0}^{1}\partial_{e_k}u(x+t\sigma e_k )\, dt\, dx\Bigr|^d\\   
            &\leq \sum_{k=1}^d\sum_{\alpha\in \sigma\mathbb{Z}^d}\int_{\alpha+[0,\sigma)^d}\int_0^1\left|\partial_{e_k}u(x+t\sigma  e_k)\right|^d\, dt\, dx\\
            &\leq C\sum_{\alpha\in \sigma\mathbb{Z}^d}\int_0^1\int_{\alpha+[0,2\sigma)^d}|\nabla u(x)|^d\, dx\, dt
            \leq C \int_{Q(L-\s(T+1))}|\nabla u |^d\, dx,
            \end{split}
    \end{equation*}
which proves \eqref{claim sui bound}.

Taking into account Remark \ref{rmk: bound uniforme}(i) and Remark \ref{rmk: funzioni ausiliarie}(i), we apply (G1), \eqref{long-range Rd}, and \eqref{claim sui bound} to get 
    \begin{align*}
    \mathcal{F}^T_\sigma(\hat{u},Q(L)\setminus\overline{Q}(\ell))
    & = \mathcal{F}^T_\sigma(\hat{u}-z,\rr^d) \\
    & \leq \int_{\rr^d}M(\xi)\int_{\rr^d}|\hat{u}(x+\s\xi)-\hat{u}(x)|^d\,d\ms(x)\,d\mu_1(\xi) \\
    & \leq C\Bigl(\int_{\rr^d}M(\xi)(1+|\xi|^d)\,d\mu_1(\xi)\Bigr) G_\s(\hat{u}-z,\rr^d) \leq C\int_{Q(L-\s(T+1))}|\nabla u|^d\, dx,
    \end{align*}
    and then, by the arbitrariness of $u$, we conclude
\begin{align*} \notag
    &(\log L)^{d-1}\f^T_{\ell,L,\sigma}(z) \\ \notag
    & \leq C(\log L)^{d-1}\inf\Bigl\{\int_{Q(L-\sigma (T+1))}|\nabla v|^d\, dx: v-z\in C^1_c(Q(L-\sigma (T+1);\mathbb{R}^m),  v=0  \text{ on } Q(\ell+\sigma (T+1))\Bigr\}\\ \notag
    &=C (\log L)^{d-1} {\rm Cap}\displaystyle_{d}\bigl(Q(\ell+\sigma (T+1)),Q(L-\sigma(T+1))\bigr)|z|^d\\
\notag     &\leq C(\log L)^{d-1}\Bigl(\log\Bigl(\frac{L-\sigma (T+1)}{\ell+\sigma(T+1)}\Bigr)\Bigr)^{1-d}|z|^d\\ 
    &\leq C(\log L)^{d-1}\Bigl(\log\Bigl(\frac{L-3/2}{5/2}\Bigr)\Bigr)^{1-d}|z|^d \leq  C|z|^d,
\end{align*}
which proves \eqref{crescita_minimi_discreti}.
    
Now, we prove \eqref{equilip_minimi_discreti} directly addressing both cases. We first observe that, if $z=0$ or $z'=0$, the inequality follows  by \eqref{crescita_minimi_discreti}. Then, we can suppose $z$ and $z'$ both not null and consider the map $\Theta: \mathbb{R}^{m}\rightarrow \mathbb{R}^{m}$ defined by
        \begin{equation*}
            \Theta(\zeta)=\frac{|z'|}{|z|}\mathcal{R}^{z'}_{z}(\zeta), \qquad \zeta\in \rr^m,
        \end{equation*}
        where $\mathcal{R}^{z'}_{z}$ is a rotation that maps $z$ into $\frac{|z|}{|z'|}z'$. Note that $\Theta(0)=0$, $\Theta(z)=z'$ and 
        \begin{equation}
            |\Theta(\zeta)|= \frac{|z'|}{|z|}|\zeta|, \qquad | (\Theta-I)(\zeta)|\leq C \frac{|z-z'|}{|z|}|\zeta|
            \label{31}
        \end{equation}
        for every $\zeta\in\rr^m$, where we let $I$ denote the identity matrix of $\rr^{m\times m}$. 
        
        Fix $\eta>0$ and let $v_{z}$ be (almost) optimal for \eqref{def_minimi_approx}; i.e., $\mathcal{F}^T_{\s}(v_{z},Q(L)\setminus \overline{Q}(\ell)) \leq(1+\eta)\f^T_{\ell,L,\s}(z)$.
 We set $v_{z'}:=\Theta \circ v_z$ and we note that $v_{z'}$ is admissible for the minimum problem $\f^T_{\ell,L,\s}(z')$. 
 Then,
 \begin{align}
\notag \f^T_{\ell,L,\s}(z')&\leq  \mathcal{F}^T_{\s}(v_{z'},Q(L)\setminus \overline{Q}(\ell))\\ \label{stima_funzione_ottimizzante}
     & \leq (1+\eta)\f^T_{\ell,L,\s}(z)+\mathcal{F}^T_{\s}(v_{z'},Q(L)\setminus \overline{Q}(\ell))-\mathcal{F}^T_{\s}(v_{z},Q(L)\setminus \overline{Q}(\ell)).
 \end{align}
 We set $\Delta_\s^\xi w(x):=w(x+\s\xi)-w(x)$. By assumption (L) and \eqref{31} we get
 \begin{align*}
& |f^T(\xi,\Delta_\s^\xi v_{z'}(x))-f^T(\xi,\Delta_\s^\xi v_{z}(x))|    \\      &\leq CM(\xi)(|\Delta_\s^\xi v_{z'}(x)|^{d-1}+|\Delta_\s^\xi v_{z}(x)|^{d-1})|\Delta_\s^\xi v_{z'}(x)-\Delta_\s^\xi v_{z}(x)|\\ \notag
                &=CM(\xi)(|\Theta(\Delta_\s^\xi v_{z}(x))|^{d-1}+|\Delta_\s^\xi v_{z}(x)|^{d-1})|(\Theta-I)(\Delta_\s^\xi v_{z}(x))|\\ \notag
                &\leq CM(\xi)\Bigl[\Bigl(\frac{|z'|}{|z|}\Bigr)^{d-1}|\Delta_\s^\xi v_{z}(x)|^{d-1}+|\Delta^{\xi}_{\s}v_{z}(x)|^{d-1}\Bigr] \frac{|z-z'|}{|z|}|\Delta^{\xi}_{\s}v_{z}(x)|\\
                &\leq CM(\xi)\frac{|z|^{d-1}+|z'|^{d-1}}{|z|^{d}}|z-z'||\Delta^{\xi}_{\s}v_{z}(x)|^{d}
        \end{align*}
        for every $x\in \rr^d$. Then, we apply \eqref{long-range Rd} to get
        \begin{align*}  \notag   
       |\mathcal{F}^T_{\s}&(v_{z'},Q(L)\setminus \overline{Q}(\ell))-\mathcal{F}^T_{\s}(v_{z},Q(L)\setminus \overline{Q}(\ell))|\\ \notag
       &\leq \int_{\rr^d}\int_{\rr^d}|f^T(\xi,\Delta^\xi_\sigma v_{z'}(x))-f^T(\xi,\Delta^\xi_\sigma v_z(x))|\, d\ms(x)\,d\mu_1(\xi)\\ \notag
       &\leq C\Bigl(\int_{\mathbb{R}^{d}}M(\xi)\int_{\rr^d} |v_{z}(x+\s\xi)-v_{z}(x)|^{d}\,d\ms(x)\, d\mu_1(\xi)\Bigr)\frac{|z|^{d-1}+|z'|^{d-1}}{|z|^{d}}|z-z'|\\ 
       &\leq C\Bigl(\int_{\mathbb{R}^{d}}M(\xi)\,(1+|\xi|^d) \,d\mu_1(\xi)\Bigr) G_\sigma(v_z-z,\mathbb{R}^d)\frac{|z|^{d-1}+|z'|^{d-1}}{|z|^{d}}|z-z'|.
   \end{align*}
Since $T>2>r_0$, inequality \eqref{eq: riferimento lower} holds, and by (G1) and Remark \ref{rmk: funzioni ausiliarie}(i) we obtain
        \begin{align*}\notag
        |\mathcal{F}^T_{\s}(v_{z'},Q(L)\setminus \overline{Q}(\ell))-\mathcal{F}^T_{\s}(v_{z},Q(L)\setminus \overline{Q}(\ell))|      
       & \leq C\mathcal{F}^{T}_{\s}(v_{z}-z,\rr^d)\frac{|z|^{d-1}+|z'|^{d-1}}{|z|^{d}}|z-z'| \\ 
       & \leq C(1+\eta)\f^{T}_{\ell,L,\s}(z) \frac{|z|^{d-1}+|z'|^{d-1}}{|z|^{d}}|z-z'|.
        \end{align*}
Multiplying by $(\log L)^{d-1}$ on both sides of \eqref{stima_funzione_ottimizzante} and using \eqref{crescita_minimi_discreti} we infer
   \begin{align*}
       (\log L)^{d-1}\f^T_{\ell,L,\s}(z') &\leq (1+\eta)(\log L)^{d-1}\f^T_{\ell,L,\s}(z)\\
       & \quad +(\log L)^{d-1}|\mathcal{F}^T_{\s}(v_{z'},Q(L)\setminus \overline{Q}(\ell))-\mathcal{F}^T_{\s}(v_{z},Q(L)\setminus \overline{Q}(\ell))|\\
       & \leq (1+\eta)\Bigl[(\log L)^{d-1}\f^T_{\ell,L,\s}(z)+C(\log L)^{d-1}\f^{T}_{\ell,L,\s}(z) \frac{|z|^{d-1}+|z'|^{d-1}}{|z|^{d}}|z-z'|\Bigr] \\
       & \leq (1+\eta)\Bigl[(\log L)^{d-1}\f^T_{\ell,L,\s}(z)+C(|z|^{d-1}+|z'|^{d-1})|z-z'|\Bigr].
   \end{align*}
Letting $\eta\to0$ and reversing the roles of $z$ and $z'$ we obtain \eqref{equilip_minimi_discreti}, which concludes the proof.
\end{proof}

The following result illustrates the asymptotics of the energy densities introduced at the beginning of this section.

\begin{proposition}\label{proposition: psi limite} The following hold:
\begin{itemize}
  \item[(i)] let $T\in \NN, T>2, \ell\leq1,$ and $L>4$, and assume that 
  \begin{equation*}
  \ms=\frac{1}{\s^d}\mathcal{L}^d \qquad \text{ or } \qquad \ms=\sum_{\alpha\in \s\ZZ^d}\delta_\alpha    
  \end{equation*}
   for every $\s\in(0,1/T)$. We have
\begin{equation*}
    \lim_{\s\to0}\frac{\f^T_{\ell,L,\s}}{\f^T_{L/\ell}}=1
\end{equation*}
uniformly on $\rr^m\setminus\{0\}$;
\item[(ii)] for every $T\in\NN,T>2$ we have
\begin{equation*}
   \lim_{L\to+\infty} \frac{(\log L)^{d-1}\f^T_{L}}{\f^T}=1
\end{equation*}
uniformly on $\rr^m\setminus\{0\}$;
\item[(iii)] we have that $\{\f^T(z)\}_T$ is an increasing sequence for every $z\in \rr^m$ and
\begin{equation*}
   \lim_{T\to+\infty}\f^T= \f
\end{equation*}
uniformly on compact subsets of $\rr^m$.
\end{itemize}
\end{proposition}

\begin{proof}
To prove $(i)$, we first show that
\begin{equation}
\label{convergenza_discreto_continuo}
    \lim_{\s\to0}\f^T_{\ell,L,\s} =\f^T_{L/\ell}
\end{equation}
uniformly on compact subsets of $\rr^m$.
By \eqref{equilip_minimi_discreti} it is enough to prove the pointwise convergence. Recalling \eqref{A^sT}, we have that
    \begin{equation*}
        \begin{split}
            \{ v&\in \A_{{\mu_\sigma}}(Q(L); \rr^m):  v(x)= 0 \text{ for } \ms\text{-a.e. } x\in\overline{Q}(\ell+\sigma T),
     v(x)= z \text{ for } \ms\text{-a.e. } x\in Q(L-\sigma T)^c \}\\
     &=\{ v\in \A_{\mu_\s}(Q(L); \rr^m) : v(x)=g(x) \text{ for } \ms\text{-a.e. } x\in Q(L)\setminus \Q(\ell), \text{\rm dist}_\infty(x,[Q(L)\setminus \Q(\ell)]^c)\leq \s T\} \\
     & = \A_{\ms}^{\s T,g}(Q(L)\setminus \Q(\ell)),
        \end{split}
    \end{equation*}
    where $g: \rr^d\to \rr^m$ is a Lipschitz function such that
    \begin{equation*}
 g(x)=
        \begin{cases}
            0 &\text{ if } x \in \overline{Q}(\ell+1), \\
            z &\text{ if } x\in Q(L-1)^c.
        \end{cases}
    \end{equation*}    
    Then, we can rewrite $\f^T_{\ell,L,\sigma}(z)$ as
\begin{equation*}
        \inf\{ \F^T_{\sigma}(v, Q(L)\setminus \overline{Q}(\ell)) : v\in \A_{\ms}^{\s T,g}(Q(L)\setminus \Q(\ell))\}
\end{equation*}
and, applying the second part of Theorem \ref{theorem: Gamma unconstrained}, we obtain
\begin{equation*}
    \lim_{\s\to0}\f^T_{\ell,L,\sigma}(z)  = \min\Bigl\{ \int_{Q(L)\setminus\Q(\ell)}f^T_{\rm hom}(\nabla v)\, dx : v-g\in W^{1,d}_0(Q(L)\setminus\Q(\ell);\mathbb{R}^m)\Bigr\}.
\end{equation*}
Upon setting $v=0$ in $Q(\ell)$, since $f^T_{\rm hom}(0)=0$, we get
\begin{align*}
     \lim_{\s\to0}\f^T_{\ell,L,\sigma}(z) 
     & = \min\Bigl\{ \int_{Q(L)}f^T_{\rm hom}(\nabla v)\, dx : v-z\in W^{1,d}_0(Q(L);\mathbb{R}^m), v= 0 \text{ on } Q(\ell)\Bigr\} \\
     & =\min\Bigl\{ \int_{Q(L/\ell)}f^T_{\rm hom}(\nabla v)\, dx : v-z\in W^{1,d}_0(Q(L/\ell);\mathbb{R}^m), v= 0 \text{ on } Q(1)\Bigr\}=\f^T_{L/\ell}(z),
\end{align*}
where the second equality follows by a change of variables and the $d$-homogeneity of $f^T_{\rm hom}$, and this proves \eqref{convergenza_discreto_continuo}.

Now, we observe that there exists a positive constant $C$ such that
    \begin{equation}
    \label{crescitaaa}
        \frac{1}{C}|z|^d\leq (\log (L/\ell))^{d-1}\f^T_{L/\ell}(z) \leq C|z|^d
    \end{equation}
    for every $T,\ell, L$ and for every $z \in \mathbb{R}^m$. The upper bound readily follows passing to the limit in \eqref{crescita_minimi_discreti} by \eqref{convergenza_discreto_continuo}. For the lower bound, it suffices to note that $f_{\rm hom}^T(M)\geq |M|^d/C$ by \eqref{in rmk f_hom^T}, and the estimate follows by \eqref{stima uniforme capacita}. 
    
    We conclude the proof of $(i)$. By Remark \ref{rmk: funzioni ausiliarie}(iv), the function $\f^T_{\ell,L,\s}/ \f^T_{L/\ell}$ is $0$-homogeneous for every $\s\in(0, 1/T)$. Letting then $\zeta:=z/|z|$ for every $z\in\rr^m\setminus\{0\}$ and using the lower bound in \eqref{crescitaaa}, we have
    \begin{equation*}
        \biggl|\frac{\f_{\ell,L,\s}^T(z)}{\f^T_{L/\ell}(z)}-1\biggr|=\biggl|\frac{\f_{\ell,L,\s}^T(\zeta)}{\f^T_{L/\ell}(\zeta)}-1\biggr|= \biggl|\frac{\f_{\ell,L,\s}^T(\zeta)-\f_{L/\ell}^T(\zeta)}{\f_{L/\ell}^T(\zeta)}\biggr|\leq C(\log (L/\ell))^{d-1}|\f_{\ell,L,\s}^T(\zeta)-\f_{L/\ell}^T(\zeta)|,
    \end{equation*}
and the conclusion follows since the convergence in \eqref{convergenza_discreto_continuo} is uniform on $\mathbb{S}^{m-1}$.
   
    The proof of $(ii)$ follows the same line as the proof of $(i)$ upon observing that we have
\begin{equation*}
   \lim_{L\to+\infty} (\log L)^{d-1}\f^T_{L} =  \f^T
\end{equation*}
uniformly on compact subsets of $\rr^m$ for every $T>2$.  To see this, we apply \eqref{convergenza_discreto_continuo} with $\ell=1$ in order to pass to the limit in \eqref{equilip_minimi_discreti} as $\s\to0$ and get
\begin{equation}
\label{equilip1}
    |(\log L)^{d-1}\f_{L}^T(z)-(\log L)^{d-1}\f_{L}^T(z')|\leq C(|z|^{d-1}+|z'|^{d-1})|z-z'|
\end{equation}
for every $T>2,L>4$, and for every $z,z'\in \rr^m$. As a consequence, $\{(\log L)^{d-1}\f^T_L\}_L$ is a family of locally equi-lipschitz continuous functions on $\rr^m$, and since the pointwise convergence follows by Proposition \ref{prop: limite log}, the uniform convergence on compact subsets of $\mathbb{R}^m$ is proved.  

    Finally, we prove $(iii)$.  Let $z\in\rr^m$. Recalling the definitions of $\f^T(z)$ and $\f(z)$ in \eqref{limite_in_L} and \eqref{limit_denisty}, respectively, our aim is to prove that
    \begin{equation*}
    \begin{split}
        \lim_{T\to +\infty}&\biggl[\lim_{L\to +\infty}(\log L)^{d-1}\min\Bigl\{ \int_{Q(L)}f^T_{\rm hom}(\nabla v)\, dx : v-z\in W^{1,d}_0(Q(L);\mathbb{R}^m), v= 0 \text{ on } Q( 1)\Bigr\}\biggr]\\
        &=\lim_{L\to +\infty}(\log L)^{d-1}\min\Bigl\{ \int_{Q(L)}f_{\rm hom}(\nabla v)\, dx : v-z\in W^{1,d}_0(Q(L);\mathbb{R}^m), v= 0 \text{ on } Q(1)\Bigr\}
    \end{split}
    \end{equation*}
      uniformly on compact subsets of $\rr^m$.
      
    We claim that there exists a positive constant $C$ with the property that for every $\eta>0$ there exists $T_\eta>0$ such that
    \begin{multline}
    \label{claim_1}
       (\log L)^{d-1}\Bigl|\min\Bigl\{ \int_{Q(L)}f_{\rm hom}(\nabla v)\, dx : v-z\in W^{1,d}_0(Q(L);\mathbb{R}^m), v=0 \text{ on } Q(1)\Bigr\} \\
       -  \min\Bigl\{ \int_{Q(L)}f^T_{\rm hom}(\nabla v)\, dx : v-z\in W^{1,d}_0(Q(L);\mathbb{R}^m), v=0 \text{ on } Q(1)\Bigr\}\Bigr|\leq C\eta|z|^d
    \end{multline}
for every $L>4, T>T_\eta,$ and $z\in\rr^m$.

Let $\eta>0$. By Proposition \ref{prop_conv_fhom_uniforme}, there exists $T_\eta$ such that
\begin{equation}
\label{fhom_vicine_a_piacere}
    |f_{\hom}(M)-f^T_{\hom}(M)|\leq \eta
\end{equation}
for every $M \in \mathbb{R}^{m\times d}$ with $|M|=1$ and for every $T>T_\eta$. Let now $
\overline{v}$ be a minimizer for $\f^T_L(z)$; i.e., for
\begin{equation*}
    \min\Bigl\{ \int_{Q(L)}f^T_{\rm hom}(\nabla v)\, dx : v-z\in W^{1,d}_0(Q(L);\mathbb{R}^m), v= 0 \text{ on } Q(1)\Bigr\}.
\end{equation*}
Recalling that $f^T_{\hom}$ and $f_{\hom}$ are $d$-homogeneous and that $f^T_{\hom}(0)=f_{\hom}(0)=0$, we use \eqref{in rmk f_hom^T}, the upper bound in \eqref{crescitaaa} with $\ell=1$, and  \eqref{fhom_vicine_a_piacere} to get
\begin{equation*}
\begin{split}
   (\log L)^{d-1}&\Bigl[\min\Bigl\{ \int_{Q(L)}f_{\rm hom}(\nabla v)\, dx : v-z\in W^{1,d}_0(Q(L);\mathbb{R}^m), v=0 \text{ on } Q(1)\Bigr\}- \f^T_L(z)\Bigr]\\
   &\leq (\log L)^{d-1}\Bigl[\int_{Q(L)}f_{\rm hom}(\nabla \overline{v})\, dx -\int_{Q(L)}f^T_{\rm hom}(\nabla \overline{v})\, dx\Bigr]\\
   &=(\log L)^{d-1}\int_{Q(L)\setminus\{\nabla \overline{v}=0\}}\Bigl[f_{\rm hom}\Bigl(\frac{\nabla \overline{v}}{|\nabla \overline{v}|}\Bigr)-f^T_{\rm hom}\Bigl(\frac{\nabla \overline{v}}{|\nabla \overline{v}|}\Bigr)\Bigr]|\nabla \overline{v}|^d\, dx\\
   &\leq (\log L)^{d-1}\eta\int_{Q(L)}|\nabla \overline{v}|^d\, dx\\
   &  \leq C(\log L)^{d-1}\eta\int_{Q(L)}f^T_{\rm hom}(\nabla \overline{v})\, dx\\
   &= C\eta(\log L)^{d-1} \f^T_L(z)\leq C\eta |z|^d.
\end{split}
\end{equation*}
At this point we can repeat the same argument choosing $\overline{v}$ minimizer for $$\min\Bigl\{ \int_{Q(L)}f_{\rm hom}(\nabla v)\, dx : v-z\in W^{1,d}_0(Q(L);\mathbb{R}^m), v=0\text{ on } Q(1)\Bigr\},$$ which proves the other inequality in order to get our claim. Finally, passing to the limit as $L\to +\infty$ in \eqref{claim_1} we get
\begin{equation*}
|\f(z)-\f^T(z)|\leq C \eta|z|^d
\end{equation*}
for every $T>T_\eta$ and $z\in\rr^m$; then the conclusion follows by the arbitrariness of $\eta$. Since, by definition, $\f^T\leq\f^{T+1}$ for every $T>2$, the convergence is monotone and the proof is concluded.
\end{proof}

\begin{remark}
By Proposition \ref{proposition: psi limite}$(ii)$, letting $L\to+\infty$ in \eqref{equilip1} we get 
    \begin{equation}\label{equilip2}
     |\f^T(z)-\f^T(z')|\leq C(|z|^{d-1}+|z'|^{d-1})|z-z'|  
    \end{equation}  
for every $T>2$, and for every $z,z'\in \rr^m$, and then $\{\f^T\}_T$ are locally equi-lipschitz continuous functions on $\rr^m$.
\end{remark}

\section{Asymptotic analysis}
\label{Sec: asymptotic analysis}

This section is devoted to the proof of Theorem \ref{thm: main discreto}. We fix $\{\e_j\}_j, \{\dej\}_j, \{\rj\}_j$ positive sequences tending to $0$ as $j\to+\infty$ and such that
\begin{equation}\label{scelta parametri}
    \lim_{j\to+\infty} \frac{\ej}{\rj}=0, \qquad \lim_{j\to+\infty}\frac{\exp(-(\gamma\dej^d)^{\frac{1}{1-d}})}{\rj}=1
\end{equation}
for some $\gamma>0$.
We immediately observe that \eqref{scelta parametri} implies
\begin{equation*}
    \lim_{j\to+\infty}\frac{|\log \rj|^{1-d}}{\dej^d}=\gamma \qquad \text{ and } \qquad \lim_{j\to+\infty}\frac{\bigl(\log\bigl(\frac{\dej}{2\rj}\bigr)\bigr)}{|\log \rj|}=1,
\end{equation*}
which, in turn, yields \begin{equation}\label{dej^d}
\lim_{j\to+\infty}\frac{\bigl(\log\bigl(\frac{\dej}{2\rj}\bigr)\bigr)^{1-d}}{\dej^d}=\gamma.
\end{equation}

\subsection{Auxiliary results}

The main ingredient of our proof is the following `joining lemma on varying domains', which is a variant of that originally devised in \cite{AnsBraJMPA}; see also, e.g., \cite{Sig-dis}. This is used to modify a sequence with equi-bounded energy so that constant Dirichlet boundary conditions are imposed on several concentric dyadic `frames' surrounding the perforations. To make the construction more transparent and consistent with the discrete case, we consider frames with side-lengths that are multiples of $\ej$, that corresponds to the spacing of the cubic lattice in the discrete setting.

\begin{lemma}\label{lemma: joining}
Let $T,k,N$ be fixed integers such that $T>2$, $2\leq k\leq N-3$, and let $\{S_j\}_j$ be a sequence of positive integers such that
\begin{equation}\label{joining geo hyp}
\rj2^{NS_j}< \frac{\dej}{2} 
\end{equation}
for every $j\in\NN$. Let $u_j\in \mathcal{A}_{\varepsilon_j}(\Omega;\mathbb{R}^{m}),j\in\NN,$ be a sequence such that $u_j\rightarrow u$ in $L^{d}(\Omega;\mathbb{R}^{m})$ for some $u \in W^{1,d}(\Omega;\mathbb{R}^{m})$ as $j\to+\infty$ and let
\begin{equation}\label{insieme di indici}
    Z_{j}(\Omega):=\Bigl\{i \in \mathbb{Z}^{d}\cap \frac{1}{\delta_j}\Omega: \text{\rm dist}_\infty(i\delta_{j},\Omega^c)>\delta_j\Bigr\}
\end{equation}
for every $j\in\NN$. Assume that 
\begin{equation}\label{hp: e/r}
    \lim_{j\to+\infty}\frac{\ej}{\rj}=0
\end{equation}
and 
\begin{equation}\label{joining equibdd hyp}
\sup_{j}\F^T_{\ej}(u_j)<+\infty.
\end{equation}
Then, there exists a sequence $w_{j}\in \mathcal{A}_{\varepsilon_{j}}(\Omega;\mathbb{R}^{m}), j\in\NN,$ converging to $u$ in $L^{d}(\Omega;\mathbb{R}^{m})$ as $j\to+\infty$ with the property that for every $j$ large enough, for every $i \in Z_{j}(\Omega)$, and for every $h\in\{1,\dots,S_j\}$ there exists $k_{i,h} \in \{0,\dots, k-1\}$ such that, having set  
\begin{equation}
\label{def_cornici}
C^{i, h}_{j}:=Q\Bigl(i\delta_{j},\Bigl\lfloor\frac{\rj}{\varepsilon_{j}}2^{Nh-k_{i, h}}\Bigr\rfloor\varepsilon_{j}\Bigr)\setminus   \Q\Bigl(i\delta_{j},\Bigl\lfloor\frac{\rj}{\varepsilon_{j}}2^{Nh-k_{i,h}-1}\Bigr\rfloor\varepsilon_{j}\Bigr),
\end{equation}
\begin{equation*}
    \rho^{i, h}_{j}:=\Bigl\lfloor\frac{3}{4}\frac{\rj}{\varepsilon_{j}}2^{Nh-k_{i,h}}\Bigr\rfloor\varepsilon_{j},
\end{equation*}
\begin{equation*}
     u_{j}^{i, h}:=\frac{1}{\mu_{\ej} (C^{i,h}_j)}\int_{C^{i,h}_j} u_{j}\, d\mu_{\ej},
\end{equation*}
the following hold:
\begin{equation}\label{outcornici}
    w_{j}=u_{j}\qquad \mu_{\ej}  \text{-a.e. on } \Omega\setminus \Bigl[ \bigcup_{\substack{i \in Z_{j}(\Omega)\\ h\in\{1,\dots, S_j\}}}C_{j}^{i,h}\Bigr],
\end{equation}
\begin{equation}\label{bordointerno}
    w_{j}=u_{j}^{i, h}\qquad \mu_{\ej} \text{-a.e. on } \Q(i\delta_j,\rho^{i,h}_{j}+\varepsilon_j T)\setminus Q(i\delta_j,\rho^{i, h}_{j}-\varepsilon_j T),
\end{equation}
\begin{equation}\label{diffenergia}
    |\F^T_{\ej}(u_j)-\F^T_{\ej}(w_j)|\leq \frac{C}{k}
\end{equation}
for some $C>0$ independent of $j$.  
\end{lemma}

\begin{proof}
For every $j\in\NN, i\in Z_j(\Omega)$, $h\in\{1,\dots,S_j\}$, and $l\in\{0,\dots,k-1\}$,  consider
\begin{equation*}
C^{i, h}_{j,l}:= Q\Bigl(i\delta_{j},\Bigl\lfloor\frac{\rj}{\varepsilon_{j}}2^{Nh-l}\Bigr\rfloor\varepsilon_{j}\Bigr)\setminus   \Q\Bigl(i\delta_{j},\Bigl\lfloor\frac{\rj}{\varepsilon_{j}}2^{Nh-l-1}\Bigr\rfloor\varepsilon_{j}\Bigr),
\end{equation*}
\begin{equation*}
    \rho^{ h}_{j,l}:=\Bigl\lfloor\frac{3}{4}\frac{\rj}{\varepsilon_{j}}2^{Nh-l}\Bigr\rfloor\varepsilon_{j},
\end{equation*}
\begin{equation*}
     u_{j,l}^{i, h}:=\frac{1}{\mu_{\ej}(C_{j,l}^{i,h})}\int_{C^{i, h}_{j,l}} u_{j}\, d\mu_{\ej}.
\end{equation*}
We define the open sets $C^{i, h}_{j,l,T}$ as
\begin{equation*}
C^{i, h}_{j,l,T}:=Q\Bigl(i\delta_{j},\Bigl\lfloor\frac{\rj}{\varepsilon_{j}}2^{Nh-l}\Bigr\rfloor\varepsilon_{j}-\ej T\Bigr)\setminus   \Q\Bigl(i\delta_{j},\Bigl\lfloor\frac{\rj}{\varepsilon_{j}}2^{Nh-l-1}\Bigr\rfloor\varepsilon_{j}+\ej T\Bigr),
\end{equation*}
and cut-off functions $\chi^{i,h}_{j,l}\in C^\infty_c(C^{i, h}_{j,l,T})$ such that
\begin{equation*}
    \begin{cases}
        0\leq \chi^{i,h}_{j,l} \leq 1 & \text{ on } \rr^d, \\[5pt]
        \chi^{i,h}_{j,l}= 1 & \text{ on } \Q(i\dej,\rho^{h}_{j,l}+\ej T)\setminus Q(i\dej,\rho^{h}_{j,l}-\ej T), \\[5pt]
        |\nabla\chi^{i,h}_{j,l}| \leq \frac{C}{\rho^{h}_{j,l}} & \text{ on } \rr^d.
    \end{cases}
\end{equation*}
We remark that, if we assume that $j$ is large enough, then $C^{i,h}_{j,l,T}$ and $\chi^{i,h}_{j,l}$ are well defined for every $i\in Z_j(\Omega), h\in\{1,\dots,S_j\},$ and $l\in\{0,\dots,k-1\}$. It suffices to note that, by a straightforward computation, the inequalities
\begin{equation*}
   \Bigl\lfloor \frac{\rj}{\ej}2^{Nh-l}\Bigr\rfloor\ej-\ej T>  \rho_{j,l}^{h}+\ej T \qquad \text{ and } \qquad \rho_{j,l}^{h}-\ej T>  \Bigl\lfloor \frac{\rj}{\ej}2^{Nh-l-1}\Bigr\rfloor\ej+\ej T
\end{equation*}
are satisfied for every $h,l$ provided that $\rj/\ej>4(2T+1)$, which holds for $j$ sufficiently large in light of the assumption \eqref{hp: e/r}. Similarly, the estimate on $|\nabla\chi^{i,h}_{j,l}|$ follows by 
\begin{equation*}
    |\nabla\chi^{i,h}_{j,l}|\leq \frac{C}{\Bigl(\Bigl\lfloor\frac{\rj}{\varepsilon_{j}}2^{Nh-l}\Bigr\rfloor\varepsilon_{j}-\ej T\Bigr)-(\rho^{h}_{j,l}+\ej T)} \leq \frac{C}{\rho^{h}_{j,l}}.
\end{equation*}

Let then
\begin{equation*}
    w_{j,l}^{i,h}(x):=\chi^{i,h}_{j,l}(x)u_{j,l}^{i,h}+(1-\chi^{i,h}_{j,l}(x))u_j(x) \qquad \text{ for } \mu_{\ej}\text{-a.e. } x\in\Omega.
\end{equation*}
We aim at estimating $\F^T_{\ej}(w_{j,l}^{i,h},C^{i, h}_{j,l})$. Using Remark \ref{rmk: bound uniforme}(i), we get
\begin{align}\notag
    \F^T_{\ej}(w_{j,l}^{i,h},C^{i, h}_{j,l}) & = \int_{\mathbb{R}^d}\int_{(C^{i, h}_{j,l})_{\mu_{\ej}}(\xi)}f^T(\xi,w^{i,h}_{j,l}(x+\ej\xi)-w^{i,h}_{j,l}(x)) \,d\mu_{\ej}(x)\, d\mu_1(\xi)\\ \label{joining 0}
    & \leq \int_{B(T)}M(\xi)\int_{(C^{i, h}_{j,l})_{\mu_{\ej}}(\xi)}|w^{i,h}_{j,l}(x+\ej\xi)-w^{i,h}_{j,l}(x)|^d\,d\mu_{\ej}(x)\, d\mu_1(\xi).
\end{align}
By definition, we have
\begin{equation*}
   w^{i,h}_{j,l}(x+\ej\xi)-w^{i,h}_{j,l}(x)=(u_{j,l}^{i,h}-u_j(x+\ej\xi)) (\chi^{i,h}_{j,l}(x+\ej\xi)-\chi^{i,h}_{j,l}(x))+(1-\chi^{i,h}_{j,l}(x))(u_{j}(x+\ej\xi)-u_j(x)),
\end{equation*}
and then, by the estimate on $\nabla \chi^{i,h}_{j,l}$, it holds that
\begin{align*}
   | w^{i,h}_{j,l}(x+\ej\xi)-w^{i,h}_{j,l}(x)|^d \leq C |u_{j,l}^{i,h}-u_j(x+\ej\xi)|^d\Bigl|    \frac{\ej}{\rho^{h}_{j,l}}\xi\Bigr|^{d} +  C|u_{j}(x+\ej\xi)-u_{j}(x)|^d. 
\end{align*}
Substituting in \eqref{joining 0} and recalling (G1), we obtain
\begin{align} \notag
     \F^T_{\ej}(w_{j,l}^{i,h},C^{i, h}_{j,l}) & \leq C\Bigl(\frac{\ej}{\rho^{h}_{j,l} }\Bigr)^d\int_{B(T)}M(\xi)|\xi|^d\int_{(C^{i, h}_{j,l})_{\mu_{\ej}}(\xi)}|u_{j,l}^{i,h}-u_j(x+\ej\xi)|^d\, d\mu_{\ej}(x)\, d\mu_1(\xi) \\ \notag
     & \quad+  C\int_{ B(T)}M(\xi)\int_{(C^{i, h}_{j,l})_{\mu_{\ej}}(\xi)}|u_{j}(x+\ej\xi)-u_{j}(x)|^d\, d\mu_{\ej}(x)\, d\mu_1(\xi) \\ \notag
     &  \leq C\Bigl(\frac{\ej}{\rho^{h}_{j,l} }\Bigr)^d\Bigl(\int_{B(T)}M(\xi)|\xi|^dd\mu_1(\xi)\Bigr)\int_{ C^{i, h}_{j,l}}|u_{j,l}^{i,h}-u_j(x)|^d\,d\mu_{\ej}(x) \\ \label{joining 1} 
     & \quad+  C\int_{ B(T)}M(\xi)\int_{C^{i, h}_{j,l}}|u_{j}(x+\ej\xi)-u_{j}(x)|^d\, d\mu_{\ej}(x)\, d\mu_1(\xi),
     \end{align}
where in the last inequality we used that, upon assuming $j$ large enough,
\begin{equation}\label{distanza per long-range}
    \text{\rm dist}
_\infty(C_{j,l}^{i,h}, \Omega^c)>\frac{\dej}{2}>(T+3\sqrt{d})\ej
\end{equation}
in light of \eqref{joining geo hyp}, \eqref{insieme di indici}, and the fact that $\ej/\dej\to0$; and therefore $C^{i,h}_{j,l}+\ej\xi\subset \Omega$ for every $\xi\in B(T), i\in Z_j(\Omega), h\in\{1,\dots,S_j\},$ and $l\in\{0,\dots,k-1\}$. 

For the sake of convenience, we introduce the sets
\begin{equation}\label{Ejlih}
    E^{i, h}_{j,l}:=Q\Bigl(i\delta_{j},4\Bigl\lfloor\frac{\rj}{\varepsilon_{j}}2^{Nh-l-1}\Bigr\rfloor\varepsilon_{j}\Bigr)\setminus   \Q\Bigl(i\delta_{j},\frac{1}{2}\Bigl\lfloor\frac{\rj}{\varepsilon_{j}}2^{Nh-l-1}\Bigr\rfloor\varepsilon_{j}\Bigr),
\end{equation}
and we note that, in light of \eqref{hp: e/r}, it holds that
\begin{equation}\label{eq: per usare lemma long range}
    \frac{\mu_{\ej}(E^{i,h}_{j,l})}{\mu_{\ej}(C^{i,h}_{j,l})}\leq 2^{d+1}
\end{equation}
and that
\begin{align}\notag
& C^{i, h}_{j,l}+Q((T+3\sqrt{d})\ej)\\ \notag
& = Q\Bigl(i\delta_{j},\Bigl\lfloor\frac{\rj}{\varepsilon_{j}}2^{Nh-l}\Bigr\rfloor\varepsilon_{j}+(T+3\sqrt{d})\ej\Bigr)\setminus   \Q\Bigl(i\delta_{j},\Bigl\lfloor\frac{\rj}{\varepsilon_{j}}2^{Nh-l-1}\Bigr\rfloor\varepsilon_{j}-(T+3\sqrt{d})\ej\Bigr)\\ \label{eq: per usare lemma long range 2}
&\subseteq E^{i, h}_{j,l}\subset \Omega
\end{align} 
for every $j$ large enough and for every $i,h,l$. 

To estimate the first term in \eqref{joining 1}, we apply the rescaled version of Poincaré-Wirtinger inequality in Proposition \ref{prop: poincare-wirtinger}, with
\begin{equation*}
\begin{split}
    \sigma &=\ej,  \qquad \tau =\Bigl\lfloor\frac{\rj}{\varepsilon_{j}}2^{Nh-l-1}\Bigr\rfloor\ej, \qquad x_0=i\dej, \\A=\frac{1}{\Bigl\lfloor\frac{\rj}{\varepsilon_{j}}2^{Nh-l-1}\Bigr\rfloor\ej}(E^{i,h}_{j,l}-i\dej)&=Q(4)\setminus \Q(1/2), \qquad A' = \frac{1}{\Bigl\lfloor\frac{\rj}{\varepsilon_{j}}2^{Nh-l-1}\Bigr\rfloor\ej}(C^{i,h}_{j,l}-i\dej).
\end{split}
\end{equation*}
Note indeed that, given $\s_0$ as in Proposition \ref{prop: poincare-wirtinger}, for $j$ large enough we have $\ej<\sigma_0\bigl\lfloor\frac{\rj}{\varepsilon_{j}}2^{Nh-l-1}\bigr\rfloor\ej$ for every $h,l$ by \eqref{hp: e/r}. Upon choosing $j$ large enough, it also holds that $\rj2^{Nh-l-1}\leq 2\rho^{h}_{j,l}$ for every $h,l$; then, combining this observation with \eqref{eq: riferimento lower} and \eqref{eq: per usare lemma long range}, we infer that
\begin{align}\notag
        \int_{ C^{i, h}_{j,l}}|u_{j,l}^{i,h}-u_j(x)|^d\,d\mu_{\ej}(x) & \leq  \int_{ E^{i, h}_{j,l}}|u_{j,l}^{i,h}-u_j(x)|^d\,d\mu_{\ej}(x) \\ \notag
        & \leq C \frac{\mu_{\ej}(E^{i,h}_{j,l})}{\mu_{\ej}(C^{i,h}_{j,l})}\Bigl\lfloor\frac{\rj}{\varepsilon_{j}}2^{Nh-l-1}\Bigr\rfloor^dG_{\ej}(u_j, E^{i, h}_{j,l}) \\ \label{joining poincare bis}
        & \leq C \Bigl(\frac{\rho^h_{j,l}}{\ej}\Bigr)^dG_{\ej}(u_j, E^{i, h}_{j,l})\leq C\Bigl(\frac{\rho^h_{j,l}}{\ej}\Bigr)^d\F^T_{\ej}(u_j,E^{i, h}_{j,l}),
        \end{align}
where $C$ depends on the Poincaré-Wirtinger constant of $Q(4)\setminus \Q(1/2)$ and is independent of all the involved indexes.

Now, we estimate the second term in \eqref{joining 1}. By \eqref{distanza per long-range}, we are in position to apply Lemma \ref{lemma: controllo long range on frames} with 
\begin{equation*}
    \sigma =\ej,  \qquad A=C^{i,h}_{j,l},
\end{equation*}
and this, combined with (G1), \eqref{eq: riferimento lower}, and \eqref{eq: per usare lemma long range 2}, yields
\begin{align}\notag
  & \int_{ B_T}M(\xi)\int_{C^{i, h}_{j,l}}|u_{j}(x+\ej\xi)-u_{j}(x)|^d\, d\mu_{\ej}(x)\, d\mu_1(\xi) \\ \notag
   &\leq C\Bigl(\int_{\rr^d}M(\xi)(1+|\xi|^{d})\, d\mu_1(\xi)\Bigr)G_{\ej}(u_j, E^{i, h}_{j,l})\\ \label{joining 1.5}
   &\leq C\F^T_{\ej}(u_j,E^{i, h}_{j,l}),
\end{align}
where, also in this case, the constant $C$ does not depend on the involved indexes.

We substitute \eqref{joining poincare bis} and \eqref{joining 1.5} in \eqref{joining 1} to obtain
        \begin{equation*}
     \F^T_{\ej}(w_{j,l}^{i,h},C^{i, h}_{j,l}) \leq  C\F^T_{\ej}(u_j,E^{i, h}_{j,l}),
     \end{equation*}
and we deduce
     \begin{equation}\label{joining 1.5bis}
          |\F^T_{\ej}(u_j,C_{j,l}^{i,h})-\F^T_{\ej}(w_{j,l}^{i,h},C_{j,l}^{i,h})|\leq C \F^T_{\ej}(u_j,E_{j,l}^{i,h}).
     \end{equation}
Since the sets $C_{j,l}^{i,h}$ are pairwise disjoint, by construction it follows that the sets $E^{i,h}_{j,l}$ mutually intersect at most four at a time when $i,h$ are fixed and $l$ ranges in $\{0,\dots,k-1\}$, and  
     \begin{equation*}
      \bigcup_{l=0}^{k-1}E^{i,h}_{j,l}\subseteq Q(i\dej, \rj2^{Nh+1})\setminus \Q(i\dej, \rj2^{Nh-k-1}-\ej/2)    =: E^{i,h}_j.
     \end{equation*}
     Hence, summing over $l\in\{0,\dots,k-1\}$ in \eqref{joining 1.5bis}, we obtain
     \begin{equation*}
         \sum_{l=0}^{k-1} |\F^T_{\ej}(u_j,C_{j,l}^{i,h})-\F^T_{\ej}(w_{j,l}^{i,h},C_{j,l}^{i,h})| \leq C\F^T_{\ej}(u_j,E_j^{i,h}),
     \end{equation*}
and therefore, there exists $k_{i,h}\in\{0,\dots,k-1\}$ such that
\begin{equation}\label{joining 2}
     |\F^T_{\ej}(u_j,C_{j,k_{i,h}}^{i,h})-\F^T_{\ej}(w_{j,k_{i,h}}^{i,h},C_{j,k_{i,h}}^{i,h})| \leq \frac{C}{k}\F^T_{\ej}(u_j,E_j^{i,h}).
\end{equation}

We obtain that conditions \eqref{outcornici}, \eqref{bordointerno}, and \eqref{diffenergia} are satisfied choosing
\begin{equation*}
    C^{i,h}_j:=C^{i,h}_{j,k_{i,h}}, \qquad u_j^{i,h}:=u_{j,k_{i,h}}^{i,h}, \qquad \rho^{i,h}_j:=\rho^{h}_{j, k_{i,h}},
\end{equation*}
     and
     \begin{equation*}
         w_j(x):= \begin{cases}
             w^{i,h}_{j, k_{i,h}}(x) &  \text{ for } \mu_{\ej}\text{-a.e. } x\in C^{i,h}_{j},\,  i\in Z_j(\Omega), h\in\{1,\dots,S_j\}, \\
             u_j(x) & \text{ otherwise}.
         \end{cases}
     \end{equation*}
     Indeed, \eqref{outcornici} is satisfied by definition and, since $\chi_{j, k_{i,h}}^{i,h}=1$ on $\Q(i\dej,\rho^{i,h}_{j}+\ej T)\setminus Q(i\dej,\rho^{i,h}_{j}-\ej T)$, \eqref{bordointerno} follows. Since $k\leq N-3$ by assumption, we have that
     \begin{equation*}
         \rj2^{Nh-k-1}-\ej/2>\rj 2^{N(h-1)+1}
     \end{equation*}
     for every $h$ provided that $\ej/\rj<1$; hence, for $j$ large and for fixed $i\in Z_j(\Omega)$, the sets $E_j^{i,h}, h\in\{1,\dots,S_j\},$ are pairwise disjoint. Moreover, by assumption \eqref{joining geo hyp}, we have that $E^{i,h}_j\subseteq Q(i\dej,\dej)$ for every $j,i,h$, and therefore,
     the sets $E_j^{i,h}$ are pairwise disjoint also when $i$ ranges in $Z_j(\Omega)$. Resorting to \eqref{outcornici} and \eqref{joining 2}, in virtue of the assumption \eqref{joining equibdd hyp}, we obtain
     \begin{align*}
          |\F^T_{\ej}(u_j)-\F^T_{\ej}(w_j)|& \leq \sum_{i\in Z_j(\Omega)}\sum_{h=1}^{S_j}  |\F^T_{\ej}(u_j,C_{j}^{i,h})-\F^T_{\ej}(w_{j},C_{j}^{i,h})| \\
         &  \leq \frac{C}{k}\sum_{i\in Z_j(\Omega)}\sum_{h=1}^{S_j}\F^T_{\ej}(u_j,E_j^{i,h}), \\
         & \leq \frac{C}{k} \F^T_{\ej}(u_j) \leq \frac{C}{k},
     \end{align*}
 which proves \eqref{diffenergia}.

Finally, we prove that $w_j\to u$ in $L^d(\Omega;\rr^m)$. Since $u_j\to u$, we prove that $u_j-w_j\to 0$. By \eqref{outcornici} and the definition of $w_j$, we have
\begin{align}\notag
 \lim_{j\to+\infty} \int_\Omega |u_j-w_j|^d\,dx & \leq \lim_{j\to+\infty} \sum_{i\in Z_j(\Omega)} \sum_{h=1}^{S_j} \int_{C_j^{i,h}} |u_j-u_j^{i,h}|^d\,dx \\ \label{joining 3}
  & = \lim_{j\to+\infty} \sum_{i\in Z_j(\Omega)} \sum_{h=1}^{S_j} \ej^d\int_{C_j^{i,h}} |u_j-u^{i,h}_j|^d\, d\mu_{\ej},
\end{align}
where, in the discrete case, the last equality follows identifying the function $u_j$ with a piecewise constant function and using that $\rj/\ej\to+\infty$.
Using \eqref{joining geo hyp} and \eqref{joining poincare bis}, we obtain
\begin{equation*}
     \ej^d\int_{C_j^{i,h}} |u_j-u_j^{i,h}|^d\,d \mu_{\ej}\leq C(\rho^{h}_{j,k_{i,h}})^d \F^T_{\ej}(u_j, \widetilde{E}_j^{i,h})  \leq C \dej^d \F^T_{\ej}(u_j, \widetilde{E}_j^{i,h}),
\end{equation*}
and then, summing over $i\in Z_j(\Omega)$ and $h\in\{1,\dots,S_j\}$ and substituting in \eqref{joining 3}, we infer
\begin{equation*}
     \lim_{j\to+\infty} \int_\Omega |u_j-w_j|^d\,dx \leq C \lim_{j\to+\infty}  \dej^d \F^T_{\ej}(u_j)=0
\end{equation*}
by \eqref{joining equibdd hyp}. This concludes the proof.    
\end{proof}

The following lemma allows us to obtain our {\it strange term} as the sum of the contributions to the total energy occurring near the perforations.

\begin{lemma}\label{lemma: integrale} Let $u_j\in \mathcal{A}_{\varepsilon_j}(\Omega;\mathbb{R}^{m}),j\in\NN,$ be a sequence such that $u_j\rightarrow u$ in $L^{d}(\Omega;\mathbb{R}^{m})$ as $j\to+\infty$ with $u \in W^{1,d}(\Omega;\mathbb{R}^{m})$, and let $T,k,N,\{S_j\}_j$, and $\{u_j^{i,S_j}\}_j$ be as in Lemma \ref{lemma: joining} for every $j$ large enough and for every $i\in Z_j(\Omega)$. Assume that
\begin{equation}\label{stima S_j}
   \rj2^{NS_j}<  \frac{\dej}{2}\leq\rj2^{N(S_j+1)}
\end{equation}
for every $j\in\NN$, that
\begin{equation}\label{hp: e/r bis}
    \lim_{j\to+\infty}\frac{\ej}{\rj}=0,
\end{equation}
and 
\begin{equation}
\label{bd_assumption}
\sup_{j}(\F^T_{\ej}(u_j)+\|u_j\|_{L^{\infty}(\Omega;\rr^m)})<+\infty.
\end{equation}
Then, letting $\f^T$ be defined by \eqref{limite_in_L}, it holds that 
    \begin{equation*}
    \lim_{j\to+\infty} \sum_{i\in Z_j(\Omega)} \dej^d \f^T(u_j^{i,S_j}) = \int_\Omega \f^T(u)\, dx.
\end{equation*}
\end{lemma}

\begin{proof}
   To simplify the notation we set $Q^i_j:=Q(i\dej,\delta_j)$. Using \eqref{equilip2} and \eqref{bd_assumption}, the following estimate holds:
    \begin{equation*}
        \begin{split}
           & \Bigl|\sum_{i \in Z_j(\Omega)}\delta_j^d\f^T(u^{i,S_j}_j)-\int_{\Omega}\varphi^T(u)\, dx\Bigr|\\
            &=\Bigl|\sum_{i \in Z_j(\Omega)}\int_{Q^i_j}(\f^T(u^{i,S_j}_j)-\f^T(u))\, dx-\int_{\Omega\setminus\bigcup_{i \in Z_j(\Omega)}Q^i_j}\f^T(u)\, dx\Bigr|\\
            &\leq\sum_{i \in Z_j(\Omega)}\int_{Q_j^i}|\f^T(u^{i,S_j}_j)-\f^T(u_j)|\, dx + \sum_{i \in Z_j(\Omega)}\int_{Q_j^i}| \f^T(u_j)-\f^T(u)|\, dx \\
            &\quad + \int_{\Omega\setminus\bigcup_{i \in Z_j(\Omega)}Q_j^i}\varphi^T(u)\, dx \\
            &\leq C\sum_{i \in Z_j(\Omega)}\int_{Q_j^i}|u^{i,S_j}_j-u_j|\, dx + C \sum_{i \in Z_j(\Omega)}\int_{Q_j^i}|u_j-u|\, dx \\
            &\quad + \int_{\Omega\setminus\bigcup_{i \in Z_j(\Omega)}Q_j^i}\varphi^T(u)\, dx  =: A^1_j+A^2_j+A^3_j.
        \end{split}
    \end{equation*}
We have 
    \begin{equation*}
        A_j^2\leq C\|u_j-u\|_{L^1(\Omega; \rr^m)},
    \end{equation*}
    which tends to $0$ as $j \to +\infty$ by assumption. Moreover, using \eqref{bd_assumption} and the fact that 
    \begin{equation*}
        \Bigl|\Omega\setminus \bigcup_{i \in Z_j(\Omega)} Q^i_j\Bigr|\to 0
    \end{equation*}
    as $j \to +\infty$, we infer that $A^3_j\to 0$ as well. 
    
    To establish the claim, it remains to prove that $A^1_j\to 0$ as $j \to +\infty$. Arguing as in the last part of the proof of Lemma \ref{lemma: joining}, we have \begin{equation*}
      \lim_{j\to+\infty} \sum_{i \in Z_j(\Omega)}\int_{Q_j^i}|u^{i,S_j}_j-u_j|\, dx =   \lim_{j\to+\infty}\sum_{i \in Z_j(\Omega)}\ej^d\int_{Q_j^i}|u^{i,S_j}_j-u_j|\, d\mu_{\ej}.
    \end{equation*}
    By H\"older's inequality we get
    \begin{align} \notag
            \sum_{i \in Z_j(\Omega)} \ej^d\int_{Q_j^i}|u^{i,S_j}_j-u_j|\, d\mu_{\ej} & \leq \varepsilon_j^d \sum_{i \in Z_j(\Omega)}(\mu_{\ej}(Q^i_j))^{1-\frac{1}{d}}\Bigl(\int_{Q_j^i}|u^{i,S_j}_j-u_j|^d\, d\mu_{\ej}\Bigr)^{\frac{1}{d}}\\ \notag
            & \leq \varepsilon_j^d\Bigl(\frac{\delta_j}{\varepsilon_j}\Bigr)^{d-1} \sum_{i \in Z_j(\Omega)} \Bigl(\int_{Q_j^i}|u^{i,S_j}_j-u_j|^d\, d\mu_{\ej}\Bigr)^{\frac{1}{d}}\\ \label{integrale 0}
            &=\delta_j^{d-1}\sum_{i \in Z_j(\Omega)} \Bigl(\ej^d\int_{Q^i_j}|u^{i,S_j}_j-u_j|^d\, d\mu_{\ej}\Bigr)^{\frac{1}{d}}.
    \end{align}
By the first inequality in \eqref{stima S_j} we have $C^{i,S_j}_j\subset Q_j^i$, where $C^{i,S_j}_j$ is as in \eqref{def_cornici}; then, recalling that
\begin{equation*}
    u_{j}^{i, S_j}= \frac{1}{\mu_{\ej}(C^{i,S_j}_j)}\int_{C^{i,S_j}_j} u_{j}\, d\mu_{\ej},
\end{equation*}
we apply the rescaled version of Poincaré-Wirtinger inequality in Proposition \ref{prop: poincare-wirtinger} with
\begin{equation*}
   \sigma=\ej, \qquad  \tau=\dej, \qquad x_0=i\dej, \qquad A=\frac{1}{\dej}(Q^i_j-i\dej)=Q(1),  \qquad A'=\frac{1}{\dej}(C^{i,S_j}_j-i\dej).
\end{equation*}
Note indeed that, given $\s_0$ as in Proposition \ref{prop: poincare-wirtinger}, we have $\ej<\sigma_0\dej$ for $j$ large enough since $\ej/\dej\to0$ as $j\to+\infty$ in light of \eqref{hp: e/r bis}. We obtain that
    \begin{align*}
      \ej^d\int_{Q^i_j}|u^{i,S_j}_j-u_j|^d\, d\mu_{\ej} & \leq C\delta_j^d \frac{\mu_{\ej}(Q^i_j)}{\mu_{\ej}(C^{i,S_j}_j)} G_{\varepsilon_j}(u_j,Q^i_j) \\
      & \leq C \dej^d \frac{\dej^d}{(2^d-1)(r_j 2^{NS_j-k_{i,S_j}-1})^d} G_{\varepsilon_j}(u_j,Q^i_j)\\
     & \leq C 2^{(2N-2)d}\dej^d G_{\varepsilon_j}(u_j,Q^i_j),
    \end{align*}
    where $C$ depends on $d$ and on the Poincaré-Wirtinger constant of $Q(1)$ but it is independent of $j$ and $i$, and we used that
    \begin{equation*}
        r_j 2^{NS_j-k_{i,S_j}-1}\geq \dej 2^{-2N+2}
    \end{equation*}
    in light of the second inequality in \eqref{stima S_j} and the fact that $0\leq k_{i,S_j}\leq N-4$. Hence, we obtain
    \begin{equation*}
        \ej^d\int_{Q^i_j}|u^{i,S_j}_j-u_j|^d\, d\mu_{\ej}\leq C\delta_j^d G_{\varepsilon_j}(u_j,Q^i_j),
    \end{equation*}
    with the constant $C$ independent of $j$. Substituting this inequality in \eqref{integrale 0} and using the concavity of the function $x \mapsto x^{\frac{1}{d}}$, \eqref{eq: riferimento lower} and \eqref{bd_assumption}, we conclude that
    \begin{equation*}
        \begin{split}
              \sum_{i \in Z_j(\Omega)} \ej^d\int_{Q_j^i}|u^{i,S_j}_j-u_j|\, d\mu_{\ej} &\leq C\delta_j^{d-1}\sum_{i \in Z_j(\Omega)}\delta_j(G_{\varepsilon_j}(u_j,Q^i_j))^{\frac{1}{d}} \\
            &            \leq C\delta^d_j(\#Z_j(\Omega))^{1-1/d}\Bigl(\sum_{i \in Z_j(\Omega)}G_{\varepsilon_j}(u_j,Q^i_j)\Bigr)^{\frac{1}{d}}\\
            &\leq C\delta_j(G_{\varepsilon_j}(u_j,\Omega))^{\frac{1}{d}}\\
            &\leq C\delta_{j}(\mathcal{F}^{T}_{\varepsilon_{j}}(u_{j}))^{\frac{1}{d}} \leq C\delta_{j},
            \end{split}
            \end{equation*}
            and we conclude that $A_j^1\to 0$ as $\delta_j\rightarrow 0$.
    \end{proof}

Finally, we recall a lemma that asserts that the asymptotic analysis can be reduced to sequences of functions that are equi-bounded in $L^\infty$, see \cite[Lemma 6.2]{AGL}  in the continuous case and \cite[Lemma 7.2]{Sig-dis} in the discrete case.
 
    \begin{lemma}\label{lemma : siga}
    Let $T>2$ and $u_j \in \mathcal{A}_{\varepsilon_j}(\Omega;\mathbb{R}^{m}),j\in\NN,$ be a sequence such that 
    \begin{equation*}
        \sup_{j}(\mathcal{F}^T_{\varepsilon}(u_{j})+\|u_{j}\|_{L^d(\Omega;\mathbb{R}^{m})})<+\infty.
    \end{equation*} 
    Then, for every $\eta >0$ and $L \in \mathbb{N}$, there exist a subsequence (not relabeled), a constant $R_{L}>L$ and a $1$-Lipschitz function $t_{L}:\mathbb{R}^{m}\rightarrow \mathbb{R}^{m}$ such that $t_{L}(z)=z$ if $|z|<R_L$,  $t_{L}(z)=0$ if $|z|>2R_{L}$,
    and
    \begin{equation*}
        \liminf_{j\to+\infty}\mathcal{F}^T_{\varepsilon_{j}}(t_{L}(u_{j}))\leq \liminf_{j\to+\infty}\mathcal{F}^T_{\varepsilon_{j}}(u_{j})+\eta.
    \end{equation*}
\end{lemma}

\subsection{Proof of Theorem \ref{thm: main discreto}}

For the sake of exposition, we divide the proof into three steps.

\smallskip

{\bf Step 1.} Let $T>2$ be a fixed natural. Let $u\in W^{1,d}(\Omega;\rr^m)$, $u_j\to u$ in $L^d(\Omega; \rr^m)$ and, without loss of generality, suppose that $\sup_j F^T_{\ej}(u_j)<+\infty$. We claim that
\begin{equation}\label{liminf T: claim}
    \liminf_{j\to+\infty} F_{\ej}^T(u_j)\geq \int_\Omega f_{\rm hom}^T(\nabla u)\, dx +\gamma \int_\Omega \f^T(u)\,dx,
\end{equation}
where $f_{\rm hom}^T$ and $\f^T$ are defined by \eqref{f^T_hom} and \eqref{limite_in_L}, respectively.

In a first instance, we further assume that 
\begin{equation}\label{ipotesi: succ limitata}
    \sup_j \|u_j\|_{L^\infty(\Omega;\rr^m)}<+\infty.
\end{equation}
Fix $k\geq2$ natural and $\eta\in(0,1)$. Applying Proposition \ref{proposition: psi limite}$(ii)$ with $L=2^{N+k}$, we have that
\begin{equation*}
    \frac{(\log 2^{N+k})^{d-1} \f^T_{2^{N+k}}}{\f^T}\to 1 \qquad \text{ as }N\to+\infty
\end{equation*}
uniformly on $\rr^m\setminus\{0\}$. Recalling that $\f^T_{2^{N+k}}(0)=\f^T(0)=0$, we deduce that there exists $N^*$ such that
\begin{equation}\label{scelta N liminf}
   (\log 2^N)^{d-1}\f^T_{2^{N+k}}(z)\geq(1-\eta)\f^T(z) \qquad \text{ if } N\geq N^*
\end{equation}
for every $z\in \rr^m$. Then, we fix $N>\max\{N^*,k+3\}$. Similarly, applying Proposition \ref{proposition: psi limite}$(i)$, with $\ell=1$ and $L=2^{N+k}$, we have
\begin{equation*}
    \frac{\f^T_{1,2^{N+k},\s}}{\f^T_{2^{N+k}}}\to1 \qquad \text{ as }\s\to0
\end{equation*}
uniformly on $\rr^m\setminus\{0\}$, and therefore there exists $\s^*<1/T$ such that
\begin{equation}\label{scelta j liminf}
    \f^T_{1,2^{N+k},\s}(z)\geq (1-\eta)\f^T_{2^{N+k}}(z) \qquad \text{ if } \s\leq \s^*
\end{equation}
for every $z\in \rr^m$. Since $k$ is fixed and $\ej/\rj\to0$, we assume that $2^k\ej/\rj<\s^*$ for every $j\in\NN$.

We apply Lemma \ref{lemma: joining} with $T,k,N$ fixed as above and a sequence of naturals $\{S_j\}_j$ such that \begin{equation}\label{scelta S liminf}
    \rj2^{NS_j}< \frac{\dej}{2}\leq\rj2^{N(S_j+1)}
\end{equation}
for every $j\in\NN$. Note that, upon considering $j$ large enough, condition \eqref{scelta S liminf} is compatible with the choice of the parameter $N$ since $\dej/\rj\to+\infty$ as $j\to+\infty$. We obtain functions $w_{j}\in \mathcal{A}_{\varepsilon_{j}}(\Omega;\mathbb{R}^{m}), j\in\NN,$  converging to $u$ in $L^d(\Omega; \rr^m)$ such that, for every $j$ large enough, we have
\begin{equation}\label{liminf-outcornici}
    w_{j}=u_{j} \qquad \mu_{\ej} \text{-a.e. on } \Omega\setminus \Bigl[ \bigcup_{\substack{i \in Z_{j}(\Omega)\\ h\in\{1,\dots, S_j\}}}C_{j}^{i,h}\Bigr],
\end{equation}
\begin{equation}\label{liminf-bordointerno}
    w_{j}=u_{j}^{i, h} \qquad \mu_{\ej} \text{-a.e. on }\overline{Q}(i\delta_j,\rho^{i,h}_{j}+\varepsilon_j T)\setminus Q(i\delta_j,\rho^{i, h}_{j}-\varepsilon_j T),
\end{equation}
\begin{equation}\label{liminf-diffenergia}
    F^T_{\ej}(u_j)\geq \F^T_{\ej}(w_j)- \frac{C}{k}
\end{equation}
for every $i\in Z_j(\Omega)$, $h\in\{1,\dots,S_j\}$, and some $C>0$ independent of $j$. 

We set 
\begin{equation*}
    E_j:= \bigcup_{i\in Z_j(\Omega)} Q(i\dej, \rho_j^{i, S_j})
\end{equation*}
and we first estimate from below the contribution to the energy of the region $\Omega\setminus E_j$; that is, far from the perforations that are strictly contained in $\Omega$.

\smallskip

\underline{{\it Analysis far from the perforations.}} We define $v_{j}\in \mathcal{A}_{\varepsilon_{j}}(\Omega;\mathbb{R}^{m}), j\in\NN,$ as
\begin{equation*}
    v_{j}(x):=
    \begin{cases}
     u^{i, S_j}_{j} & \text{ for } \mu_{\ej}\text{-a.e. } x\in Q(i\delta_{j},\rho^{i, S_j}_{j}), i \in Z_{j}(\Omega),\\
                w_{j}(x) &\text{otherwise},
            \end{cases}
        \end{equation*}
and we claim that $v_j\to u$ in $L^d(\Omega;\rr^m)$. To this end, we note that, by \eqref{liminf-bordointerno} and \eqref{liminf-diffenergia}, it holds \begin{equation*}
    \F^T_{\ej}(v_j)=\F^T_{\ej}(w_j,\Omega\setminus E_j)\leq F^T_{\ej}(u_j)+C,
\end{equation*}
so that $\sup_j \F^T_{\ej}(v_j)<+\infty$. Moreover, by \eqref{ipotesi: succ limitata} and the construction of Lemma \ref{lemma: joining}, we have
\begin{equation*}
    \sup_j\|v_j\|_{L^\infty(\Omega;\rr^m)}<+\infty;
\end{equation*}
and then, by applying Proposition \ref{prop: compactness}, we infer that a (not relabeled) subsequence of $\{v_j\}_j$ converges in $L^d(\Omega;\rr^m)$ to a certain $v$. Let $\chi$ denote the periodic extension of the characteristic function of the cube $Q(2/3)$. In virtue of \eqref{scelta S liminf}, we have
\begin{equation*}
\bigcup_{h=1}^{S_j}C^{i,h}_j\subset Q\Bigl(i\dej,\frac{\dej}{2}\Bigr)
\end{equation*}
for every $i\in Z_j(\Omega)$; hence, recalling \eqref{liminf-outcornici}, we get
\begin{equation}\label{near: identità}
    u_j(x)\Bigl(1-\chi\Bigl(\frac{x}{\dej}\Bigr)\Bigr)=v_j(x)\Bigl(1-\chi\Bigl(\frac{x}{\dej}\Bigr)\Bigr) \qquad \text{ for } \mathcal{L}^d\text{-a.e. } x\in\Omega,
\end{equation}
up to identifying $u_j$ and $v_j$ with piecewise constant functions in the discrete case. Since $\chi(\cdot/\dej)$ weakly$^*$-converges to the constant $(2/3)^{d}$ in $L^\infty(\Omega; \rr^m)$ and since $u_j\to u, v_j\to v$ in $L^d(\Omega; \rr^m)$, we pass to the limit in \eqref{near: identità} to infer that
$u=v$ $\mathcal{L}^d$-a.e. on $\Omega$.

Using now that $v_j\to u$ and resorting to Theorem \ref{theorem: Gamma unconstrained}, we obtain
        \begin{equation}\label{estimate far from the ps} 
\liminf_{j\to+\infty} \F^{T}_{\varepsilon_j}(w_j,\Omega\setminus E_j)=\liminf_{j\to+\infty}\mathcal{F}^{T}_{\varepsilon_{j}}(v_{j}) \geq \int_{\Omega}f^T_{\hom}(\nabla u)\, dx,
\end{equation}
which is the desired lower bound.
\smallskip 

\underline{{\it Analysis near the perforations.}} Now we perform the asymptotic analysis on $E_j$.

Fix $i\in Z_j(\Omega)$ and let \begin{equation*}
    w^{i,h}_j(x):=
    \begin{cases}
    0 & \text{ if }x \in \overline{Q}(\rho_j^{i,h}+\ej T), \\[2pt]
    w_j(x+i\dej)-u_j^{i,h} & \text{ if }  x \in Q(\rho^{i,h+1}_j-\ej T)\setminus \overline{Q}(\rho_j^{i,h}+\ej T), \\[3pt]
    u_j^{i,h+1}-u_j^{i,h} & \text{ if } x\in  Q(\rho^{i,h+1}_j-\ej T)^c,
    \end{cases}
\end{equation*}
for every $h\in\{1,\dots,S_j-1\}$, where the equality is valid for $\mu_{\ej}$-a.e. $x\in\rr^d$. To deal with the Dirichlet boundary condition imposed on the perforation, we let
\begin{equation*}
    w^{i,0}_j(x):=
    \begin{cases}
    0 & \text{ if } x\in \overline{Q}(\rj), \\[2pt]
    w_j(x+i\dej) & \text{ if }  x\in Q(\rho^{i,1}_j-\ej T)\setminus \overline{Q}(\rj), \\[3pt]
    u_j^{i,1} & \text{ if } x\in  Q(\rho^{i,1}_j-\ej T)^c,
    \end{cases}
\end{equation*}
 where the equality is valid for $\mu_{\ej}$-a.e. $x\in\rr^d$. We observe that $\rj2^{-k}+\ej T \leq\rj $ for $j$ large enough; hence, it also holds
\begin{equation*}
    w_j^{i,0}=
    \begin{cases}
    0 & \text{ for } \mu_{\ej}\text{-a.e. } x\in \overline{Q}(\rj2^{-k}+\ej T), \\[3pt]
    u_j^{i,1} & \text{ for } \mu_{\ej}\text{-a.e. }  x\in Q(\rho_j^{i,1}-\ej T)^c.
    \end{cases}
\end{equation*}
We have
\begin{align*} \notag
    \F^T_{\ej}(w_j, Q(i\delta_{j},\rho^{i, S_j}_{j}))& \geq  \F^T_{\ej}(w_j, Q(i\dej,\rho_j^{i,1})\setminus \overline{Q}(i\dej, \rj)) + \sum_{h=1}^{S_j-1}  \F^T_{\ej}(w_j, Q(i\dej,\rho_j^{i,h+1})\setminus \overline{Q}(i\dej, \rho_j^{i,h})) \\ \notag
    & = \F^T_{\ej}(w^{i,0}_j, Q(\rho_j^{i,1})\setminus \overline{Q}(\rj2^{-k})) + \sum_{h=1}^{S_j-1}  \F^T_{\ej}(w^{i,h}_j, Q(\rho_j^{i,h+1})\setminus \overline{Q}(\rho_j^{i,h})) \\ 
    & \geq \f^T_{\rj2^{-k}, \rho_j^{i,1},\ej}(u_j^{i,1}) + \sum_{h=1}^{S_j-1} \f^T_{\rho_j^{i,h}, \rho_j^{i,h+1},\ej}(u_j^{i,h+1}-u_j^{i,h}).
\end{align*} 
For every $j$ large enough, it holds
\begin{equation*}
    \rj 2^{Nh-k}\leq \rho^{i,h}_j\leq \rj2^{Nh}
\end{equation*}
for every $i\in Z_j(\Omega)$ and $h\in\{1,\dots, S_j\}$; hence, using Remark \ref{rmk: funzioni ausiliarie}(vi) we infer
\begin{equation*}
    \F^T_{\ej}(w_j, Q(i\delta_{j},\rho^{i, S_j}_{j}))\geq \sum_{h=0}^{S_j-1} \f^T_{\rj2^{Nh-k}, \rj2^{N(h+1)},\ej}(u_j^{i,h+1}-u_j^{i,h}),
\end{equation*}
where we let $u_j^{i,0}:=0$. Upon setting 
\begin{equation*}
    \sj^h:=\frac{\ej}{\rj}2^{-Nh+k}, \qquad j\in\NN, \ h\in\{0,\dots,S_j-1\},
\end{equation*}
we apply Remark \ref{rmk: funzioni ausiliarie}(v) to get 
\begin{equation*}
    \f^T_{\rj2^{Nh-k}, \rj2^{N(h+1)},\ej}(u_j^{i,h+1}-u_j^{i,h})= \f^T_{1, 2^{N+k},\sj^h}(u_j^{i,h+1}-u_j^{i,h})
\end{equation*}
for every $h\in\{0,\dots, S_j-1\}$. As a consequence, we obtain
\begin{align*}
      \F^T_{\ej}(w_j, Q(i\delta_{j},\rho^{i, S_j}_{j}))& \geq  \sum_{h=0}^{S_j-1} \f^T_{1, 2^{N+k},\sj^h}(u_j^{i,h+1}-u_j^{i,h}).
\end{align*}

We note that $\sj^h\leq 2^k\ej/\rj<\s^*$ for every $j\in\NN$ and $h\in\{0,\dots,S_j-1\}$. Then, using \eqref{scelta N liminf} and \eqref{scelta j liminf}, we get
\begin{align}\notag
 \F^T_{\ej}(w_j, Q(i\delta_{j},\rho^{i, S_j}_{j})) 
 & \geq  (1-\eta)\sum_{h=0}^{S_j-1} \f_{2^{N+k}}^T(u_j^{i,h+1}-u_j^{i,h})
\\ \label{liminf 1}
   & \geq (1-\eta)^2 (\log 2^N)^{1-d}\sum_{h=0}^{S_j-1}\f^T(u_j^{i,h+1}-u_j^{i,h}).
 \end{align}
Now, recalling that $u^{i,0}_j=0$, we observe that
\begin{equation*}
    \sum_{h=0}^{S_j-1} (u_j^{i,h+1}-u_j^{i,h})=u^{i,S_j}_j-u^{i,0}_j=u^{i,S_j}_j,
\end{equation*}
 and therefore, using that the function $\f^T$ is convex and $d$-homogeneous (see Remark \ref{rmk: funzioni ausiliarie}(iii),(iv)), we have
 \begin{equation}\label{liminf 2}
     \sum_{h=0}^{S_j-1}\f^T(u_j^{i,h+1}-u_j^{i,h}) \geq S_j \f^T\Bigl(\frac{u_j^{i,S_j}}{S_j}\Bigr)=S_j^{1-d}\f^T(u_j^{i,S_j}).
 \end{equation}
We substitute \eqref{liminf 2} in \eqref{liminf 1} and recall that $S_j\log 2^N\leq \log(\dej/(2\rj))$  by the first inequality in \eqref{scelta S liminf}. We obtain
 \begin{equation*}
 \F^T_{\ej}(w_j, Q(i\delta_{j},\rho^{i, S_j}_{j})) \geq (1-\eta)^2 (S_j\log 2^N)^{1-d}\f^T(u_j^{i,S_j}) \geq (1-\eta)^2 \Bigl(\log\Bigl(\frac{\dej}{2\rj}\Bigr)\Bigr)^{1-d}\f^T(u_j^{i,S_j}).
\end{equation*}

Summing over $i\in Z_j(\Omega)$ we get
\begin{align*}
\liminf_{j\to+\infty} \F^T_{\ej}(w_j, E_j) & = \liminf_{j\to+\infty} \sum_{i\in Z_j(\Omega)} \F^T_{\ej}(w_j, Q(i\delta_{j},\rho^{i, S_j}_{j})) \\
& \geq (1-\eta)^2\liminf_{j\to+\infty} \sum_{i\in Z_j(\Omega)}\Bigl(\log\Bigl(\frac{\dej}{2\rj}\Bigr)\Bigr)^{1-d}\f^T(u_j^{i,S_j}) \\
& =  (1-\eta)^2\liminf_{j\to+\infty} \sum_{i\in Z_j(\Omega)} \gamma\dej^d\f^T(u_j^{i,S_j}),
\end{align*}
where the last equality follows by \eqref{dej^d}. Since $\{u_j\}_j$ is bounded in $L^\infty(\Omega;\rr^m)$ by \eqref{ipotesi: succ limitata}, we apply Lemma \ref{lemma: integrale} and we infer
\begin{equation}\label{estimate vicino perf liminf}
    \liminf_{j\to+\infty} \F^T_{\ej}(w_j, E_j) \geq (1-\eta)^2\gamma \int_\Omega \f^T(u)\,dx.
\end{equation}

Substituting \eqref{estimate far from the ps} and \eqref{estimate vicino perf liminf} in \eqref{liminf-diffenergia}, we obtain
\begin{equation*}
    \liminf_{j\to+\infty} F_{\ej}^T(u_j)\geq \int_\Omega f^T_{\rm hom}(\nabla u)\, dx + (1-\eta)^2\gamma\int_\Omega \f^T(u)\,dx-\frac{C}{k}
\end{equation*}
and we infer \eqref{liminf T: claim} by letting first $\eta\to0$ and then $k\to+\infty$.

To conclude, we prove the claim removing the assumption \eqref{ipotesi: succ limitata}. Let $u\in W^{1,d}(\Omega;\rr^m)$, $u_j\to u$ in $L^d(\Omega; \rr^m)$ and suppose that $\sup_j \F^T_{\ej}(u_j)<+\infty$. By Lemma \ref{lemma : siga} and the lower bound we just proved for a sequence bounded in $L^\infty(\Omega;\rr^m)$, up to extracting a not relabeled subsequence we have
\begin{align*}
    \liminf_{j\to+\infty} \F^T_{\ej}(u_j) & \geq \liminf_{j\to+\infty} \F^T_{\ej}(t_L(u_j))-\eta \\
    & \geq \int_{\Omega} f^T_{\rm hom}(t_L(u))\,dx+\gamma\int_\Omega\f^T(t_L(u))\,dx-\eta, 
\end{align*}
 and then, letting $L\to+\infty$, we get
\begin{equation*}
    \liminf_{j\to+\infty} \F^T_{\ej}(u_j) \geq \int_{\Omega} f^T_{\rm hom}(u)\,dx+\gamma\int_\Omega\f^T(u)\,dx-\eta.
\end{equation*}
The conclusion follows by the arbitrariness of $\eta$.

\smallskip

{\bf Step 2.} Let $T>2$ be a fixed natural and let $u\in W^{1,d}(\Omega;\rr^m)$. We prove that there exists a sequence $u_j\to u$ in $L^d(\Omega; \rr^m)$ such that
\begin{equation}\label{claim limsup T}
     \limsup_{j\to+\infty} F^T_{\ej}(u_j)\leq \int_\Omega f^T_{\rm hom}(\nabla u)\,dx+\gamma\int_\Omega\f^T(u)\, dx.
\end{equation}

Fix $\widetilde{\Omega}$ a bounded open subset of $\rr^d$ with Lipschitz boundary such that $\overline{\Omega}\subset\widetilde{\Omega}$ and, by Theorem~\ref{theorem: Gamma unconstrained}, consider a sequence $\widetilde{u}_j\in \A_{\mu_{\ej}}(\widetilde{\Omega};\rr^m),j\in\NN,$ converging to $u$ in $L^d(\widetilde{\Omega};\rr^m)$ such that
\begin{equation}\label{recovery unconstrained}
    \lim_{j\to+\infty}\F^T_{\ej}(\widetilde{u}_j, \widetilde{\Omega})= \int_{\widetilde{\Omega}}f^T_{\rm hom}(\nabla u)\,dx.
\end{equation}
It is not restrictive to assume that $u$ is bounded and then, resorting to Lemma~\ref{lemma : siga}, we may also suppose that 
\begin{equation}\label{bound infty recovery}
\sup_j\|\widetilde{u}_j\|_{L^\infty(\widetilde{\Omega};\rr^m)}<+\infty.
\end{equation}

As in the first step, we fix $k\geq2$ natural and $\eta\in(0,1)$, and we use Proposition \ref{proposition: psi limite}$(ii)$ with $L=2^{N-k}$ to obtain that there exists $N^*$ such that
\begin{equation}\label{scelta N limsup}
   (\log 2^N)^{d-1}\f^T_{2^{N-k}}(z)\leq(1+\eta)\f^T(z) \qquad \text{ if } N\geq N^*
\end{equation}
for every $z\in \rr^m$. Then, we fix $N>\max\{N^*,k+3\}$. We apply  Proposition \ref{proposition: psi limite}$(i)$ with $\ell=1$ and $L=2^{N-k}$ to obtain that there exists $\s^*< 1/T$ such that
\begin{equation}\label{scelta j limsup}
    \f^T_{1,2^{N-k},\s}(z)\leq (1+\eta)\f^T_{2^{N-k}}(z) \qquad \text{ if } \s\leq \s^*
\end{equation}
for every $z\in \rr^m$, and we assume that $\ej/\rj<\s^*$ for every $j\in\NN$.

We apply Lemma \ref{lemma: joining} to the sequence $\{\widetilde{u}_j\}_j$ on $\widetilde{\Omega}$, with $T,k,N$ fixed as above and $\{S_j\}_j$ a sequence of naturals such that
\begin{equation}\label{scelta Sj limsup}
    \rj2^{NS_j}< \frac{\dej}{2}\leq \rj2^{N(S_j+1)}
\end{equation}
for every $j\in\NN$. We obtain a sequence $w_{j}\in \mathcal{A}_{\varepsilon_{j}}(\widetilde{\Omega};\mathbb{R}^{m}), j\in\NN,$ converging to $u$ in $L^{d}(\widetilde{\Omega};\mathbb{R}^{m})$ such that, for every $j$ large enough, the following hold:
\begin{equation*}
    w_{j}=\widetilde{u}_{j}  \qquad \mu_{\ej} \text{-a.e. on } \widetilde{\Omega} \setminus \Bigl[ \bigcup_{\substack{i \in Z_{j}(\widetilde{\Omega}) \\ h\in\{1,\dots, S_j\}}}C_{j}^{i, h}\Bigr],
\end{equation*}
\begin{equation}\label{bordointerno-limsup}
    w_{j}=\widetilde{u}_{j}^{i, S_j}  \qquad \mu_{\ej} \text{-a.e. on } \Q(i\delta_j,\rho^{i,S_j}_{j}+\varepsilon_j T)\setminus Q(i\delta_j,\rho^{i, S_j}_{j}-\varepsilon_j T),
\end{equation}
\begin{equation}\label{differenzaeneriga-limsup}
    \F_{\ej}^T(w_j,\widetilde{\Omega})\leq \F_{\ej}^T(\widetilde{u}_j,\widetilde{\Omega})+\frac{C}{k}
\end{equation}
for every $i\in Z_j(\widetilde{\Omega})$ and some $C>0$ independent of $j$.

In analogy with the first step, we let
\begin{equation*}
    E_j:=\bigcup_{i\in Z_j(\Omega)} Q(i\dej, \rho_j^{i, S_j})
\end{equation*}
and
\begin{equation*}
    \widetilde{E}_j:=\bigcup_{i\in Z_j(\widetilde{\Omega})} Q(i\dej, \rho_j^{i, S_j})
\end{equation*}
and we start by defining our recovery sequence on $\Omega\setminus \widetilde{E}_j$.

\smallskip

\underline{{\it Analysis far from the perforations.}} We set
\begin{equation*}
    u_j(x):=w_j(x) \qquad \text{ for }\mu_{\ej} \text{-a.e. } x\in\Omega\setminus \widetilde{E}_j,
\end{equation*}
and observe that by \eqref{differenzaeneriga-limsup} we have
\begin{equation*}
    \F^T_{\ej}(u_j, \Omega\setminus \widetilde{E}_j )\leq\F^T_{\ej}(w_j, \widetilde{\Omega}) \leq \F_{\ej}^T(\widetilde{u}_j, \widetilde{\Omega})+\frac{C}{k}.
\end{equation*}
Hence, by \eqref{recovery unconstrained} we obtain
\begin{equation}\label{claim limsup stima fuori}
\limsup_{j\to+\infty}  \F^T_{\ej}(u_j, \Omega\setminus \widetilde{E}_j) \leq \int_{\widetilde{\Omega}} f^T_{\rm hom}(\nabla u)\,dx + \frac{C}{k},
\end{equation}
which is the desired upper bound.

\smallskip

\underline{{\it Analysis near the perforations.}} Now we further modify the functions $\{w_j\}_j$ near the perforations that are far from $\partial \widetilde{\Omega}$. Recalling the definition of $Z_j(\widetilde{\Omega})$ in \eqref{insieme di indici}, upon assuming that $j$ is large enough, we have that
\begin{equation*}
\ZZ^d\cap \frac{1}{\dej}\Omega \subset Z_j(\widetilde{\Omega}),
\end{equation*} 
hence, this amounts to modifying the functions $\{w_j\}_j$ on all the perforations of $\Omega$, including those that intersect $\partial\Omega$.

For every $j\in\NN$, $i\in Z_j(\widetilde{\Omega})$, and $h\in\{0,\dots,S_j-2\}$, consider $\overline{v}^{i,h}_j\in\A_{\ej}(\rr^d;\rr^m)$ an admissible function for the minimum problem $\f^T_{\rj2^{Nh},\rj2^{N(h+1)}, \ej}(\widetilde{u}^{i,S_j}_j/S_j)$ such that
\begin{equation}\label{funzione ottimale limsup}
    \F_{\ej}^T(\overline{v}^{i,h}_j, Q(\rj 2^{N(h+1)})\setminus \Q(\rj 2^{Nh})) \leq (1+\eta)\f^T_{\rj2^{Nh},\rj2^{N(h+1)}, \ej}(\widetilde{u}^{i,S_j}_j/S_j).
\end{equation}
The case $h=S_j-1$ has to be treated separately taking into account that $\rho_j^{i,S_j}$ differs from $r_j2^{NS_j}$. In such a case, we consider $\overline{v}^{i,S_j-1}_j\in\A_{\ej}(\rr^d;\rr^m)$ an admissible function for the minimum problem $\f^T_{\rj2^{N(S_j-1)},\rho^{i,S_j}_j,\ej}(\widetilde{u}^{i,S_j}_j/S_j)$ such that
\begin{equation}\label{funzione ottimale limsup bis}
    \F_{\ej}^T(\overline{v}^{i,S_j-1}_j, Q(\rho_j^{i,S_j})\setminus \Q(\rj 2^{N(S_j-1)})) \leq (1+\eta)\f^T_{\rj2^{N(S_j-1)},\rho^{i,S_j}_j,\ej}(\widetilde{u}^{i,S_j}_j/S_j).
\end{equation}
For every $h\in\{0,\dots,S_j-1\}$ we let 
\begin{equation*}
v^{i,h}_j(x):=\overline{v}_j^{i,h}(x-i\dej)+\frac{h}{S_j}\widetilde{u}^{i,S_j}_j  \qquad \text{ for } \mu_{\ej} \text{-a.e. }x\in \rr^d,
\end{equation*}
and we note that, by the admissibility of $\overline{v}_j^{i,h}$, we have
\begin{equation*}
v^{i,h}_j=
     \begin{cases}
  \displaystyle\frac{h}{S_j}\widetilde{u}^{i,S_j}_j &  \mu_{\ej} \text{-a.e. on }\Q(i\dej,\rj 2^{Nh}+\ej T),\\[10pt]
   \displaystyle\frac{h+1}{S_j}\widetilde{u}^{i,S_j}_j & \mu_{\ej} \text{-a.e. on } [Q(i\dej, \rj2^{N(h+1)}-\ej T)]^c,
    \end{cases}
\end{equation*}
for every $h\in\{0,\dots,S_j-2\}$ and 
\begin{equation*}
v^{i,S_j-1}_j=
     \begin{cases}
    \displaystyle\frac{S_j-1}{S_j}\widetilde{u}^{i,S_j}_j & \mu_{\ej} \text{-a.e. on }\Q(i\dej,\rj 2^{N(S_j-1)}+\ej T), \\[10pt]
\widetilde{u}_j^{i,S_j} & \mu_{\ej} \text{-a.e. on } [Q(i\dej, \rho_j^{i,S_j}-\ej T)]^c.
    \end{cases}
\end{equation*}
For every $i\in Z_j(\widetilde{\Omega})$, we set
\begin{equation*}
    u_j(x):=\begin{cases}
    v_j^{i,0}(x) & \text{ if } x\in \Q(i\dej, \rj 2^N), \\[3pt]
    v_j^{i,h}(x) & \text{ if } x\in Q(i\dej,\rj 2^{N(h+1)})\setminus \Q(i\dej, \rj2^{Nh}),  h\in\{1,\dots,S_j-2\}, \\[3pt]
     v_j^{i,S_j-1}(x) & \text{ if } x\in Q(i\dej,\rho_j^{i,S_j})\setminus Q(i\dej, \rj2^{N(S_j-1)}),
    \end{cases}
\end{equation*}
where the equality is valid for $\mu_{\ej}$-a.e. $x\in Q(i\dej, \rho_j^{i,S_j})$, and we observe that
\begin{equation}\label{limsup DBC dentro}
    u_j(x)=   v^{i,0}_j(x)=0, \qquad  \text{ for } \mu_{\ej}\text{-a.e.} \ x\in \Q(i\dej, \rj).
\end{equation}

We now estimate the energy near the perforations that are far from $\partial \widetilde{\Omega}$. Recalling \eqref{funzione ottimale limsup} and \eqref{funzione ottimale limsup bis}, we get
\begin{align} \notag
    \F^T_{\ej}(u_j, Q(i\dej,\rho_j^{i,S_j})) & =  \sum_{h=0}^{S_j-2} \F^T_{\ej}(v_j^{i,h},  Q(i\dej,\rj 2^{N(h+1)})\setminus Q(i\dej,\rj 2^{Nh})) \\ \notag 
    & \quad+ \F^T_{\ej}(v_j^{i,S_j-1},  Q(i\dej,\rho_j^{i,S_j})\setminus Q(i\dej, \rj2^{N(S_j-1)})) \\ \notag
    & =  \sum_{h=0}^{S_j-2} \F_{\ej}^T(\overline{v}^{i,h}_j,  Q(\rj 2^{N(h+1)})\setminus Q(\rj 2^{Nh})) \\ \notag
    & \quad +  \F_{\ej}^T(\overline{v}^{i,S_j-1}_j, Q(\rho_j^{i,S_j})\setminus Q(\rj2^{N(S_j-1)}))\\ \notag
    &  \leq \sum_{h=0}^{S_j-2} (1+\eta)\f^T_{\rj2^{Nh},\rj2^{N(h+1)}, \ej}(\widetilde{u}^{i,S_j}_j/S_j) + (1+\eta)\f^T_{\rj2^{N(S_j-1)},\rho^{i,S_j}_j, \ej}(\widetilde{u}^{i,S_j}_j/S_j) \\ \label{stima limsup no riscalamento}
    & \leq (1+\eta)\sum_{h=0}^{S_j-1} \f^T_{\rj2^{Nh},r_j2^{N(h+1)-k}, \ej}(\widetilde{u}^{i,S_j}_j/S_j),
\end{align}
where the last inequality follows by Remark \ref{rmk: funzioni ausiliarie}(vi) and by the fact that, for $j$ large enough, 
\begin{equation*}
    \rho_j^{i,S_j} = \Bigl\lfloor\frac{3}{4}\frac{\rj}{\varepsilon_{j}}2^{NS_j-k_{i,S_j}}\Bigr\rfloor\varepsilon_{j} \geq \rj2^{NS_j-k}
\end{equation*}
for every $i$ since $k_{i,S_j}\leq k-1$. Upon setting 
\begin{equation*}
    \sj^h:=\frac{\ej}{\rj}2^{-Nh}, \qquad j\in\NN, h\in\{0,\dots S_j-1\},
\end{equation*}
Remark \ref{rmk: funzioni ausiliarie}(v) yields
\begin{equation*}
    \f^T_{\rj2^{Nh},\rj2^{N(h+1)-k}, \ej}(\widetilde{u}^{i,S_j}_j/S_j)=\f^T_{1,2^{N-k}, \sj^h}(\widetilde{u}^{i,S_j}_j/S_j),
\end{equation*}
and then \eqref{stima limsup no riscalamento} can be rewritten as
\begin{equation*}
    \F^T_{\ej}(u_j, Q(i\dej,\rho_j^{i,S_j})) \leq (1+\eta)\sum_{h=0}^{S_j-1}\f^T_{1,2^{N-k}, \sj^h}(\widetilde{u}^{i,S_j}_j/S_j).
\end{equation*}
As $\sj^h\leq\ej/\rj<\s^*$ for every $j\in\NN$ and $h\in\{0,\dots,S_j-1\}$, we use \eqref{scelta N limsup} and \eqref{scelta j limsup} and we employ the $d$-homogeneity of $\f^T$ to get 
\begin{align*}
    \F^T_{\ej}(u_j, Q(i\dej,\rho_j^{i,S_j})) & \leq (1+\eta)^2S_j \f^T_{2^{N-k}}(\widetilde{u}^{i,S_j}_j/S_j) \\
    & \leq  (1+\eta)^3S_j(\log 2^N)^{1-d}\f^T(\widetilde{u}^{i,S_j}_j/S_j) = (1+\eta)^3(S_j\log 2^N)^{1-d}\f^T(\widetilde{u}^{i,S_j}_j).
\end{align*}
By the second
inequality in \eqref{scelta Sj limsup}, we have
\begin{equation*}
    S_j\log 2^N \geq \log\Bigl(\frac{\dej}{2\rj}\Bigr)-\log 2^N
\end{equation*}
for every $j\in\NN$; hence, taking into account that $\dej/\rj\to+\infty$ as $j\to+\infty$ and that $N$ and $\eta$ are fixed, it is not restrictive to suppose that
\begin{equation*}
 S_j\log 2^N \geq (1+\eta)^{\frac{1}{1-d}}\log\Bigl(\frac{\dej}{2\rj}\Bigr)
\end{equation*}
for every $j\in\NN$. We obtain that
\begin{equation*}
    \F^T_{\ej}(u_j, Q(i\dej,\rho_j^{i,S_j})) \leq (1+\eta)^4\Bigr(\log\Bigl(\frac{\dej}{2\rj}\Bigr)\Bigl)^{1-d}\f^T(\widetilde{u}^{i,S_j}_j).
\end{equation*}

Finally, we sum over $i\in Z_j(\widetilde{\Omega})$. By \eqref{bound infty recovery}, we are in position to apply Lemma \ref{lemma: integrale} and then, also using \eqref{dej^d}, we get
\begin{align} \notag
    \limsup_{j\to+\infty}\F^T_{\ej}(u_j, \widetilde{E}_j) & = \limsup_{j\to+\infty}\sum_{i\in Z_j(\widetilde{\Omega})}\F^T_{\ej}(u_j, Q(i\dej,\rho_j^{i,S_j})) \\ \notag
    & \leq (1+\eta)^4\limsup_{j\to+\infty}\sum_{i\in Z_j(\widetilde{\Omega})} \Bigl(\log\Bigl(\frac{\dej}{2\rj}\Bigr)\Bigr)^{1-d}\f^T(\widetilde{u}^{i,S_j}_j) \\
\notag    & = (1+\eta)^4\limsup_{j\to+\infty} \sum_{i\in Z_j(\widetilde{\Omega})}\gamma\dej^d\f^T(\widetilde{u}^{i,S_j}_j)  \\ \label{claim limsup stima dentro}
    & = (1+\eta)^4 \gamma\int_{\widetilde{\Omega}}\f^T(u)\,dx.
    \end{align}    

Using \eqref{bordointerno-limsup} and combining \eqref{claim limsup stima fuori} and \eqref{claim limsup stima dentro}, we conclude that
\begin{align*}    \limsup_{j\to+\infty}\F^T_{\ej}(u_j) & = \limsup_{j\to+\infty}\, (\F^T_{\ej}(u_j, \Omega\setminus \widetilde{E}_j)+\F^T_{\ej}(u_j, \widetilde{E}_j))  \\
    &  \leq \int_{\widetilde{\Omega}} f^T_{\rm hom}(\nabla u)\, dx + (1+\eta)^4\gamma\int_{\widetilde{\Omega}}\f^T(u)\,dx+\frac{C}{k},
    \end{align*}
and \eqref{claim limsup T} follows by the arbitrariness of $\eta,k$ and $\widetilde{\Omega}$ and the fact that, in virtue of \eqref{limsup DBC dentro}, $\{u_j\}_j$ is an admissible sequence for $\{F^T_{\ej}\}_j$.

It remains to verify that $u_j\to u$ in $L^d(\Omega;\rr^m)$. By construction, we have
\begin{align*}
         \lim_{j\to+\infty}   \int_{\widetilde{\Omega}}|u_j-w_j|^{d}\,dx & = \lim_{j\to+\infty}   \int_{\widetilde{E}_j}|u_j-w_j|^{d}\,dx \\
                    &= \lim_{j\to+\infty} \sum_{i \in Z_j(\widetilde{\Omega})} \int_{Q(i\dej, \rho_j^{i,S_j})}|u_{j}-w_j|^{d}\,dx \\
                    & = \lim_{j\to+\infty} \sum_{i \in Z_j(\widetilde{\Omega})} \ej^d\int_{Q(i\dej, \rho_j^{i,S_j})}|u_{j}-w_j|^{d}\,d\mu_{\ej}.
\end{align*}
By \eqref{bordointerno-limsup} we have that $u_j-w_j=0\ \mu_{\ej}$-a.e. on $Q(i\dej,\rho_j^{i,S_j})\setminus Q(i\dej,\rho_j^{i,S_j}-\ej T)$ for every $i\in Z_j(\widetilde{\Omega})$. Therefore, we apply a rescaled version of Poincaré inequality in Proposition \ref{prop: pcre} with
\begin{equation*}
   \sigma=\ej, \qquad  \tau=\rho_j^{i,S_j}, \qquad x_0=i\dej, \qquad A=\frac{1}{\rho_j^{i,S_j}}(Q(i\delta_j,\rho_j^{i,S_j})-i\dej)=Q(1),
\end{equation*}
and we obtain
\begin{align*}
&\hspace{-0.8cm}\lim_{j\to+\infty}  \sum_{i \in Z_j(\widetilde{\Omega})}  \ej^d\int_{Q(i\dej, \rho_j^{i,S_j})}|u_{j}-w_j|^{d}\,d\mu_{\ej} \\
& \leq C\lim_{j\to+\infty}   \sum_{i \in Z_j(\widetilde{\Omega})}(\rho_j^{i,S_j})^dG_{\ej}(u_j-w_j, Q(i\dej, \rho_j^{i,S_j})) \\
    & \leq  C\lim_{j\to+\infty} \dej^d  \sum_{i \in Z_j(\widetilde{\Omega})}\bigl(G_{\ej}(u_j, Q(i\dej, \rho_j^{i,S_j}))+G_{\ej}(w_j, Q(i\dej, \rho_j^{i,S_j}))\bigr) \\
    & \leq  C \lim_{j\to+\infty}  \dej^d  (\F^T_{\ej}(u_j, \widetilde{E}_j)+\F^T_{\ej}(w_j, \widetilde{\Omega})) =0
\end{align*}
in light of \eqref{recovery unconstrained}, \eqref{differenzaeneriga-limsup}, and \eqref{claim limsup stima dentro}. Taking into account that $w_j\to u$ by Lemma \ref{lemma: joining}, we infer that $u_j\to u$.

\smallskip

{\bf Step 3.} In the first two steps we proved that, for every $T>2$ natural, it holds
\begin{equation*}
    \Gamma\text{-}\lim_{j\to+\infty} F^T_{\ej}(u) = \int_\Omega f^T_{\rm hom}(\nabla u)\,dx+\gamma\int_\Omega\f^T(u)\, dx, \qquad u\in W^{1,d}(\Omega;\rr^m).
\end{equation*}
Hence, resorting to Proposition \ref{prop_conv_fhom_uniforme},  Proposition \ref{proposition: gamma_troncamento}, and Proposition \ref{proposition: psi limite}$(iii)$, we obtain
\begin{equation*}
       \Gamma\text{-}\lim_{j\to+\infty} F_{\ej}(u) = \int_\Omega f_{\rm hom}(\nabla u)\,dx+\gamma\int_\Omega\f(u)\, dx,
\end{equation*}
which concludes the proof of Theorem \ref{thm: main discreto}.

\appendix

\section{Appendix}
\label{appendix}

In this appendix we prove the Poincaré-Wirtinger and the Poincaré inequality corresponding to Propositions \ref{prop: poincare-wirtinger} and \ref{prop: pcre}, with a particular focus on the discrete case. For the sake of completeness, we consider a slightly more general setting. For every $\s>0$, $A \subseteq \mathbb{R}^{d}$ a $\ms$-measurable set, and $p\in[1,+\infty)$, we let
\begin{equation*}
    \mathcal{A}_{\ms}(A;\mathbb{R}^{m}):=\{u: \mathbb{R}^d\rightarrow \mathbb{R}^{m}: u(x)=u(\alpha) \text{ if } x\in \alpha +[0,\s)^{d},\alpha \in A\cap \s\ZZ^d \}\cap L^p(A;\rr^m)
\end{equation*}
in the discrete case and 
\begin{equation*}
    \mathcal{A}_{\ms}(A;\mathbb{R}^{m}):=\{u:\rr^d\to\rr^m: u \text{ is }\mathcal{L}^d\text{-measurable}\}\cap L^p(A;\rr^m)
\end{equation*}
in the continuous case. We consider reference energies
\begin{equation*}
G^{p}_\s(u,A):=\int_{B(r_0)}\int_{A_{\ms}(\xi)}
|u(x+\s\xi)-u(x)|^p\,d\ms(x) \, d\mu_1(\xi),
\end{equation*} 
with $r_0\in(1,2)$ as in (G0) and $u \in \A_{\ms}(A ;\mathbb{R}^{m})$.  

In the discrete case, the following lemma allows one to control interactions of pairs of points in a possibly non-convex domain by means of interactions of pairs of points that are connected by a segment that lies in the domain.   

\begin{lemma}\label{lemma: cammini}
Let $A\subset \mathbb{R}^{d}$ be a connected set which is finite union of open rectangles $A=\cup_{h=1}^NA_h$ with $N\geq2$ and such that
\begin{equation*}
    A_h\cap A_{h+1}\cap \s\ZZ^d\neq \emptyset \qquad \text{ for every } h\in\{1,\dots,N-1\}.
\end{equation*}
There exist positive constants $\s_{0}$ and $C$ depending on $A_1,\dots,A_N$ such that it holds 
\begin{align}\label{eq: archi}
    \sum_{1\leq j<\ell\leq N}\sum_{\alpha\in (A_j\setminus A_\ell) \cap\s\ZZ^d}\sum_{\beta \in (A_\ell\setminus A_j)\cap\s\ZZ^d}|u(\alpha)-u(\beta)|^p & \leq CN^{p+1} \sum_{h=1}^{N}\sum_{\alpha\in A_h\cap \s \ZZ^d}\sum_{\beta\in A_h\cap \s \ZZ^d}|u(\alpha)-u(\beta)|^p
\end{align}
    for every $\s<\s_0$ and $u:A\cap\s \ZZ^d\rightarrow \mathbb{R}^{m}$.
\end{lemma}

\begin{proof}
We consider $Q_1,\dots, Q_{N-1}$ open cubes with sides parallel to the coordinate axes with the property that $Q_h\subseteq A_h\cap A_{h+1}$ for every $h\in\{1,\dots, N-1\}$,
\begin{equation*}
   |Q_1|=\dots= |Q_{N-1}|\geq \frac{1}{2}\min\{|A_h\cap A_{h+1}| :  h\in\{1,\dots, N-1\}\}, 
\end{equation*}
and 
\begin{equation*}
   C(\s):= \#(Q_1\cap \s\ZZ^d)=\dots = \#(Q_{N-1}\cap \s\ZZ^d)
\end{equation*}
for every $\s$ small enough. We label the elements of each set $Q_h\cap \s\ZZ^d$ as $\{q_h^1,\dots, q_h^{C(\s)}\}$.

We define surjective functions
\begin{equation*}
   I_{h}: A_h\cap\s\mathbb Z^d \to \{1,\dots, C(\s)\}, \qquad h\in\{1,\dots,N\},
\end{equation*}
such that
\begin{equation*}
    \#(I_h^{-1}(\{k\})) \leq \Bigl\lceil\frac{\#(A_{h}\cap \s \mathbb Z^d)}{C(\s)}\Bigr\rceil
\end{equation*}
for every $h\in\{1,\dots,N\}$ and $k\in\{1,\dots, C(\s)\}$, where we let $\lceil x\rceil$ denote the upper integer part of $x$. Upon assuming that $\s$ is sufficiently small, we can then suppose that
\begin{equation}\label{stima cardinalità}
    \#(I_h^{-1}(\{k\})) \leq 8|A_h|/\min\{|A_h\cap A_{h+1}| :  h\in\{1,\dots, N-1\}\} 
\end{equation}
for every $h\in\{1,\dots,N\}$ and $k\in\{1,\dots, C(\s)\}$. 

Let $1\leq j<\ell\leq N$ and consider points $\alpha\in (A_j\setminus A_\ell)\cap\s\ZZ^d$ and $\beta \in (A_\ell\setminus A_j)\cap\s\ZZ^d$. We define the intermediate nodes that shall allow us to connect these points. We set
 \begin{equation*}
     p_{j+k}:=\begin{cases}
         q_{j+k}^{I_\ell(\beta)} & \text{ if } k \text{ is even,}\\[5pt]
         q_{j+k}^{I_j(\alpha)} & \text{ if } k \text{ is odd,}
     \end{cases}
     \qquad k\in\{0,\dots,\ell-j-1\}.
 \end{equation*}
For sake of notation, it is convenient to set 
 \begin{equation*}
n(j,\ell):= \begin{cases}
        \ell & \text{ if } \ell-j \text{ is even,} \\
     \ell+1 & \text{ if } \ell-j \text{ is odd,}
    \end{cases}   
\end{equation*}
and 
\begin{equation*}
    p_{j-1}:=\alpha, \qquad p_{n(j,\ell)}:=\beta.
\end{equation*}
Now we distinguish two cases:
\begin{itemize}
 \item if $\ell-j$ is even, we have that $p_{\ell-1}=q_{\ell-1}^{I_j(\alpha)}$ and we define
 \begin{equation*}
     \mathcal P(\alpha,\beta):=
        \{\alpha=p_{j-1}, p_j, \dots, p_{\ell-1}, p_\ell=p_{n(j,\ell)}=\beta\}, 
 \end{equation*}
\item if $\ell-j$ is odd, we have that $p_{\ell-1}=q_{\ell-1}^{I_\ell(\beta)}$, then we set $p_\ell:=q_{\ell-1}^{I_j(\alpha)}$ and we define
\begin{equation*}
    \mathcal P(\alpha,\beta):=\{\alpha=p_{j-1}, p_j, \dots, p_{\ell}, p_{\ell+1}=p_{n(j,\ell)}=\beta\} .
\end{equation*}   
 \end{itemize}
Finally, we connect $\alpha$ and $\beta$ using the path obtained connecting through segments the points of $\mathcal P(\alpha,\beta)$ in the established order.

If $\ell-j$ is even, the path associated with $\mathcal P(\alpha,\beta)$ is made up of at most $\ell-j+1$ segments and 
\begin{equation*}
 [p_h,p_{h+1}]\subset A_{h+1} \quad \text{ for every } h\in\{j-1,\dots, \ell-1\}.
\end{equation*}

If $\ell-j$ is odd, the path associated with $\mathcal P(\alpha,\beta)$ is made up of at most $\ell-j+2$ segments and 
\begin{equation*}
  [p_h,p_{h+1}]\subset A_{h+1} \quad\text{ for every } h\in\{j-1,\dots, \ell-1\}, \qquad [p_{\ell},p_{\ell+1}]\subset A_\ell.
\end{equation*}

Let us consider a segment $[\lambda,\eta]\subset A_{h}$ for some $h\in\{1,\dots,N\}$. We let $S(\lambda,\eta)$ denote the set of pairs $(\alpha,\beta)\in(A_j\setminus A_\ell)\cap\s\ZZ^d\times (A_\ell\setminus A_j)\cap\s\ZZ^d, 1\leq j<\ell\leq N,$ such that $\lambda$ and $\eta$ are two consecutive points belonging to the path associated with $\mathcal{P}(\alpha,\beta)$. 

To estimate the cardinality of $S(\lambda,\eta)$, let us assume that $\ell-j$ is even and consider $(\alpha,\beta)\in S(\lambda,\eta)$.  Upon exchanging the roles of $\lambda$ and $\eta$, we have the following possibilities:
\begin{itemize}
    \item[(i)] $\lambda=\alpha=p_{j-1}, \eta=p_j$,
     \item[(ii)] $  \lambda=p_{\ell-1}, \eta=\beta=p_\ell$,
      \item[(iii)] $\lambda=p_m,  \eta=p_{m+1}$ for some $m\in\{j,\dots,\ell-2\}$.
\end{itemize}
In the first case $\alpha$ is uniquely determined and $\eta\in Q_j$, so that $\eta=q_j^{k}$ for a certain $k\in\{1,\dots, C(\s)\}$. Hence, we have that 
\begin{equation*}
q_j^k=\eta=p_j=q^{I_\ell(\beta)}_j.    
\end{equation*}
Therefore, $\beta\in I_\ell^{-1}(\{k\})$ and, by \eqref{stima cardinalità}, the number of possible pairs $(\alpha,\beta)$ is less than or equal to 
\begin{equation*}
    \#(I_\ell^{-1}(\{k\})) \leq 8|A_\ell|/\min\{|A_h\cap A_{h+1}| :  h\in\{1,\dots, N-1\}\}.
\end{equation*}
In the second case $\beta$ is uniquely determined and $\lambda\in Q_{\ell-1}$, so that $\lambda=q_{\ell-1}^{k}$ for a certain $k\in\{1,\dots, C(\s)\}$. Hence, we have that 
\begin{equation*}
    q_{\ell-1}^{k}= \lambda=p_{\ell-1}=q^{I_j(\alpha)}_{\ell-1}.
\end{equation*}
Therefore, $\alpha\in I_j^{-1}(\{k\})$ and the number of possible pairs $(\alpha,\beta)$ is less than or equal to 
\begin{equation*}
    \#(I_j^{-1}(\{k\})) \leq  8|A_j|/\min\{|A_h\cap A_{h+1}| :  h\in\{1,\dots, N-1\}\}.
\end{equation*}
In the third case, since $\lambda\in Q_m$ and $\eta\in Q_{m+1}$, we have that $\lambda=q_m^k$ and $\eta=q_{m+1}^{k'}$ for some $k,k'\in\{1,\dots, C(\s)\}$, and then
\begin{equation*}
  \begin{cases}
      q_m^k=q_m^{I_\ell(\beta)}, \\[3pt]
      q_{m+1}^{k'}=q_{m+1}^{I_j(\alpha)},
  \end{cases}\qquad \text{ or } \qquad  
  \begin{cases}
q_m^k=q_{m}^{I_j(\alpha)}, \\[3pt]
q_{m+1}^{k'}=q_{m+1}^{I_\ell(\beta)}.
  \end{cases}
\end{equation*}
Therefore,
\begin{equation*}
\begin{cases}
   \alpha\in I_j^{-1}(\{k'\}), \\
   \beta \in I_\ell^{-1}(\{k\}), 
\end{cases}
    \qquad \text{ or } \qquad
    \begin{cases}
        \alpha\in I_j^{-1}(\{k\}),\\ \beta \in I_\ell^{-1}(\{k'\}),
    \end{cases} 
\end{equation*}
and the number of possible pairs $(\alpha,\beta)$ in this case is less than or equal to 
\begin{align*}
   & \max\{\#(I_j^{-1}(\{k'\}))\times \#(I_\ell^{-1}(\{k\})), \#(I_j^{-1}(\{k\}))\times \#(I_\ell^{-1}(\{k'\}))\} \\ &\leq64|A_j||A_\ell|/\min\{|A_h\cap A_{h+1}| :  h\in\{1,\dots, N-1\}\}^2.
\end{align*}

The case $\ell-j$ odd is analogous, provided that the additional segment $[p_{\ell-1},p_\ell]$ is also considered.
For this segment, upon exchanging the roles of $\lambda$ and $\eta$, one has
\[
\lambda=p_{\ell-1}=q_{\ell-1}^{I_\ell(\beta)},\qquad
\eta=p_\ell=q_{\ell-1}^{I_j(\alpha)}.
\]
Therefore, the number of admissible pairs $(\alpha,\beta)$ is bounded by
\[
\#I_\ell^{-1}(\{k\})\times\# I_j^{-1}(\{k'\})
\leq
64|A_j||A_\ell|/
\min\{|A_h\cap A_{h+1}|:h \in \{1,\dots,N-1\}\}^2,
\]
which is the same estimate obtained for (iii) in the case $\ell-j$ is even. All other segments are treated as in the even case.

Letting $j,\ell$ vary in $\{1,\dots, N\}$, we obtain that
\begin{align} \notag
    \#S(\lambda,\eta) & \leq 2(64+8) \sum_{1\leq j<\ell\leq N} |A_j||A_\ell|/\min\{|A_h\cap A_{h+1}| :  h\in\{1,\dots, N-1\}\}^2 \\ \label{stima archi}
    & \leq N^2 144 \Bigl(\frac{\max\{|A_h|: h\in\{1,\dots, N\}\}}{\min\{|A_h\cap A_{h+1}| :  h\in\{1,\dots, N-1\}\}}\Bigr)^2=:N^2C.
\end{align}

By Jensen's inequality we have 
\begin{align*}
    |u(\alpha)-u(\beta)|^p &= \Bigl|\sum_{h=j-1}^{n(j,\ell)-1}u(p_{h+1})-u(p_h)\Bigr|^p \\
    &\leq (n(j,\ell)-j+1)^{p-1} \sum_{h=j-1}^{n(j,\ell)-1}|u(p_{h+1})-u(p_h)|^p \leq (2N)^{p-1}\sum_{h=j-1}^{n(j,\ell)-1}|u(p_{h+1})-u(p_h)|^p.
\end{align*}
Recalling that $[p_h,p_{h+1}]\subset A_{h+1}$ for every $h\in\{j-1,\dots, \ell-1\}$ and $[p_{n(j,\ell)-1},p_{n(j,\ell)}]\subset A_\ell$, and using \eqref{stima archi}, we sum over $\alpha$ and $\beta$ and $1\leq j<\ell\leq N$ to get 
\begin{multline*}
     \sum_{1\leq j<\ell\leq N}\sum_{\alpha\in (A_j\setminus A_\ell) \cap\s\ZZ^d}\sum_{\beta \in (A_\ell\setminus A_j)\cap\s\ZZ^d}|u(\alpha)-u(\beta)|^p \\ \leq 2^{p-1}N^{p+1}C \sum_{h=1}^{N}\sum_{\lambda\in A_h\cap \s \ZZ^d}\sum_{\eta\in A_h\cap \s \ZZ^d}|u(\lambda)-u(\eta)|^p,
\end{multline*}
which concludes the proof.
\end{proof}

\begin{remark}
    The constant in \eqref{eq: archi} is far from optimal as the paths connecting pairs of points are not constructed following the geodesics inside $A$. Nevertheless, we note that such a constant is invariant under rescaling of $A$ as it depends on the number of rectangles $A_1,\dots,A_N$ that constitute $A$ and on
    \begin{equation*}
        \frac{\max\{|A_h|: h\in\{1,\dots, N\}\}}{\min\{|A_h\cap A_{h+1}| :  h\in\{1,\dots, N-1\}\}},
    \end{equation*}
    which are unaffected by homotheties of $A$.
\end{remark}

The following result is a more general version of Proposition \ref{prop: poincare-wirtinger}.

\begin{proposition}[Poincaré-Wirtinger inequality on union of rectangles]
\label{Poincaré-Wirtinger inequality on union of rectangles}
Let $A\subset \mathbb{R}^{d}$ be a connected set which is finite union of open rectangles with sides parallel to the coordinate axes $A=\cup_{h=1}^NA_h$ such that, if $N\geq 2$, it holds 
\begin{equation*}
    A_h\cap A_{h+1}\cap \s\ZZ^d\neq \emptyset \qquad \text{ for every } h\in\{1,\dots,N-1\};
\end{equation*}
let $A'\subseteq A$ be a $\ms$-measurable set and $x_0\in \rr^d$. There exist positive constants $\s_{0}$ and $C$ depending on $A, A_1,\dots, A_N, N$ such that, having set
\begin{equation*}
    (u)^{\ms}_E:=\frac{1}{\ms(E)}\int_{E} u\,d\ms
\end{equation*}
for every $E$ $\ms$-measurable set with $\ms(E)>0$, it holds 
\begin{equation}\label{PW-app}
    \int_{x_0+\tau A} |u(x)-(u)^{\ms}_{x_0+\tau A'}|^p\,d\ms(x)\leq C\frac{\ms(x_0+\tau A)}{\ms(x_0+\tau A')}\Bigl(\frac{\tau}{\s}\Bigr)^p G^{p}_\s(u,x_0+\tau A)
\end{equation}
    for every $\tau>0$, $\s\in(0,\s_0\tau)$, and $u \in \mathcal{A}_{\mu_\sigma}(x_0+\tau A;\mathbb{R}^m)$.  
    \end{proposition}
    
\begin{proof}
    It suffices to prove the statement in the case $x_0=0$ and $\tau=1$ as the general case follows by a translation and scaling argument. Once the statement is proved in the case $A'=A$, the general case follows observing that
    \begin{equation*}
        |(u)_A^{\mu_\s}-(u)_{A'}^{\mu_\s}|^p=\Bigl|\frac{1}{\mu_\s(A')}\int_{A'}(u)_A^{\mu_\s}-u\, d\mu_\s\Bigr|^p
        \leq \frac{1}{\mu_\s(A')}\int_{A'}|(u)_A^{\mu_\sigma}-u|^p\, d\mu_\sigma,
    \end{equation*}
    so that \eqref{PW-app} applied with $A=A'$, $x_0=0$, and $\tau=1$ yield
    \begin{equation*}
        \begin{split}
            \int_A|u(x)-(u)_{A'}^{\mu_\sigma}|^p\, d\mu_\s(x)&\leq 2^{p-1}\Bigl[\int_A|u(x)-(u)_{A}^{\mu_\sigma}|^p\, d\mu_\s(x)+\int_A|(u)_A^{\mu_\sigma}-(u)_{A'}^{\mu_\sigma}|^p\, d\mu_\s(x)\Bigr] \\
            &\leq 2^{p-1}\Bigl[\int_A|u(x)-(u)_{A}^{\mu_\sigma}|^p\, d\mu_\s(x)+\frac{\mu_\s(A)}{\mu_\sigma(A')}\int_{A'}|u(x)-(u)_A^{\mu_\sigma}|^p\, d\mu_\sigma(x)
        \Bigr]\\
        &\leq 2^p\frac{\mu_\s(A)}{\mu_\sigma(A')}\int_{A}|u(x)-(u)_A^{\mu_\sigma}|^d\, d\mu_\sigma(x)\leq C\frac{\mu_\s(A)}{\mu_\sigma(A')}\frac{1}{\s^p}G^{p}_\s(u,A).
        \end{split}
    \end{equation*}
We then prove the statement further assuming that $A'=A$. 

In the continuous setting, the proof can be found in \cite[Proposition 4.2]{AABPT} and, in fact, the result holds for general bounded open sets with Lipschitz boundary. In the discrete case we have to prove that 
 \begin{equation*}
 \sum_{\alpha\in A\cap \s\ZZ^d}|u(\alpha)-(u)_A|^p \leq  C \frac{1}{\s^p}\sum_{k=1}^d \sum_{\alpha\in R_\s^{e_k}(A)}|u(\alpha+\s e_k)-u(\alpha)|^p,
    \end{equation*}
where
\begin{equation*}
    (u)_A=\frac{1}{\#(A\cap \s\ZZ^d)}\sum_{\alpha\in A\cap \s\ZZ^d} u(\alpha)
\end{equation*}
and $R_\s^{e_k}(A):=\{\alpha\in  A\cap\s\ZZ^d: [\alpha,\alpha+\s e_k]\subset A\}$ for every $k\in\{1,\dots, d\}$. 

We have
    \begin{align}\notag
       \sum_{\alpha\in  A\cap\s\ZZ^d}|u(\alpha)-(u)_A|^p & =  \sum_{\alpha\in A\cap\s\ZZ^d}\Bigl|u(\alpha)-\frac{1}{\# (A\cap\s\ZZ^d)}\sum_{\beta \in A\cap\s\ZZ^d}u(\beta)\Bigr|^p \\ \notag
       & = \sum_{\alpha\in A\cap\s\ZZ^d}\Bigl|\frac{1}{\# (A\cap\s\ZZ^d)}\sum_{\beta \in A\cap\s\ZZ^d}(u(\alpha)-u(\beta))\Bigr|^p \\ \label{poincare 1}
       & \leq \frac{1}{\# (A\cap\s\ZZ^d)}\sum_{\alpha\in A\cap\s\ZZ^d}\sum_{\beta \in A\cap\s\ZZ^d}|u(\alpha)-u(\beta)|^p.  
    \end{align}

Suppose $N=1$; i.e., $A$ is a rectangle with sides parallel to the coordinate axes and, without loss of generality, further assume that it is centered at the origin. Slightly adapting \cite[Lemma 3.6]{AliCic}, we get
\begin{align} \notag
    \sum_{\alpha\in A\cap\s\ZZ^d}\sum_{\beta \in A\cap\s\ZZ^d}|u(\alpha)-u(\beta)|^p
    & \leq \sum_{\xi\in \frac{2}{\s}A \cap\ZZ^d} \sum_{\alpha\in R_\s^{\xi}(A)}|u(\alpha+\s\xi)-u(\alpha)|^p \\ \notag
    & \leq C \sum_{\xi\in \frac{2}{\s}A \cap\ZZ^d} |\xi|^p\sum_{k=1}^d\sum_{\alpha\in R_\s^{e_k}(A)}|u(\alpha+\s e_k)-u(\alpha)|^p \\ \label{poincare 2}
    & \leq  C\#(2 A\cap \s\ZZ^d) \Bigl(\frac{\text{diam } A}{\s}\Bigr)^p\sum_{k=1}^d\sum_{\alpha\in R_\s^{e_k}(A)}|u(\alpha+\s e_k)-u(\alpha)|^p.
\end{align}
Recalling \eqref{poincare 1} we infer
\begin{align*}
   \sum_{\alpha\in  A\cap\s\ZZ^d}|u(\alpha)-(u)_A|^p 
    & \leq  C\frac{\#(2 A\cap \s\ZZ^d) }{\# (A\cap\s\ZZ^d)}\Bigl(\frac{\text{diam } A}{\s}\Bigr)^p\sum_{k=1}^d\sum_{\alpha\in R_\s^{e_k}(A)}|u(\alpha+\s e_k)-u(\alpha)|^p \\
    & \leq C 2^{d+1} (\text{diam } A)^p\frac{1}{\s^p}\sum_{k=1}^d\sum_{\alpha\in R_\s^{e_k}(A)}|u(\alpha+\s e_k)-u(\alpha)|^p,
\end{align*}
where we used that $\#(2 A_h\cap \s\ZZ^d)\leq 2^{d+1}\#(A\cap \s\ZZ^d)$ for $\s$ sufficiently small. 

If $N\geq 2$, by Lemma \ref{lemma: cammini} we have
\begin{equation*}
      \sum_{\alpha\in A\cap\s\ZZ^d}\sum_{\beta \in A\cap\s\ZZ^d}|u(\alpha)-u(\beta)|^p\leq (1+CN^{p+1})\sum_{h=1}^N\sum_{\alpha\in A_h\cap\s\ZZ^d}\sum_{\beta \in A_h\cap\s\ZZ^d}|u(\alpha)-u(\beta)|^p
\end{equation*}
and then, arguing as in \eqref{poincare 2} for fixed $h\in\{1,\dots,N\}$, we get
\begin{align*}
 &\sum_{\alpha\in A\cap\s\ZZ^d}\sum_{\beta \in A\cap\s\ZZ^d}|u(\alpha)-u(\beta)|^p \\ &\leq C(1+CN^{p+1})\sum_{h=1}^N
    \#(2 A_h\cap \s\ZZ^d) \Bigl(\frac{\text{diam } A_h}{\s}\Bigr)^p\sum_{k=1}^d\sum_{\alpha\in R_\s^{e_k}(A_h)}|u(\alpha+\s e_k)-u(\alpha)|^p \\
    & \leq CN(1+CN^{p+1})\#(2 A\cap \s\ZZ^d) \Bigl(\frac{\text{diam } A}{\s}\Bigr)^p\sum_{k=1}^d\sum_{\alpha\in R_\s^{e_k}(A)}|u(\alpha+\s e_k)-u(\alpha)|^p.
\end{align*}
Therefore, recalling \eqref{poincare 1}, we conclude
\begin{equation*}
        \sum_{\alpha\in  A\cap\s\ZZ^d}|u(\alpha)-(u)_A|^p \leq CN(1+CN^{p+1})2^{d+1}(\text{diam } A)^p \frac{1}{\s^p}\sum_{k=1}^d\sum_{\alpha\in R_\e^{e_k}(A)}|u(\alpha+\s e_k)-u(\alpha)|^p,
    \end{equation*}
for $\s$ sufficiently small. This proves the thesis for the case $A'=A$ in the discrete setting and concludes the proof.
\end{proof}

\begin{remark}
   In the discrete case, the constant in inequality \eqref{PW-app} depends on a universal constant coming from \cite[Lemma 3.6]{AliCic}, the scaling-invariant constant of inequality \eqref{eq: archi}, the number of rectangles that constitute $A$ and $(\text{diam }A)^p$. Hence, such a constant exhibits the correct behaviour under homothety of the domain, in contrast with \cite[Lemma 2 and Lemma 3]{BraSig}. In particular, if we let $C(A)$ denote the constant associated with $A$, we have that $C(\tau A)=\tau^p C(A)$.
\end{remark}

We conclude this appendix with the proof of Proposition \ref{prop: pcre}, the rescaled version of Poincaré inequality. 

\begin{proposition}
[Discrete Poincaré inequality]
\label{Appendix Discrete Poincare}
 Let $T>1, x_0\in \rr^d,$ and let $A\subset \mathbb{R}^d$ be a bounded open set. There exists a positive constant $C$ depending on $A$ such that, having set
\begin{equation*}
    R_\s^{e_k}(E):=\{\alpha\in E\cap \s\ZZ^d: [\alpha,\alpha+\s e_k]\subset E\}
\end{equation*}
for every set $E$ such that $\#(E\cap \s\ZZ^d)\neq0$ and $k\in\{1,\dots, d\}$, it holds
 \begin{equation*}
    \sum_{\alpha \in (x_0+\tau A) \cap \s\mathbb{Z}^d}|u(\alpha)|^p\leq C\Big(\frac{\tau}{\s}\Big)^p\sum_{k=1}^d\sum_{\alpha\in R_\s^{e_k}(x_0+\tau A)}|u(\alpha+\s e_k)-u(\alpha)|^p
\end{equation*}
for every $\tau,\s>0$ and for every $u:(x_0+\tau A)\cap \s\mathbb{Z}^d\to \mathbb{R}^m$ such that $u(\alpha)=0$ if  \mbox{$\text{dist}_\infty(\alpha ,(x_0+\tau A)^c)\leq \s T$.}
\end{proposition}
\begin{proof}
We prove the result for $x_0=0$ and $\tau=1$ since the general case follows by a translation and scaling argument. We argue associating to any admissible function $u$ a piecewise-affine function obtained by linearly interpolating the values of $u$ on a triangulation of the lattice $\s\mathbb{Z}^{d}$. For the sake of simplicity, we prove in detail the case $d=2$ and we remark that the same argument can be generalized to any dimension resorting to the Kuhn's decomposition of a cube \cite{Kuhn1960} (see for example \cite{ALP}). 

We identify $u$ with its extension to the whole $\s\mathbb{Z}^2$ that equals zero outside $A\cap \s\ZZ^2$. We set
    \begin{equation*}
        T^{-}:=\{(x_{1},x_{2})\in [0,1]^{2}: x_{2}\leq 1-x_{1}\},
    \end{equation*}
    \begin{equation*}
         T^{+}:=\{(x_{1},x_{2})\in [0,1]^{2}: x_{2}\geq 1-x_{1}\},
    \end{equation*}
    and, given any admissible function $u: \s \ZZ^2\to \rr^m$, we define $\hat{u}:\rr^2\to \rr^m$ as 
    \begin{equation*}
    \hat{u}(x):=u(\alpha)+\frac{u(\alpha+\s e_{1})-u(\alpha)}{\s}(x_{1}-\alpha_{1})+\frac{u(\alpha+\s e_{2})-u(\alpha)}{\s}(x_{2}-\alpha_{2})\ \quad \text{if}\ x \in \alpha +\s T^{-},
    \end{equation*}
    \begin{equation*}
    \begin{split}
        \hat{u}(x):=u(\alpha+\s(e_1+e_2))&+\frac{u(\alpha+\s(e_1+e_2))-u(\alpha+\s e_2)}{\s}(x_1-\alpha_1-\s)\\
&+\frac{u(\alpha+\s(e_1+e_2))-u(\alpha+\s e_1)}{\s}(x_2-\alpha_2-\s)\ \quad \text{if}\  x\in \alpha +\s T^{+}
    \end{split}
    \end{equation*}
for every $\alpha \in \s \ZZ^2$. By the boundary condition on $u$ we infer that  $\hat{u} \in W^{1,p}_0(A;\mathbb{R}^m)$. Hence the standard Poincaré inequality holds and we have
    \begin{equation}
    \label{strd_pcre}
        \int_A |\hat u(x)|^p\,dx
\leq
C\int_A |\nabla \hat u(x)|^p\,dx.
    \end{equation}
    Let us set $P^{\s,1}_{\alpha}:=\alpha+\s P_1$ and $P^{\s,2}_{\alpha}:=\alpha+\s P_2$ where 
    \begin{equation*}
        P_1=\{(x_1,x_2) \in \mathbb{R}^{2}: 0\leq x_{1}\leq 1, -x_{1}\leq x_{2}\leq 1-x_{1}\},
    \end{equation*}
    \begin{equation*}
        P_2=\{(x_1,x_2) \in \mathbb{R}^{2}: 0\leq x_{2}\leq 1, -x_{2}\leq x_{1}\leq 1-x_{2}\}.
    \end{equation*}
    One can easily show that
    \begin{equation*}
\partial_{e_{k}}\hat{u}(x)=\frac{u(\alpha+\s e_k)-u(\alpha)}{\s}, \qquad \text{ for every } \alpha\in \s\ZZ^2, k\in\{1,2\}, \text{ and }  x\in P^{\s,k}_{\alpha},
    \end{equation*}
and since $|P^{\s,1}_{\alpha}|=|P^{\s,2}_{\alpha}|=\s^{2}$, we get the following estimate
    \begin{align}\notag
\int_{A}|\nabla \hat{u}|^{p}& \leq C \int_{A}|\partial_{e_{1}}\hat{u}|^{p}+|\partial_{e_{2}}\hat{u}|^{p}\,dx\\ \notag & = C \sum_{\alpha \in A\cap \s\mathbb{Z}^{2}}\Big(\int_{P^{\s,1}_{\alpha}\cap A}|\partial_{e_{1}}\hat{u}|^{p}\,dx+\int_{P^{\s,2}_{\alpha}\cap A}|\partial_{e_{2}}\hat{u}|^{p}\,dx\Big)\\ 
&= C\sum_{k=1}^2\sum_{\alpha\in R_\s^{e_k}(A)}\s^{2}\Bigl|\frac{u(\alpha+\s e_k)-u(\alpha)}{\s}\Bigr|^{p}.
\label{estimate_gradient_term}
    \end{align}

On the other hand, we have that 
\begin{equation}\label{equiv_norm}
     \sum_{\alpha \in A \cap \s\mathbb{Z}^2}\s^2|u(\alpha)|^p=\sum_{\alpha \in A \cap \s\mathbb{Z}^2}\s^2|\hat{u}(\alpha)|^p\leq C\int_A |\hat{u}(x)|^p\,dx
\end{equation}
for a positive constant $C$ independent of $\s$ and $u$. Hence, combining \eqref{strd_pcre}, \eqref{estimate_gradient_term} and \eqref{equiv_norm} we obtain 
    \begin{equation*}
        \sum_{\alpha \in A \cap \s\mathbb{Z}^2}\s^2|u(\alpha)|^p\leq  C\sum_{k=1}^2\sum_{\alpha\in R_\s^{e_k}(A)}\s^{2}\Bigl|\frac{u(\alpha+\s e_k)-u(\alpha)}{\s}\Bigr|^{p},
    \end{equation*}
    which is the thesis.
\end{proof}

{\textbf{Acknowledgements.}} 
    The authors thank Roberto Alicandro, Andrea Braides, Marco Cicalese, and Chiara Leone for useful comments. The authors are member of Gruppo Nazionale per l'Analisi Matematica, la Probabilità e le loro Applicazioni (GNAMPA) of Istituto Nazionale di Alta Matematica (INdAM). G. C. B. acknowledges the support of the INdAM - GNAMPA 2026 Project ``Analisi variazionale per operatori locali e nonlocali possibilmente singolari o degeneri" (CUP: E53C25002010001). G. F. acknowledges the support of the INdAM - GNAMPA 2026 Project ``Analisi e Gamma-convergenza per alcuni funzionali non locali” (CUP: E53C25002010001).
    
\smallskip

{\textbf{Statements and Declarations.} The authors declare to have no financial or non-financial interests related to the work submitted for publication.}

\smallskip

{\textbf{Data Availability statement.} Data sharing not applicable to this article since no datasets were generated or analyzed during the current study.}

\bibliographystyle{plain} 
\bibliography{refs} 

\end{document}